\documentclass{amsart}

\usepackage[utf8]{inputenc}
\usepackage[english]{babel}
\usepackage{graphicx}
\usepackage{amssymb}
\usepackage{amsmath}
\usepackage{tikz-cd}
\usepackage{adjustbox}
\usepackage{mathabx}
\usepackage{hyperref}
\usepackage{framed}
\usepackage{mathtools}
\usepackage{amsaddr}
\usepackage[backend=bibtex,
style=alphabetic,
bibencoding=ascii
]{biblatex}
\usepackage{rotating}

\newtheorem{theorem}{Theorem}[section]
\newtheorem{lemma}[theorem]{Lemma}
\newtheorem{proposition}[theorem]{Proposition}
\newtheorem{corollary}[theorem]{Corollary}

\theoremstyle{definition}
\newtheorem{definition}[theorem]{Definition}

\newtheorem{keyexample}[theorem]{Key Example}
\newtheorem{claim}[theorem]{Claim}

\newtheorem{setup}[theorem]{Setup}

\theoremstyle{remark}
\newtheorem{remark}[theorem]{Remark}

\numberwithin{equation}{section}

\newcommand{\ov}{\overline}

\newcommand{\Int}{\mathrm{Int}}

\newcommand{\der}{\mathrm{der}}
\newcommand{\mc}{\mathcal}

\newcommand{\Q}{\mathbb{Q}}
\newcommand{\mb}{\mathbf}
\newcommand{\mf}{\mathfrak}

\newcommand{\inv}{\mathrm{inv}}
\newcommand{\obs}{\mathrm{obs}}
\newcommand{\A}{\mathbb{A}}
\newcommand{\Sc}{\mathrm{sc}}
\newcommand{\ad}{\mathrm{ad}}
\newcommand{\iso}{\mathrm{iso}}
\newcommand{\bas}{\mathrm{bas}}
\newcommand{\alg}{\mathrm{alg}}
\newcommand{\bb}{\mathbb}
\newcommand{\C}{\mathbb{C}}
\newcommand{\R}{\mathbb{R}}
\newcommand{\Gal}{\mathrm{Gal}}
\newcommand{\Z}{\mathbb{Z}}

\newcommand{\im}{\mathrm{im}}

\newcommand{\Gm}{{\mathbb{G}_{\mathrm{m}}}}
\newcommand{\Hom}{\mathrm{Hom}}

\newcommand{\sr}{\mathrm{sr}}

\newcommand{\Res}{\mathrm{Res}}
\newcommand{\sep}{\mathrm{sep}}
\newcommand{\reg}{\mathrm{reg}}
\newcommand{\semis}{\mathrm{ss}}
\newcommand{\wh}{\widehat}
\newcommand{\Lie}{\mathrm{Lie}}
\newcommand{\Ad}{\mathrm{Ad}}

\title[Global adelic $\mb{B}(G)$ and non-regular transfer factors]{Global $\mb{B}(G)$ with adelic coefficients and transfer factors at non-regular elements}
\author{Alexander Bertoloni Meli}
\address{University of Michigan, 530 Church St, Ann Arbor, MI 48109, United States\qquad email: \href{mailto:abertolo@umich.edu}{abertolo@umich.edu}
}
\begin{document}
\sloppy
\begin{abstract}
    The goal of this paper is extend Kottwitz's theory of $B(G)$ for global fields. In  particular, we show how to extend the definition of ``$B(G)$ with adelic coefficients'' from tori to all connected reductive groups. As an application, we give an explicit construction of certain transfer factors for non-regular semisimple elements of non-quasisplit groups. This generalizes some results of Kaletha and Taibi. These formulas are used in the stabilization of the cohomology of Shimura  and Igusa varieties.
\end{abstract}

\maketitle

\tableofcontents
\section{Introduction}

Let $F$ be a $p$-adic field and $\breve{F}$ be the completion of the maximal unramified extension of $F$ with $\sigma$ the Frobenius endomorphism. Then for a connected reductive group $G$ defined over $F$, the Kottwitz set $\mb{B}(G)$ is given as the set of $\sigma$-twisted conjugacy classes of $G(\breve{F})$. Namely we say that $g, g' \in G(\breve{F})$ are in the same $\sigma$-conjugacy class if $g' = h^{-1} g \sigma(h)$ for some $h \in G(\breve{F})$. The set $\mb{B}(G)$ appears throughout the theories of $p$-adic geometry and $p$-adic representation theory.

In \cite{Kot9}, Kottwitz constructed a set $\mb{B}(F,G)$ for every local and global field $F$ and linear algebraic group $G$ defined over $F$. The sets $\mb{B}(F,G)$ are defined to be certain cohomology sets and we have a natural bijection $\mb{B}(G) \cong \mb{B}(F,G)$ in the $p$-adic case. More precisely, the sets $\mb{B}(F,G)$ are defined in terms of the cohomology of \emph{Galois gerbs}. In each case there is a certain $F$-protorus $\bb{D}_F$ and for each finite Galois extension $K/F$, Kottwitz defines an extension of groups:
\begin{equation*}
    1 \to \bb{D}_F(K) \to \mc{E}(K/F) \to \Gal(K/F) \to 1.
\end{equation*}
He then defines a certain cohomology set $H^1_{\alg}( \mc{E}(K/F), G(K))$ and $\mb{B}(F, G)$ is defined as the limit:
\begin{equation*}
    \varinjlim_K H^1_{\alg}( \mc{E}(K/F), G(K)),
\end{equation*}
where the maps between these sets are given by certain inflation maps. We frequently use the notation $H^1_{\alg}(\mc{E}, G(\ov{F}))$ for the above limit when we want to stress the cohomological nature of the construction.

In the global case, the situation is more complicated than we have indicated because Kottwitz defines three pro-tori $D_1, D_2, D_3 (= \bb{D}_F)$ and additional gerbs
\begin{equation*}
    1 \to D_2(\A_K) \to \mc{E}_2(K/F) \to \Gal(K/F) \to 1,
\end{equation*}
and 
\begin{equation*}
    1 \to D_1(\A_K)/D_1(K) \to \mc{E}_1(K/F) \to \Gal(K/F) \to 1.
\end{equation*}
In the case where $G=T$ is a torus, these other gerbs give rise to cohomology sets $H^1_{\alg}(\mc{E}_2(K/F), T(\A_K))$ and $H^1_{\alg}(\mc{E}_1(K/F), T(\A_K)/T(K))$ respectively. By taking injective limits, we get sets $\mb{B}_2(F, T)$ and $\mb{B}_1(F, T)$ which can be thought of as ``$\mb{B}(F, T)$ with $\A_K$- and $\A_K/K$-coefficients'' respectively. Unfortunately, this construction does not extend to general $G$ since, for instance, when $i=2$ it requires an action of $G(\A_K)$ on $\Hom_K(D_2, G)$ that restricts to the action of $G(K)$. In the torus case this action can be defined to be trivial, but in general there does not appear to be a natural way to define such an action. 

Our first main result is to extend the theory of the cohomology of $\mc{E}_2(K/F)$ and $\mc{E}_1(K/F)$ beyond the case of tori to allow $G$ to be any connected reductive group. We do this by generalizing the cohomology set ``$H^1_Y(E, M)$'' that Kottwitz constructs in \cite[\S 12]{Kot9}. In the specific example mentioned in the previous paragraph, our construction allows us to consider the pair of sets $\Hom_K(D_2, G) \hookrightarrow \prod\limits_v \Hom_{K_v}(D_2, G)$ and we only require that $G(\A_K)$ acts on the larger space.

This allows us to define sets $H^1_{\alg}(\mc{E}_2(K/F), G(\A_K))$ and $H^1_{\bas}(\mc{E}_1(K/F), G(\A_K)/Z_G(K))$. We develop the theory for these sets in analogy with \cite{Kot9} and in particular define a ``total localization map'' relating $H^1_{\alg}(\mc{E}_2(K/F), G(K))$ to the local gerbs at each place $u$ of $F$ and $v$ of $K$ such that $v \mid u$. We get
\begin{equation*}
    l^F: H^1_{\alg}(\mc{E}_2(K/F), G(\A_K)) \to \bigoplus\limits_u H^1_{\alg}(\mc{E}_{\iso}(K_v/F_u), G(K_v)),
\end{equation*}
where $\mc{E}_{\iso}$ is the notation we use for Kottwitz's local gerb.

We then prove this map fits into a fundamental commutative diagram connecting the cohomology of these different gerbs and certain character groups:
\begin{equation}{\label{introfundiag}}
\begin{tikzcd}
\bigoplus\limits_{u \in V_F} H^1_{\bas}(\mc{E}_{\iso}(K_v/F_u), G(K_v)) \arrow[d] & H^1_{\bas}(\mc{E}_2(K/F), G(\A_K))  \arrow[l, "l^F", swap] \arrow[r] \arrow[d] & H^1_{\bas}(\mc{E}_1(K/F), G(\A_K)/Z_G(K)) \arrow[d] \\
   \bigoplus\limits_{u \in V_F} X^*(Z(\widehat{G}))_{\Gal(K_v/F_u)} & \left[\bigoplus\limits_{v \in V_K} X^*(Z(\widehat{G}))\right]_{\Gal(K/F)} \arrow[r, "\Sigma"] \arrow[l, "\sim",swap]& X^*(Z(\widehat{G}))_{\Gal(K/F)}. 
\end{tikzcd}    
\end{equation}
This diagram generalizes an analogous diagram for tori appearing in Kottwitz's paper (\cite[\S 1.5]{Kot9}).

The three global gerbs correspond to cohomology classes that were first studied systematically by Tate (\cite{Tatecohomology}) and appear to be very important objects. The group $D_1$ equals $\Gm$ and hence the Galois gerb $\mc{E}_1(K/F)$ corresponds to the canonical class in $H^2(\Gal(K/F), \Gm(K))$ of global class field theory. On the other hand, the $\mc{E}_2(K/F)$ gerb is constructed from the local canonical classes at each place of $K$. Scholze (\cite[Conjecture 9.5]{scholzeicm}) has conjectured the existence of a  cohomology theory for varieties over $\ov{\bb{F}}_p$ valued in the representation category of the $\mc{E}_3$ gerb that would specialize to most known cohomology theories. Scholze notes that an important first step in the direction of this conjecture is to give a linear algebraic description of this representation category in analogy with the theory of isocrystals for $p$-adic fields. The $n$-dimensional representations of $\mc{E}_3$ are classified by the set $\mb{B}(F,GL_n)$. The cohomology of the $\mc{E}_3$ gerb is closely related to the cohomology of the other gerbs as Diagram \eqref{introfundiag} indicates. In fact, the existence of the $\mc{E}_3$-gerbe itself is only deduced as a consequence of the construction of the $\mc{E}_2$ and $\mc{E}_1$ gerbs. Hence, one motivation for developing the results of this paper is to define the sets $\mb{B}_2(F, GL_n)$ and $\mb{B}_1(F, GL_n)$ which should shed light on $\mb{B}(F, GL_n)$.

Another application of the global theory of $\mb{B}(F, G)$ for a number field $F$ is in the normalization of the Langlands correspondence. In particular, for $G$ satisfying the Hasse principle, this set is used to state the ``global multiplicity formula'' describing the decomposition of the discrete part of $L^2_{\chi}(G(F) \setminus G(\A_F))$. This is accomplished by Kaletha and Taibi in \cite{KalTai}. When the group $G$ is not quasi-split, the statement of this formula seems to require use of either the $\mb{B}(F,G)$-normalization or the more general but more complicated rigid normalization. This problem is discussed in \cite{Kal1} and \cite{KalTai} and the $\mb{B}(F,G)$-normalization is used in \cite{KMSW} to prove the global multiplicity formula for unitary groups in the non quasi-split case.  

In trace formula arguments where the global multiplicity formula is used, one needs a normalization of local and global transfer factors between $G$ and an endoscopic group $H$ that is compatible with the normalization of the global multiplicity formula. Such a normalization is constructed in \cite{KalTai} for strongly regular semi-simple $\gamma \in G(F)$ using the theory of $\mb{B}(F,G)$. We recall that $\gamma$ is strongly regular if its centralizer in $G$ is a torus. 

Crucially, the construction of Kaletha and Taibi requires that $\gamma$ is strongly regular because they need to use the ``adelic form'' of $\mb{B}(F, G)$ which was only known for $G$ a torus.  As an application of the first part of the paper, we prove:
\begin{theorem}[Imprecise version of Theorem \ref{localtrans}]
Suppose $F$ is a number field and $G$ is a connected reductive group over $F$ that satisfies the Hasse principle and has simply connected derived subgroup. Then the theory of $H^1_{\alg}(\mc{E}_2, G(\ov{\A_F}))$ gives an explicit normalization of the transfer factors between $G$ and any endoscopic group $H$ for semisimple $\gamma \in G(F)$. This normalization is compatible with the isocrystal normalization of the global multiplicity formula as in \cite{KalTai}.
\end{theorem}
The normalization of transfer factors for non-strongly regular elements is needed in the analysis of the trace formula for the cohomology of Shimura varieties. In particular, the results of this paper are used in work of the author that uses the cohomology of  Shimura varieties to deduce new formulas for the cohomology of Rapoport-Zink spaces (\cite{BM2}) and related work of the author and K.H. Nguyen proving the Kottwitz conjecture on the cohomology of Rapoport-Zink spaces for odd unramified unitary similitude groups (\cite{BMN1}).

Finally, we make some remarks about the organization of the paper. In \S\ref{globalBGappendix} we develop the abstract theory of the cohomology of $\mc{E}_2(K/F)$ and $\mc{E}_1(K/F)$, in particular constructing the maps and proving the commutativity of Diagram \eqref{introfundiag}. In \S\ref{s: normalizingtransferfactors} we discuss the $\mb{B}(F,G)$-normalization of transfer factors for $(G,H)$-regular elements. We remark that to do this, we do not need the full strength of the theory developed in \S\ref{globalBGappendix} because we need only work with the basic sets $\mb{B}_i(F,G)_{\bas}$. However, the theory from \S\ref{globalBGappendix} is used in Proposition \ref{obseqz}, which is then used in Corollary \ref{cor: globaltranscor}. We also use \S\ref{globalBGappendix} to prove (before Corollary \ref{cor: globaltranscor}) that for a fixed pair $(\gamma^{\mb{H}}, \gamma) \in \mb{H}(F)_{(\mb{G}, \mb{H})-\reg} \times \mb{G}(F)$ and $F$ a number field, the local transfer factors vanish at almost every place.

\subsection*{Acknowledgements}

I would like to thank Tasho Kaletha for many helpful discussions and suggestions. I am also grateful for conversations with Sug Woo Shin and Alex Youcis which influenced my thoughts on the contents of this paper. This research was partially supported by NSF RTG grant DMS-1840234.

\section{Global \texorpdfstring{$\mb{B}(G)$}{B(G)} with Adelic Coefficients}{\label{globalBGappendix}}

In this section we develop a formalism that allows us to define a global cohomology set $H^1_{\alg}(\mc{E}_2(K/F), G(\A_K))$ for a Galois extension $K/F$, the Galois gerbe $\mc{E}_2(K/F)$ defined by Kottwitz in \cite{Kot9}, and a  general reductive group $G$ defined over $F$. This generalizes the construction given by Kottwitz in \cite{Kot9} of the set $H^1_{\alg}(\mc{E}_2(K/F), T(\A_K))$ for $T$ an algebraic torus split by $K$. We then develop the theory of the set $H^1_{\alg}(\mc{E}_2(K/F), G(\A_K))$ in analogy with \cite{Kot9}.

In Kottwitz's article, these groups are defined in the case where $T$ is a torus using the $H^1_Y(E,M)$ construction of his \S$3$ and \S$12$. This construction is not sufficient for our purposes because Kottwitz requires the group $M$ to act on $Y$. In our setting, we would therefore need $G(\A_K)$ to act on the set of algebraic maps $\Hom_K(D , G)$ (where $D$ is a pro-torus) which it does not. Our solution is to develop a theory in the spirit of \cite[\S 12]{Kot9} but in a setting that allows for adelic coefficients.

Our setup is as follows.
\begin{setup}{\label{setup: cohdat}}
 We suppose we have the following objects:
\begin{itemize}
    \item an abstract group $G$,
    \item an abelian group $A$ equipped with a $G$-action (i.e. a $G$-module),
    \item an extension
    \begin{equation*}
        1 \to A \to E \to G \to 1,
    \end{equation*} such that the conjugation action of $G$ on $A$ coincides with the action in the previous item,
    \item a possibly non-abelian group $M$ equipped with an action of $G$ by automorphisms of $M$,
    \item an $M \rtimes G$-set $\ov{Y}$,
    \item a map $\xi: \ov{Y} \to \Hom(A, M)$ of $M \rtimes G$ sets where $M \rtimes G$ acts on $\Hom(A, M)$ by $\phi \mapsto \Int(m) \circ (g \circ \phi \circ g^{-1})$,
    \item A subset $Y \subset \ov{Y}$ which need not be an $M \rtimes G$-subset.
\end{itemize}
Note that $E$ acts on $M$ through $G$. We further require
\begin{itemize}
    \item $\xi(y)(A) \subset M_y$ for all $y \in Y$ (where $M_y$ is the stabilizer of $y$ in $M$).
\end{itemize}

For a fixed extension $1 \to A \to E \to G \to 1$, we call the tuple $(M, \ov{Y}, \xi, Y)$ a \emph{cohomology datum} for $E$.
\end{setup}

\begin{definition}{\label{cocycledef}}
Given an extension $E$ as in Setup \ref{setup: cohdat} and a cohomology datum $(M, \ov{Y}, \xi, Y)$ for $E$, we define $Z^1_Y(E,M)$ to be the set of pairs $(\nu, x)$ such that $\nu \in Y$ and $x \in Z^1(E, M)$ is an abstract cocycle satisfying the following conditions.
\begin{enumerate}
    \item The restriction of $x$ to $A$ gives $\xi(\nu)$.
    \item $x_w  \cdot \sigma(\nu) = \nu$ for each $w \in E$ where $\sigma$ is the projection of $w$ to $G$.
\end{enumerate}
Note that when $\xi$ is injective, the second condition above is implied by the cocycle relation.

We define $H^1_Y(E,M)$ to be the quotient of $Z^1_Y(E,M)$ by the equivalence relation that $(\nu, x) \sim (\nu', x')$ if there exists $m \in M$ such that $\nu = m \cdot \nu'$ and for all $w \in E$, we have $m^{-1} x_w w(m)=x'_w$.
\end{definition}

Suppose that $H^1(G,A)=0$. Then all automorphisms of the extension
\begin{equation*}
        1 \to A \to E \to G \to 1,
\end{equation*}
are given by conjugation by some element $a \in A$. Such an automorphism induces an automorphism of $Z^1_Y(E,M)$ given by $(\nu, x) \mapsto (\nu, x \circ \Int(a))$. If we let $m=\xi(\nu)(a)=x_a$ then we see that by assumption $m$ acts trivially on $\nu$ and hence that $(\nu,x)$ and $(\nu, x \circ \Int(a) )=(m^{-1} \cdot \nu, w \mapsto mx_w w(m^{-1}))$ agree inside $H^1_Y(E,M)$. In particular, we have proven that if $H^1(G,A)=0$ then the set $H^1_Y(E,M)$ depends up to canonical isomorphism only on $M$ and the class $\alpha \in H^2(G,A)$ giving the extension $E$.

\begin{keyexample}{\label{keyexample}}
When $Y=\ov{Y}$, this construction specializes to the $H^1_Y(E,M)$ construction given in \cite[\S 12]{Kot9}. In particular we review the following key examples.
\begin{itemize}
    \item Let $K/F$ be a finite extension of local fields and consider $\Gm(K)$ with the natural $\Gal(K/F)$ action. Then the fundamental class $\alpha \in H^2(\Gal(K/F), \Gm(K))$ corresponds to an isomorphism class of extensions. We choose a representative which  we denote by
\begin{equation*}
    1 \to \Gm(K) \to \mc{E}_{\iso}(K/F) \to \Gal(K/F) \to 1.
\end{equation*}
Then for any connected reductive group $G$ over $F$, we give $G(K)$ the natural $\Gal(K/F)$-action and define $\ov{Y_{\iso}}=Y_{\iso}=\Hom_K(\Gm, G)$. Then we have a natural map 
\begin{equation*}
\xi: \Hom_K(\Gm, G) \to \Hom(\Gm(K) , G(K))
\end{equation*}
and we can define the set $H^1_{\alg}(\mc{E}_{\iso}(K/F) , G(K))$ to be equal to $H^1_{Y_{\iso}}(\mc{E}_{\iso}(K/F), G(K))$. 
\item Now fix $K/F$ a finite Galois extension of global fields and $D_1, D_2, D_3$ the $F$ pro-tori with character groups $X_1=\Z, X_2=\Z[V_K], X_3=\Z[V_K]_0$ where $V_K$ is the set of places of $K$, where $\Z[V_K]$ is the free abelian group generated by $K$, and $\Z[V_K]_0$ is the subgroup of elements whose coefficients sum to $0$. Let $A_1=\A^{\times}_K/K^{\times}, A_2=\A^{\times}_K, A_3=K^{\times}$.  In \cite[\S $6.2$]{Kot9}, Kottwitz describes the construction due to Tate for $i=1,2,3$ of canonical classes $\alpha_i \in H^2( \Gal(K/F), \Hom(X_i, A_i))$ corresponding to extensions
\begin{equation*}
    1 \to \Hom(X_i, A_i) \to \mc{E}_i(K/F) \to \Gal(K/F) \to 1. 
\end{equation*}

Now fix an $F$-torus $T$ that is split by $K$ and define $Y_i= \Hom_K(D_i, T)$. Then one can form the groups
\begin{equation*}
    H^1_{Y_1}(\mc{E}_1, T(\A_K)/T(K)), H^1_{Y_2}(\mc{E}_2, T(\A_K)), H^1_{Y_3}(\mc{E}_3, T(K)),
\end{equation*}
using the above definition. These are the sets
\begin{equation*}
    H^1_{\alg}(\mc{E}_1(K/F), T(\A_K)/T(K)), H^1_{\alg}(\mc{E}_2(K/F), T(\A_K)), H^1_{\alg}(\mc{E}_3(K/F), T(K)),
\end{equation*}
as given in \cite{Kot9}. 
\item Using the notation in the previous item, we define $Y_3=\Hom(D_3, G)$ and then can define the set $H^1_{\alg}(\mc{E}_3(K/F), G(K))$ to be equal to $H^1_{Y_3}(\mc{E}_3(K/F), G(K))$.
\end{itemize}
\end{keyexample}

\begin{definition}{\label{adeliccoeffs}}
Our definitions are slightly more general than those of Kottwitz because we allow $Y \subset \ov{Y}$ to be a proper subset and allow $Y$ to not be an $M \rtimes G$ set. This means that we can define $H^1_{\alg}(\mc{E}_2(K/F), G(\A_K))$ for a general reductive group $G$ defined over a global field $F$ and $K$ a finite Galois extension. 

Let $V_K$ be the set of places of $K$ as before. Then we define $Y_2 := \Hom_K(D_2,G)$ and note there is a natural inclusion
\begin{equation*}
    Y_2 \hookrightarrow \prod_{v \in V_K} \Hom_{K_v}(D_2, G).
\end{equation*}
We then define  $\ov{Y_2}$ to be the $G(\A_K) \rtimes \Gal(K/F)$-orbit of $Y_2$ inside $\prod_{v \in S_K} \Hom_{K_v}(D_2, G)$. Then $\ov{Y_2}$ is naturally a $G(\A_K) \rtimes \Gal(K/F)$- -set and we have a natural map $\xi_2 : \ov{Y_2} \to  \Hom(D_2(\A_K), G(\A_K))$. Finally, we define 
\begin{equation*}
      H^1_{\alg}(\mc{E}_2(K/F), G(\A_K)) := H^1_{Y_2}(\mc{E}_2(K/F), G(\A_K)).
\end{equation*}
Note that $Y_2$ does not have a natural $G(\A_K)$-action so we do indeed need the more general formalism.
\end{definition}

We now study, as in \cite[\S 12]{Kot9}, the naturality of our construction. 
\subsection{Naturality with respect to \texorpdfstring{$(M, \ov{Y}, \xi, Y)$}{M,Y, Y-bar, X}}
The most basic situation to consider is for $E$ fixed. Then we suppose we have two cohomology data $(M, \ov{Y}, \xi, Y)$ and $(M', \ov{Y}', \xi', Y')$ such that we have a $G$-map $f: M \to M'$ and a $M \rtimes G$-map $g: \ov{Y} \to \ov{Y'}$ (where $M$ acts on $\ov{Y'}$ through $f$) such that $g(Y) \subset Y'$ and such that the diagram
\begin{equation*}
\begin{tikzcd}
\ov{Y}  \arrow[d, "g", swap]  \arrow[r, "\xi"]& \Hom(A,M) \arrow[d, "f \circ"] \\
\ov{Y'} \arrow[r, "\xi'"] & \Hom(A, M')
\end{tikzcd}    
\end{equation*}
commutes. 

We have a map $Z^1_Y(E, M) \to Z^1_{Y'}(E, M')$ given by $(\nu, x) \mapsto (g(\nu), f \circ x)$. This induces a map 
\begin{equation}
   H^1_Y(E, M) \to H^1_{Y'}(E, M'), 
\end{equation}
since if $(\nu, x) \sim (\nu', x')$ via $m$, then 
\begin{equation*}
f(m) \cdot g(\nu)=g(m \cdot \nu))=g(\nu'),
\end{equation*}
and
\begin{equation*}
    f(m)^{-1} f(x_w) w(f(m)) = f( m^{-1} x_w w(m) )=f(x'_w).
\end{equation*}

\subsection{Changing \texorpdfstring{$G$}{G}}
Suppose we have a map $\rho: H \to G$ and an extension
\begin{equation*}
    1 \to A \to E \to G \to 1.
\end{equation*}
Let 
\begin{equation*}
    1 \to A \to E_H \to H \to 1
\end{equation*}
be the extension defined so that $E_H=E \times_G H$ and consider the diagram of extensions given by:
    \begin{equation*}
        \begin{tikzcd}
        1 \arrow[r] & A \arrow[r] \arrow[d, equal] & E_H \arrow[r] \arrow[d, "\tilde{\rho}"] & H \arrow[r] \arrow[d, "\rho"] & 1 \\
        1 \arrow[r] & A \arrow[r] & E \arrow[r] & G \arrow[r] & 1
        \end{tikzcd}
    \end{equation*}

Then define a map $Z^1_Y(E, M) \to Z^1_Y(E_H, M)$ so that $(\nu, x) \mapsto (\nu, x_H)$ where $x_H$ is the pullback of $x$ to $E_H$. This clearly induces a map 
\begin{equation}{\label{restriction}}
H^1_Y(E, M) \to H^1_Y(E_H, M).
\end{equation}

\subsection{The map \texorpdfstring{$\Phi(f, g, \tilde{h})$}{Phi(f,g,h)}}{\label{sec: Phi}}
Suppose we have extensions
\begin{equation*}
    1 \to A \to E \to G \to 1
\end{equation*}
and
\begin{equation*}
    1 \to A' \to E' \to G \to 1
\end{equation*}
and cohomology data $(M, \ov{Y}, \xi, Y)$ and $(M', \ov{Y'}, \xi', Y')$ giving us sets $H^1_Y(E, M)$ and $H^1_{Y'}(E', M')$. Suppose further that we have the following maps:
\begin{itemize}
    \item A $G$-homomorphism $f: M \to M'$,
    \item An $M \rtimes G$-map $g: \ov{Y} \to \ov{Y'}$ such that $g(Y) \subset Y'$,
    \item A homomorphism $\tilde{h}: E \to E'$ of extensions:
    \begin{equation*}
        \begin{tikzcd}
        1 \arrow[r] & A \arrow[r] \arrow[d, "h"] & E \arrow[r] \arrow[d, "\tilde{h}"] & G \arrow[r] \arrow[d, equal] & 1 \\
        1 \arrow[r] & A' \arrow[r] & E' \arrow[r] & G \arrow[r] & 1
        \end{tikzcd}
    \end{equation*}
\end{itemize}
We further require that the following diagram commutes:
\begin{equation*}{\label{Phidiagram}}
   \begin{tikzcd}
   \ov{Y} \arrow[r, "\xi"] \arrow[d, equal] & \Hom(A, M) \arrow[d, "f \circ"] \\
   \ov{Y} \arrow[d, "g", swap] & \Hom(A, M') \\
   \ov{Y'} \arrow[r, "\xi'", swap] & \Hom(A', M') \arrow[u, "\circ h", swap]
   \end{tikzcd}
\end{equation*}
We now define $\Phi(f, g, \tilde{h}): Z^1_Y(E,M) \to Z^1_{Y'}(E', M')$ so that $(\nu, x) \mapsto (g(\nu), x')$ where $x'$ is the unique cocycle so that the restriction of $x'$ to $A'$ is equal to $\xi'(g(\nu))$ and the pullback of $x'$ to $E$ via $\tilde{h}$ equals $f(x)$. It is a tedious but straightforward check that such a cocycle exists and is unique. 

We check that $(g(\nu), x')$ satisfies $\Int(x'_{w'}) \circ \sigma(\nu) = \nu$ for $w' \in E'$ projecting to $
\sigma \in G$. Write $w'=a'\tilde{h}(e)$. Then we have
\begin{align*}
    \Int(x'_{w'}) \circ \sigma(g(\nu)) &= \Int(x'_{a'}) \circ [f(x_e) \sigma(g(\nu)) f(x_e)^{-1}]\\
    &=\Int(x'_{a'}) \circ [g( x_e\sigma(\nu)x_e^{-1})]\\
    &=\Int(x'_{a'}) \circ g(\nu) = g(\nu).
\end{align*}
The last equality follows from the condition that $\xi'(g(\nu))(A') \subset M'_{g(\nu)}$. 

\begin{lemma}{\label{phi}}
The map $\Phi(f,g,\tilde{h})$ on cocycles induces a map
\begin{equation*}
\Phi(f, g, \tilde{h}): H^1_Y(E, M) \to H^1_{Y'}(E', M').
\end{equation*}
\end{lemma}
\begin{proof}
It is an easy check that if $(\nu_1 , x_1) \sim (\nu_2, x_2)$ via $m \in M$, then $(g(\nu_1), x'_1) \sim (g(\nu_2), x'_2)$ via $f(m)$. 
\end{proof}

From the definitions, it is clear that if we have triples $(f_1, g_1, \tilde{h_1})$ between $(M, \ov{Y}, \xi, Y)$ and $(M', \ov{Y'}, \xi', Y')$ and $(f_2, g_2, \tilde{h_2})$ between $(M', \ov{Y'}, \xi', Y')$ and $(M'', \ov{Y''}, \xi'', Y'')$ then we can form a triple $(f_2 \circ f_1, g_2 \circ g_1, \tilde{h_2} \circ \tilde{h_1})$ satisfying the necessary conditions.

\begin{lemma}{\label{Philem}}
We have $\Phi(f_2 \circ f_1, g_2 \circ g_1, \tilde{h_2} \circ \tilde{h_1})= \Phi(f_2, g_2, \tilde{h_2}) \circ \Phi(f_1, g_1, \tilde{h_1})$.
\end{lemma}
\begin{proof}
Clear from uniqueness of $x'$ in the definition of $\Phi(f, g, \tilde{h})$.
\end{proof}
Observe that the map $h: A \to A'$ induces a map $H^2(G, A) \to H^2(G, A')$ and that if $\alpha \in H^2(G,A)$ and $\alpha' \in H^2(G,A')$ are the classes of the extensions $E$ and $E'$ respectively then $h(\alpha)=\alpha'$. This follows because a $2$-cocycle giving $\alpha$ can be constructed from any section $s: G \to E$ and $h \circ s: G \to E'$ is then a section of $E'$. Conversely, a map $h: A \to A'$ and an extension
\begin{equation*}
        1 \to A \to E \to G \to 1,
    \end{equation*}
with class $\alpha \in H^2(G,A)$ yields a class $h(\alpha) \in H^2(G,A)$. Via the natural bijection between $H^2(G,A')$ and extensions of $G$ by $A'$, one constructs an extension $E'$ and a diagram:
    \begin{equation*}
        \begin{tikzcd}
        1 \arrow[r] & A \arrow[r] \arrow[d, "h"] & E \arrow[r] \arrow[d, "\tilde{h}"] & G \arrow[r] \arrow[d, equal] & 1 \\
        1 \arrow[r] & A' \arrow[r] & E' \arrow[r] & G \arrow[r] & 1
        \end{tikzcd}
    \end{equation*}
where $\tilde{h}$ is canonical up to an automorphism of the extension $E'$. In particular, if $H^1(G,A')=0$ then $\tilde{h}$ is canonical up to $A'$-conjugacy. It is easy to check that $\Phi(f,g, \tilde{h})=\Phi(f,g, \tilde{h}')$ for $\tilde{h}$ and $\tilde{h}'$ in the same $A'$-conjugacy class. Hence we have proven that if $H^1(G,A')=0$ then $\Phi(f,g,\tilde{h})$ only depends on $f,g,h$ and the extensions $E, E'$. If we furthermore have $H^1(G,A)=0$, then by the remark before Key Example \ref{keyexample} the map $\Phi(f, g, \tilde{h})$ only depends on $f,g,h$ and the $\alpha \in H^2(G,A)$ and $\alpha' \in H^2(G,A')$. In particular, whenever we have maps $f,g,h$ such that Diagram \eqref{Phidiagram} commutes and $h(\alpha)=\alpha'$ for any classes in the relative cohomology groups, there exists a canonical map $\Phi(f,g,h): H^1_Y(E,M) \to H^1_{Y'}(E',M')$ where $E$ and $E'$ are any extensions representing the relevant cohomology classes.
\subsection{\texorpdfstring{The map $\Psi(g, \tilde{h})$}{Psi(g,h)}}{\label{Psi}}
Suppose we have extensions 
\begin{equation*}
    1 \to A \to E \to G \to 1
\end{equation*}
and
\begin{equation*}
    1 \to A' \to E' \to G \to 1
\end{equation*}
with cohomology data  $(M, \ov{Y}, \xi, Y)$ and $(M', \ov{Y'}, \xi, Y')$ as before. Suppose further that $M=M'$ and that we have the following data:
\begin{itemize}
    \item an $M \rtimes G$-map $g: \ov{Y} \to \ov{Y'}$ with $g(Y) \subset Y'$,
    \item a homomorphism $\tilde{p}: E' \to E$ of extensions:
        \begin{equation*}
        \begin{tikzcd}
        1 \arrow[r] & A \arrow[r]  & E \arrow[r] & G \arrow[r] \arrow[d, equal] & 1 \\
        1 \arrow[r] & A' \arrow[u, "p"] \arrow[r] & E' \arrow[r] \arrow[u, "\tilde{p}"] & G \arrow[r] & 1
        \end{tikzcd}
    \end{equation*}
\end{itemize}
such that the following diagram commutes
\begin{equation*}
    \begin{tikzcd}
    Y \arrow[d, "g", swap] \arrow[r, "\xi"] & \Hom(A, M) \arrow[d, "\circ p"] \\
    Y' \arrow[r, "\xi'", swap] & \Hom(A', M)
    \end{tikzcd}
\end{equation*}

We define 
\begin{equation}
   \Psi(g, \tilde{p}) : H^1_Y(E, M) \to H^1_{Y'}(E', M) 
\end{equation}
as a map on cocycles by $(\nu, x) \mapsto (g(\nu), x')$ where $x'$ is the pullback of $x$ via $\tilde{p}$. If $(\nu_1, x_1) \sim (\nu_2, x_2)$ via $m$, then we also have $(g(\nu_1), x'_1) \sim (g(\nu_2), x'_2)$ via $m$.

It's clear that if $(g_1, \tilde{p_1})$ and $(g_2, \tilde{p_2})$ are tuples satisfying the requisite conditions then so are $(g_2 \circ g_1, \tilde{p_2} \circ \tilde{p_1})$. Moreover, we have
\begin{lemma}{\label{Psilem}}
$\Psi(g_2, \tilde{p_2}) \circ \Psi(g_1, \tilde{p_1}) = \Psi( g_2 \circ g_1, \tilde{p_1} \circ \tilde{p_2})$
\end{lemma}
\begin{proof}
Clear.
\end{proof}
It is easy to check that we have $\Psi(g, \tilde{p})=\Psi(g, \tilde{p}')$ if $\tilde{p}$ and $\tilde{p}'$ are in the same $A$-conjugacy class. Hence, as we noted for $\Phi$, when $H^1(G,A)=0$ the map $\Psi(g, \tilde{p})$ only depends on $g,p$ and when $H^1(G,A')=0$ as well, then whenever we have maps $g$ and $p$ and classes $\alpha \in H^2(G,A), \alpha' \in H^2(G,A')$ such that $p(\alpha')=
\alpha$, we have a canonical map up to canonical isomorphism $\Phi(g,p): H^1_Y(E,M) \to H^1_{Y'}(E', M)$ where $E$ and $E'$ are any representatives of the classes $\alpha, \alpha'$.
\subsection{Compatibility of \texorpdfstring{$\Phi$ and $\Psi$ and change of $G$}{Phi and Psi}}
We first check that $\Phi$ commutes with change of $G$. Suppose we have maps $\Phi(f,g,\tilde{h})$ as in \S\ref{sec: Phi} and $\rho: H \to G$. Then we have a commutative diagram of extensions:
\begin{equation*}
 \begin{tikzcd}
 E  \arrow[r, "\tilde{h}"] & E'  \\
 E_H \arrow[u, "\tilde{\rho}"] \arrow[r, "\tilde{h}", swap] & E'_H \arrow[u, "\tilde{\rho}", swap]
\end{tikzcd} 
\end{equation*}
where the map $E_H \to E'_H$ is the unique map so that the above diagram commutes. 
\begin{lemma}
In the above setup, we get a commutative diagram
\begin{equation*}
 \begin{tikzcd}
 H^1_Y(E,M)  \arrow[r, "{\Phi(f,g,\tilde{h})}"] \arrow[d, "\tilde{\rho}"] & H^1_{Y'}(E', M') \arrow[d, "\tilde{\rho}", swap]  \\
 H^1_Y(E_H,M)  \arrow[r, "{\Phi(f,g,\tilde{h})}", swap] & H^1_{Y'}(E'_H, M') 
\end{tikzcd} 
\end{equation*}
\end{lemma}
\begin{proof}
The left vertical map takes a a cocycle $(\nu, x)$ to one of the form $(\nu, x \circ \tilde{\rho})$ and the bottom map takes this to the unique cocycle $(g(\nu), (x \circ \tilde{\rho})')$ such that $(x \circ \tilde{\rho})'$ pulls back to $x \circ \tilde{\rho}$ via $\tilde{h}$. Hence it suffices to show that the cocycle $(g(\nu), x' \circ \tilde{\rho})$ also has this property. For $e \in E_H$, we have $(x' \circ \tilde{\rho})(\tilde{h}(e))= x'(\tilde{h}(\tilde{\rho}(e)))=(x \circ \tilde{\rho})(e)$ as desired.
\end{proof}

We now check the compatibility of $\Psi$ with change of $G$. We suppose we have maps $\Psi(g, \tilde{p})$ and $\rho: H \to G$ giving a diagram of extensions
\begin{equation*}
 \begin{tikzcd}
 E   & E' \arrow[l, "\tilde{p}", swap] \\
 E_H \arrow[u, "\tilde{\rho}"]  & E'_H \arrow[l, "\tilde{p}"]\arrow[u, "\tilde{\rho}", swap]
\end{tikzcd} 
\end{equation*}
such that $E'_H=E' \times_G H$ and the map $E'_H \to E_H$ is induced by the composition of $\tilde{p}$ and the projection given by $E'_H \to E' \to E$, as well as the projection $E'_H \to H$. 
\begin{lemma}
In the above setup, we get a commutative diagram
\begin{equation*}
 \begin{tikzcd}
 H^1_Y(E,M)  \arrow[r, "{\Psi(g,\tilde{p})}"] \arrow[d, "\tilde{\rho}"] & H^1_{Y'}(E', M') \arrow[d, "\tilde{\rho}", swap]  \\
 H^1_Y(E_H,M)  \arrow[r, "{\Psi(g,\tilde{p})}", swap] & H^1_{Y'}(E'_H, M') 
\end{tikzcd} 
\end{equation*}
\end{lemma}
\begin{proof}
This is immediate from the fact that all maps are defined via pullback and that the above diagram of extensions commutes.
\end{proof}

Finally, we check the following compatibility of $\Phi$ and $\Psi$. Suppose we have the following commutative diagram of extensions
\begin{equation*}
    \begin{tikzcd}
     E \arrow[d, "\tilde{h}", swap] & E_1 \arrow[l, "\tilde{p}", swap] \arrow[d, "\tilde{h_1}"] \\
     E' & E'_1 \arrow[l, "\tilde{p'}"]
    \end{tikzcd}
\end{equation*}
a $G$-homomorphism $f: M \to M'$, and a diagram
\begin{equation*}
\begin{tikzcd}
    \ov{Y} \arrow[d, "g", swap] \arrow[r, "g'"] & \ov{Y_1} \arrow[d, "g_1"]\\
    \ov{Y'} \arrow[r, "g''"] & \ov{Y'_1}
\end{tikzcd}
\end{equation*}
such that $g, g', g_1$ are maps of $M \rtimes G$-sets, $g'_1$ is a map of $M' \rtimes G$-sets, and $g(Y) \subset Y', g'(Y) \subset Y_1, g_1(Y_1) \subset Y'_1, g''(Y') \subset Y'_1$. We further assume that $(f, g, \tilde{h}), (f, g_1, \tilde{h_1})$ satisfy the requirements of the definition of $\Phi$ and $(g', \tilde{p}), (g'', \tilde{p'})$ satisfy the requirements as in the definition of $\Psi$.
\begin{lemma}{\label{PhiPsilem}}
Under the above assumptions, the following diagram commutes:
\begin{equation*}
\begin{tikzcd}
H^1_Y(E, M) \arrow[r, "{\Psi(g',\tilde{p})}"] \arrow[d, swap, "{\Phi(f,g,\tilde{h})}"] & H^1_{Y_1}(E_1, M) \arrow[d, "{\Phi(f,g_1,\tilde{h_1})}"] \\
H^1_{Y'}(E', M') \arrow[r, swap, "{\Psi(g'',\tilde{p'})}"] & H^1_{Y'_1}(E'_1, M')
\end{tikzcd}
\end{equation*}
\end{lemma}
\begin{proof}
This is straightforward but somewhat tedious to check.
\end{proof}

\subsection{Localization}{\label{BGloc}}
Fix a finite Galois extension $K/F$ of global fields and a connected reductive group $G$. Let $v$ be a place of $K$ over a place $u$ of $F$. We now study the localization of the set $H^1_{\alg}(\mc{E}_2(K/F), G(\A_K))$ introduced in Definition \ref{adeliccoeffs}. We let $E^v \subset K$ be the fixed field of the decomposition group at $v$ of $\Gal(K/F)$. Then $\Gal(K/E^v)=\Gal(K_v/F_u)$ and hence acts on $D_2(K_v)$ such that the natural projection $\pi: D_2(\A_K) \to D_2(K_v)$ is equivariant with respect to the $\Gal(K/E^v)$-action. Following Kottwitz, we define the gerb $\mc{E}^v_2(K/E^v)$ via pushout as $\mc{E}^v_2(K/E^v) := D_2(K_v) \rtimes \mc{E}_2(K/E^v) / N$ where $N=\{ (\pi(d)^{-1}, \iota(d)) : d \in D_2(\A_K)\}$, giving the following commutative diagram of extensions
\begin{equation}
\begin{tikzcd}
1 \arrow[r] & D_2(\A_K) \arrow[r] & \mc{E}_2(K/F) \arrow[r] & \Gal(K/F) \arrow[r] & 1\\
1 \arrow[r] & D_2(\A_K) \arrow[r, "\iota"] \arrow[d, "\pi"] \arrow[u, equal] & \mc{E}_2(K/E^v) \arrow[r] \arrow[d, "\tilde{\pi}"] \arrow[u, "j"] & \Gal(K/E^v) \arrow[r] \arrow[d, equal] \arrow[u, hook] & 1\\
1 \arrow[r] & D_2(K_v) \arrow[r] & \mc{E}^v_2(K/E^v) \arrow[r] & \Gal(K/E^v) \arrow[r] & 1.
\end{tikzcd}
\end{equation}
By restriction as in Equation \eqref{restriction} we get a natural map $H^1_{\alg}(\mc{E}_2(K/F) , G(\A_K)) \to H^1_{\alg}(\mc{E}_2(K/E^v), G(\A_K))$. Now define $\ov{Y_v}=Y_v= \Hom_{K_v}(D_2, G)$ and then define $H^1_{\alg}(\mc{E}^v_2(K/E^v), G(K_v))$ to equal $H^1_{Y_v}(\mc{E}^v_2(K/E^v), G(K_v))$. Let $g: \ov{Y_2} \to \ov{Y_v}$ be the base-change map and $f: G(\A_K) \to G(K_v)$ be the natural projection. Then we get a map $\Phi(f, g, \tilde{\pi}): H^1_{\alg}(\mc{E}_2(K/E^v), G(\A_K)) \to H^1_{\alg}(\mc{E}^v_2(K/E^v), G(K_v))$ as in Equation \eqref{phi}. By composing these maps we have now constructed a map
\begin{equation*}
    H^1_{\alg}(\mc{E}_2(K/F), G(\A_K)) \to H^1_{\alg}(\mc{E}^v_2(K/E^v), G(K_v)).
\end{equation*}

We now construct a map 
\begin{equation*}
H^1_{\alg}(\mc{E}^v_2(K/E^v), G(K_v)) \to H^1_{\alg}(\mc{E}_{\iso}(K_v/F_u), G(K_v)),
\end{equation*}
where $H^1_{\alg}(\mc{E}_{\iso}(K_v/F_u), G(K_v))$ is defined as in Key example \ref{keyexample}. For each place $v \in V_K$, we have a natural map $\mu_v : \Gm \to D_2$ coming from the map of character groups $\Z[V_K] \to \Z$ given by projecting to the $v$th coordinate.  In \cite[Remark 7.2]{Kot9}, Kottwitz shows there is a map $\tilde{\mu_v}$ making the following diagram commute.
\begin{equation}
\begin{tikzcd}
1 \arrow[r] & \Gm(K_v) \arrow[r] \arrow[d, "\mu_v"] & \mc{E}_{\iso}(K_v/F_u) \arrow[r] \arrow[d, "\tilde{\mu_v}"] & \Gal(K_v/F_u) \arrow[r] \arrow[d, equal] & 1\\
1 \arrow[r] & D_2(K_v) \arrow[r] & \mc{E}^v_2(K/E^v) \arrow[r] & \Gal(K/E^v) \arrow[r] & 1.
\end{tikzcd}
\end{equation}
Then $\mu_v$ induces a map $g: \ov{Y_v} \to \ov{Y_{\iso}}$ and this gives a map $\Psi(g, \tilde{\mu_v}): H^1_{\alg}(\mc{E}^v_2(K/E^v), G(K_v)) \to H^1_{\alg}(\mc{E}_{\iso}(K_v/F_u), G(K_v))$ as in Equation \eqref{Psi}.

Composing with our earlier map gives the localization map
\begin{equation}{\label{localization}}
  l^F_u:  H^1_{\alg}(\mc{E}_2(K/F), G(\A_K)) \to H^1_{\alg}(\mc{E}_{\iso}(K_v/F_u), G(K_v)).
\end{equation}

\subsection{Total Localization Map}{\label{totalloc}}
    We now want to check that we can promote the localization map defined in the previous subsection to a map:
\begin{equation}{\label{totallocalization}}
   l^F:  H^1_{\alg}(\mc{E}_2(K/F), G(\A_K)) \to \bigoplus\limits_{u \in V_F} H^1_{\alg}(\mc{E}_{\iso}(K_v/F_u), G(K_v)),
\end{equation}
where on the righthand side we choose for each $u \in V_F$ a $v \in V_K$ over $u$. The right-hand side is a direct sum of pointed sets consisting of tuples $(s_u)_u$ such that at all but finitely many $u$, $s_u$ equals the distinguished point $t_u$.

To do so, it suffices to show that for each $[\nu, x] \in H^1_{\alg}(\mc{E}_2(K/F), G(\A_K))$, its image in $H^1_{\alg}(\mc{E}_{\iso}(K_v/F_u), G(K_v))$ is trivial for almost all $v$. To prove this result, we emulate Kottwitz's argument in \cite[\S 14]{Kot9}. In fact, the reader will note that the argument in this subsection is nothing more than a detailed verification that Kottwitz's argument goes through in our setting.

To begin, we recall the setup of \cite[\S 14]{Kot9}. We let $K/F$ be a finite Galois extension of global fields. For a place $v \in K$, we often write $\Gal(K/F)_v$ for the decomposition group of $\Gal(K/F)$ at $v$. We let $V_F$ denote the set of places of $F$. For any subset $S \subset V_F$ we denote by $S_K$ the pre-image of $S$ under the surjection $V_K \twoheadrightarrow V_F$. We let $S_{\infty}$ denote the set of infinite places of $F$. If $S_{\infty} \subset S$ then we have
\begin{itemize}
    \item $F_S := \{ x \in F : x \in \mc{O}_{F_u} \forall u \in V_F \setminus S\}$,
    \item $K_S := \{ x \in K:  x \in \mc{O}_{K_v} \forall v \in V_K \setminus S_K\}$,
    \item $\A_{K,S} := \{ x \in \A_K:  x_v \in \mc{O}_{K_v} \forall v \in V_K \setminus S_K\}$.
\end{itemize}
We define $D_{i, S}$ for $i=1,2,3$ to be the pro-tori with character groups $X_1(S) :=\Z, X_2(S) := \Z[S_K], X_3(S) := \Z[S_K]_0$ respectively. 

We have the following lemma
\begin{lemma}{\label{vanishlem}}
Let $S \subset V_F$ be any subset such that we have equality of the following sets:
\begin{equation*}
    \{ \Gal(K/F)_w : w \in S_K\} = \{\Gal(K/F)_w: w \in V_K\}. 
\end{equation*}
Then
\begin{itemize}
    \item For every subgroup $G' \subset \Gal(K/F)$, we have $H^1(G', D_{2,S}(\A_K))=0$.
    \item For every place $v$ of $K$, we have $H^1(\Gal(K_v/F_u) , D_{2,S}(K_v))=0$
\end{itemize}
\end{lemma}
\begin{proof}
To prove the first statement, we note that similarly to \cite[Lemma 6.2]{Kot9}, we have a canonical isomorphism
\begin{equation*}
    H^1(G', D_{2,S}(\A_K) ) = \prod\limits_{[v] \in (S_K/G')} H^1(G'_v , \A^{\times}_K),
\end{equation*}
where $v \in S_K$ is some lift of $[v]$. For each $v$, we let $F^v \subset K$ be the fixed field of $G'_v$. Then by a standard argument involving Hensel's lemma and Lang's theorem we have an isomorphism
\begin{equation*}
    H^1(G'_v, \A^{\times}_K) = \bigoplus\limits_{w \in V_{F^v}} H^1((G'_v)_{w_K}, K^{\times}_{w_K}),
\end{equation*}
where for each $w \in V_{F^v}$, we have that $w_K$ is some chosen place of $K$ over $w$. The groups on the right all vanish by Hilbert's Theorem 90, which proves the first claim.

The second statement is deduced in the proof of the second part of \cite[Lemma 14.4]{Kot9}.
\end{proof}

We now restrict to those $S \subset V_F$ satisfying the properties of \cite[\S6.1]{Kot9}. Namely:
\begin{itemize}
    \item $S$ contains all infinite places
    \item $S$ contains finite places that ramify in $K$.
    \item For every intermediate field $E$ of $K/F$, every ideal class of $E$ contains an ideal with support in $S_E$.
\end{itemize}
We define  $A_2(S)= \A^{\times}_{K,S}, A_3(S)= K^{\times}_S$ and $A_1(S)$ to be the set of $S_K$ idele classes of $K$. We have a short exact sequence
\begin{equation*}
    1 \to A_3(S) \to A_2(S) \to A_1(S) \to 1.
\end{equation*}
Finally we define the set $\Hom(X,A)$ to be the subgroup of $\Hom(X_1, A_1) \times \Hom(X_2,A_2) \times \Hom(X_3, A_3)$ consisting of triples $(h_1, h_2, h_3)$ such that the following diagram commutes:
\begin{equation*}
\begin{tikzcd}
X_3 \arrow[r] \arrow[d, "h_3"] & X_2\arrow[r] \arrow[d, "h_2"] & X_1 \arrow[d, "h_1"]\\
A_3 \arrow[r] & A_2 \arrow[r]& A_1
\end{tikzcd}    
\end{equation*}
Then Tate \cite{Tatecohomology} defines a canonical class $\alpha \in H^2(\Gal(K/F), \Hom(X, A))$ and defines the classes $\alpha_i \in H^2(\Gal(K/F), \Hom(X_i, A_i)$ via the projections $\pi_i: \Hom(X,A) \to \Hom(X_i, A_i)$. For each $S$, there are analogous constructions and we get similarly $\alpha_i(S) \in H^2(\Gal(K/F), \Hom(X_i(S), A_i(S))$. We then have natural maps $p^S_i$ fitting into a diagram
\begin{equation*}
\begin{tikzcd}
0 \arrow[r] & X_3(S) \arrow[r] \arrow[d, "p^S_3"] & X_2(S) \arrow[r] \arrow[d, "p^S_2"] & X_1(S) \arrow[r] \arrow[d, "p^S_1"] & 0\\
0 \arrow[r] & X_3 \arrow[r] & X_2 \arrow[r] & X_1 \arrow[r]& 0
\end{tikzcd}
\end{equation*}
and hence inducing a a morphism $p^S$ of exact sequences. We also have maps $k^S_i$ and a morphism $k^S$ of exact sequences
\begin{equation*}
\begin{tikzcd}
0 \arrow[r] & A_3(S) \arrow[r] \arrow[d, "k^S_3"] & A_2(S) \arrow[r] \arrow[d, "k^S_2"] & A_1(S) \arrow[r] \arrow[d, "k^S_1"] & 0\\
0 \arrow[r] & A_3 \arrow[r] & A_2 \arrow[r] & A_1 \arrow[r]& 0
\end{tikzcd}
\end{equation*}
We now record a lemma of Kottwitz comparing $\alpha(S)$ and $\alpha$. Observe that we have the following diagram
\begin{equation*}
\begin{tikzcd}
\Hom(X(S), A(S)) \arrow[r, "k^S"] \arrow[d, "\pi_i", swap] & \Hom(X(S),A) \arrow[d, "\pi_i", swap]& \Hom(X,A) \arrow[l, "p^S", swap] \arrow[d, "\pi_i"]\\
\Hom(X_i(S), A_i(S)) \arrow[r, "k^S_i", swap] & \Hom(X_i(S), A_i) & \Hom(X_i, A_i) \arrow[l, "p^S_i"]
\end{tikzcd}    
\end{equation*}
Then we have
\begin{lemma}{\label{alphailem}}
\begin{equation}
    k^S(\alpha(S))=p^S(\alpha)
\end{equation}
and
\begin{equation}
    k^S_i(\alpha_i(S))=p^S_i(\alpha_i)
\end{equation}
for $i=1,2,3$.
\end{lemma}
\begin{proof}
This is \cite[Lemma 14.6]{Kot9}.
\end{proof}
With the above notation and preliminaries, we now return to our connected reductive group $G$. We extend $G$ to a smooth affine group scheme $\mc{G}$ defined over $F_{S(\mc{G})}$, where $S(\mc{G})$ is a finite subset of $V_F$ containing all infinite places. We now define a subset $S \subset V_F$ to be \emph{adequate} if it satisfies all of the following conditions:
\begin{itemize}
    \item $S$ is finite
    \item $S$ contains $S(\mc{G})$
    \item $S$ contains all finite places that ramify in $K$
    \item For every intermediate field $E$ of $K/F$, every ideal class of $E$ contains an ideal in the support of $S_E$
    \item $S$ satisfies the condition of Lemma \ref{vanishlem}.
\end{itemize}
Such sets exist (for instance see \cite[Remark 14.4]{Kot9}) and if $S' \subset V_F$ is finite $S' \supset S$ for $S$ adequate, then $S'$ is also adequate. The first condition implies that $D_{2,S}$ is a torus and in addition, the third implies that $D_{2,S}$ extends uniquely to a torus $\mc{D}_{2,S}$ over $F_S$. Indeed, this last fact is the same as a factoring of the Galois action on $X^*(D_{2,S})$ through $\Gal(M/F)$ where $M$ is the maximal extension of $F$ unramified outside of $S$. Since $K \subset M$, and $D_{2,S}$ splits over $K$, such a factoring indeed exists.

 We will need the following lemma;
\begin{lemma}{\label{othervanishlem}}
Let $u \in V_F \setminus S$ and $v \in V_K$ lying over $u$. Then the following Tate cohomology groups satisfy
\begin{equation*}
    H^r(\Gal(K_v/F_u), \mc{D}_{2,S}(\mc{O}_v))=0
\end{equation*}
for all $r \in \Z$.
\end{lemma}
\begin{proof}
\cite[Lemma 14.7]{Kot9}
\end{proof}
Now, for adequate $S$, we want to construct a set $H^1_{\alg}(\mc{E}_2(S), \mc{G}(\A_{K,S}))$ where $\mc{E}_2(S)$ is the extension
\begin{equation*}
    1 \to \mc{D}_{2,S}(\A_{K,S}) \to \mc{E}_2(S) \to \Gal(K/F) \to 1,
\end{equation*}
with corresponding cohomology class $\alpha_2(S)$. We define $H^1_{\alg}(\mc{E}_2(S), \mc{G}(\A_{K,S}))$ via the $H^1_Y(E,M)$ construction above. We let $M=\mc{G}(\A_{K,S})$, $E=\mc{E}_2(S)$, $Y=\Hom_{K_S}(\mc{D}_{2,S} , \mc{G})$. Let $\ov{Y}$ be the orbit of $Y$ inside $\prod\limits_{v \in V_K} \Hom_{K_v}(\mc{D}_{2,S}, \mc{G})$ under the action of $\mc{G}(\A_{K,S})$ and let $\xi: \ov{Y} \to \Hom(\mc{D}_{2,S}(\A_{K,S}) , \mc{G}(\A_{K,S}))$ be the natural map.

Next, we want to define a canonical map
\begin{equation}{\label{finitemap}}
    H^1_{\alg}(\mc{E}_2(S), \mc{G}(\A_{K,S})) \to H^1_{\alg}(\mc{E}_2(K/F), G(\A_K))
\end{equation}
as a composition
\begin{equation*}
\begin{tikzcd}
    H^1_{\alg}(\mc{E}_2(S), \mc{G}(\A_{K,S})) \arrow[r, "BC"] &  H^1_{\alg}(\mc{E}^K_2(S), G(\A_K)) \arrow[r, "p^*"]  & H^1_{\alg}(\mc{E}_2(K/F), G(\A_K)),
\end{tikzcd}    
\end{equation*}
where $\mc{E}^K_2(S)$ is the pushout of $\mc{E}_2(S)$ along the map $\mc{D}_{2,S}(\A_{K,S}) \to D_{2,S}(\A_K)$.  We first define $H^1_{\alg}(\mc{E}^K_2(S), G(\A_K))$ via the $H^1_Y(E,M)$ construction letting $E=\mc{E}^K_2(S)$, $M=G(\A_K)$, $Y= \Hom_K(D_{2,S}, G)$. We let $\ov{Y}$ be the $G(\A_K)$-orbit of $Y$ in $\prod\limits_{v \in V_K} \Hom_{K_v}(D_{2,S} , G)$ and $\xi: \ov{Y} \to \Hom( D_{2,S}(\A_K), G(\A_K))$ be the natural map. The map $BC$ is then defined via the $\Phi$ construction.

Now, by Lemma \ref{alphailem} we have $k^S_2(\alpha_2(S)) = p^S_2(\alpha_2)$. Since $k^S_2(\alpha_2(S))$ is a class in $H^2(\Gal(K/F), D_{2,S}(\A_K))$, this implies there is a map $\tilde{p^S_2}$ giving a map of extensions:
\begin{equation*}
\begin{tikzcd}
1 \arrow[r] & D_2(\A_K) \arrow[r] \arrow[d, "p^S_2"] & \mc{E}_2(K/F) \arrow[r] \arrow[d, "\tilde{p^S_2}"] & \Gal(K/F) \arrow[d, equal] \arrow[r] & 1 \\
1 \arrow[r] & D_{2,S}(\A_K) \arrow[r] & \mc{E}^K_2(S) \arrow[r] & \Gal(K/F) \arrow[r] & 1.
\end{tikzcd}    
\end{equation*}
Since by Lemma \ref{vanishlem} we have that $H^1(\Gal(K/F), D_{2,S}(\A_K))=0$, the induced map $\tilde{p^S_2}^*: H^1_{\alg}(\mc{E}^K_2(S), G(\A_K)) \to H^1_{\alg}(\mc{E}_2(K/F), G(\A_K))$ coming from the $\Psi$ construction does not depend on the choice of $\tilde{p^S_2}$. Hence, we call this map $p^*$. We need the following lemma:
\begin{lemma}{\label{finitemapfactor}}
For each $b \in H^1_{\alg}(\mc{E}_2(K/F), G(\A_K))$, there exists an adequate set $S$ so that $b$ lies in the image of the map in Equation \ref{finitemap}.
\end{lemma}
\begin{proof}
Pick a cocycle $(\nu, x)$ representing $b$. We have $\nu: D_2 \to G$ is a map over $K$. Since $X^*(D_2)=\Z[V_K]$, we can find an adequate set $S$ such that $\nu$ factors to give a map $\nu': D_{2,S} \to G$. This implies that $(\nu, x)$ comes from an algebraic cocycle $(\nu', x')$ of $\mc{E}^K_2(S)$ for some adequate $S$. By enlarging $S$, we can assume that $\nu'$ comes from a map $\nu'': \mc{D}_{2,S} \to \mc{G}_{K_S}$. Since $\mc{D}_{2,S}(\A_{K,S})$ has finite index in $\mc{E}_2(S)$, we can enlarge $S$ so that the restriction $x''$ of $x'$ to $\mc{E}_2(S)$ has image in $\mc{G}(\A_{K,S})$. We then note that $(\nu'', x'') \in Z^1_{\alg}(\mc{E}_2(S), \mc{G}(\A_{K,S}))$ and  maps to $(\nu, x)$. This completes the lemma. 
\end{proof}
We now construct for every adequate set $S$ and place $u \notin S$, a localization map $l^S_u$ such that the following diagram commutes:
\begin{equation}{\label{keydiagram}}
\begin{tikzcd}
H^1_{\alg}(\mc{E}_2(S), \mc{G}(\A_{K,S}))  \arrow[r] \arrow[d, "l^S_u"] & H^1_{\alg}(\mc{E}_2(K/F), G(\A_K)) \arrow[d, "l^F_u"]\\
H^1(\Gal(K_v/F_u), \mc{G}(\mc{O}_v)) \arrow[r] & H^1_{\alg}(\mc{E}_{\iso}(K_v/F_u), G(K_v))
\end{tikzcd}
\end{equation}
In the above diagram top map is the canonical one we constructed previously and the bottom is the composition of the map on cohomology induced by $\mc{G}(\mc{O}_v) \hookrightarrow G(K_v)$ and the canonical inclusion $H^1(\Gal(K_v/F_u), G(K_v)) \hookrightarrow H^1_{\alg}(\mc{E}_{\iso}(K_v/F_u) , G(K_v))$. If we can construct such a diagram, we will have completed our construction of the total localization map. This is because the bottom left group in the above diagram is trivial by Lang's theorem and Hensel's lemma.

As in our construction of the localization map of Equation \eqref{localization}, we may restrict the entire top row of the diagram to be over the extension $K/E^v$ where $E^v$ is the fixed field of the decomposition group of $v$. Hence we can and do assume that $\Gal(K_v/F_u)=\Gal(K/F)$.

To construct the above commutative diagram, we will construct a larger diagram:
\begin{equation}{\label{bigdiagram}}
\begin{tikzcd}
H^1_{\alg}(\mc{E}_2(S), \mc{G}(\A_{K,S})) \arrow[d, "Loc"] \arrow[r, "BC"] & H^1_{\alg}(\mc{E}^K_2(S), G(\A_K)) \arrow[d, "Loc"] \arrow[r, "p^*"] & H^1_{\alg}(\mc{E}_2(K/F), G(\A_K)) \arrow[d, "Loc"]\\
H^1_{\alg}(\mc{E}^{\mc{O}_v}_2(S), \mc{G}(\mc{O}_v)) \arrow[r, "BC"] \arrow[d, "\mu^*_0"] & H^1_{\alg}(\mc{E}^{K_v}_2(S), G(K_v)) \arrow[r, "p^*"] \arrow[d, "\mu^*_0"] & H^1_{\alg}(\mc{E}^v_2(K/F), G(K_v)) \arrow[d, "\mu^*_v"]\\
H^1(\Gal(K_v/F_u), \mc{G}(\mc{O}_v)) \arrow[r] & H^1(\Gal(K_v/F_u), G(K_v)) \arrow[r] & H^1_{\alg}(\mc{E}_{\iso}(K_v/F_u), G(K_v)).
\end{tikzcd}
\end{equation}
The top two maps compose to give Equation \ref{finitemap}, the right vertical maps compose to give the localization map $l^F_u$, and the bottom maps compose to give the bottom map in Diagram \ref{keydiagram}. Hence, if we can construct all the relevant objects and maps and show that the diagram commutes, the above left arrows will compose to give $l^S_u$ as desired.

We need to define two of the above sets: $H^1_{\alg}(\mc{E}^{\mc{O}_v}_2(S), \mc{G}(\mc{O}_v))$ and  $H^1_{\alg}(\mc{E}^{K_v}_2(S), G(K_v))$. These are defined analogously to the way $\mc{E}^v_2(K/F)$ is defined relative to $
\mc{E}_2(K/F)$. In particular, we define $\mc{E}^{\mc{O}_v}_2(S)$ as the pushout of $\mc{E}_2(S)$ via the map $\mc{D}_{2,S}(\A_{K,S}) \to \mc{D}_{2,S}(\mc{O}_v)$ and $\mc{E}^{K_v}_2(S)$ as the pushout of $\mc{E}^K_2(S)$ via the map $D_{2,S}(\A_K)\to D_{2,S}(K_v)$. We then define $H^1_{\alg}(\mc{E}^{\mc{O}_v}_2(S), \mc{G}(\mc{O}_v))$ using the $H^1_Y(E,M)$ construction with $\ov{Y}=Y=\Hom_{\mc{O}_v}(\mc{D}_{2,S}, \mc{G})$. Similarly, we define $H^1_{\alg}(\mc{E}^{K_v}_2(S), G(K_v))$ such that $\ov{Y}=Y=\Hom_{K_v}(D_{2,S} , G)$.

We now turn to constructing the maps and showing they commute. To start, we have the following commutative diagram:
\begin{equation}{\label{torusdiagram}}
\begin{tikzcd}
\mc{D}_{2,S}(\A_{K,S}) \arrow[d] \arrow[r] & D_{2,S}(\A_K) \arrow[d] & D_2(\A_K) \arrow[l, "p"] \arrow[d]\\
 \mc{D}_{2,S}(\mc{O}_v) \arrow[r] & D_{2,S}(K_v)& D_2(K_v) \arrow[l, "p"] \\
 1 \arrow[u, "\mu_0"] \arrow[r] & 1 \arrow[u, "\mu_0"] & \Gm(K_v) \arrow[l, "q"] \arrow[u, "\mu_v"],
\end{tikzcd}    
\end{equation}
where the map $\mu_0$ and $q$ are trivial and $\mu_v$ was defined as part of the localization map. The unlabeled maps are induced from the following commutative diagram:
\begin{equation*}
\begin{tikzcd}
 \A_{K,S} \arrow[r] \arrow[d] & \A_K \arrow[d] & \A_K \arrow[d] \\
 \mc{O}_v \arrow[r] & K_v  & K_v
\end{tikzcd}
\end{equation*}
We remark that Diagram \ref{torusdiagram} is commutative since $p \mu_v$ is trivial because $u \notin S$. 
\begin{claim}
We now claim that for each group $A$ in Diagram \ref{torusdiagram}, the group $H^1(\Gal(K_v/F_u), A)$ vanishes. Indeed, these groups vanish by \cite[Lemma 6.5]{Kot9}, Lemma \ref{vanishlem}, Lemma \ref{othervanishlem}, and Hilbert's Theorem 90. As a result of this claim all, the sets in \ref{bigdiagram} are well-defined up to canonical isomorphism.
\end{claim}
\begin{claim}
We claim that the maps in \ref{torusdiagram} can be extended to homomorphisms of extensions as follows, and that each smaller square is essentially commutative in that it is commutative up to conjugation by $D_{2,S}(K_v)$.
\begin{equation*}{\label{extensiondiagram}}
\begin{tikzcd}
\mc{E}_2(S) \arrow[r] \arrow[d] & \mc{E}^K_2(S) \arrow[d] & \mc{E}_2(K/F) \arrow[l, "\tilde{p}"] \arrow[d] \\
\mc{E}^{\mc{O}_v}_2(S) \arrow[r] & \mc{E}^{K_v}_2(S) & \mc{E}^v_2(K/F) \arrow[l, "\tilde{p}"]  \\
 \Gal(K_v/F_u) \arrow[u, "\tilde{\mu_0}"] \arrow[r] & \Gal(K_v/F_u) \arrow[u, "\tilde{\mu_0}"] & \mc{E}_{\iso}(K_v/F_u) \arrow[l, "\tilde{q}"] \arrow[u, "\tilde{\mu_v}"]
\end{tikzcd}    
\end{equation*}
If we can construct such a diagram, the essential commutativity will follow from the previous claim. The diagram exists by the following claim.
\end{claim}
\begin{claim}
We claim that for each $A$ in Diagram \ref{torusdiagram}, there is a unique element $\alpha_A \in H^2(\Gal(K_v/F_u), A)$ such that 
\begin{itemize}
    \item $\alpha_{\mc{D}_{2,S}(\A_{K,S})}=\alpha_2(S)$,
    \item  $\alpha_{D_2(\A_K)}=\alpha_2$,
    \item $\alpha_1=1$,
    \item $\alpha_{\Gm(K_v)}=\alpha(K_v/F_u)$ (the local fundamental class),
    \item Each arrow $A \to A'$ maps $\alpha_A$ to $\alpha_{A'}$,
    \item Each $\alpha_A$ gives the cohomology class corresponding to the relevant extension in Diagram \ref{extensiondiagram}.
\end{itemize}
To verify the claim, the commutativity of Diagram \ref{torusdiagram} implies that we need only check that along each outer edge, the maps $A' \to A \leftarrow A''$ map the canonical elements $\alpha_{A'}, \alpha_{A''}$ to the same $\alpha_A \in H^2(\Gal(K_v/F_u), A)$. For the top of the diagram, this follows from Lemma \ref{alphailem}, for the left this follows from Lemma \ref{othervanishlem}, for the bottom this is trivial, and for the right this is \cite[Equation (7.7)]{Kot9}. That these cohomology classes correspond to the various extensions is clear from their definitions. This implies the requisite maps of extensions in the previous claim do indeed exist.
\end{claim}

We are now in a position to define the maps in Diagram \ref{bigdiagram}. We first give a diagram relating the various sets $\ov{Y}$. We use the notation $G \cdot A$ to denote the $G$-orbit of $A$ inside a natural product space.
\begin{equation}{\label{ovYeqn}}
\begin{tikzcd}
\mc{G}(\A_{K,S}) \cdot \Hom_{K_S}(\mc{D}_{2,S} , 
\mc{G}) \arrow[d] \arrow[r] & G(\A_K) \cdot \Hom_K(D_{2,S}, G) \arrow[d], \arrow[r, "p"] & G(\A_K) \cdot \Hom_K(D_2, G) \arrow[d] \\
 \Hom_{\mc{O}_v}(\mc{D}_{2,S}, \mc{G}) \arrow[d, "\mu_0"] \arrow[r] & \Hom_{K_v}(D_{2,S} , G) \arrow[d, "\mu_0"] \arrow[r, "p"] & \Hom_{K_v}( D_2, G) \arrow[d, "\mu_v"] \\
 1 \arrow[r]  & 1 \arrow[r, "q"] & \Hom_{K_v}(\Gm , G) 
\end{tikzcd}    
\end{equation}
In the above diagram, the horizontal maps have all been defined or are clear from inspection. The $\mu_0$ maps are trivial and $\mu_v$ was defined when we defined the localization map. An element $\nu \in \mc{G}(\A_{K,S}) \cdot \Hom_{K_S}(\mc{D}_{2,S} ,
\mc{G})$ consists of a sequence of maps $\nu_v$ for each place $v$ of $K$ such that the if $v \notin S_K$, then $\nu_v$ is defined over $\mc{O}_v$. In particular, there is a natural projection to $\Hom_{\mc{O}_v}(\mc{D}_{2,S}, \mc{G})$ given by $\nu \mapsto \nu_v$. The other vertical maps are defined analogously. It is also clear that each map $x: \ov{Y} \to  \ov{Y'}$ satisfies $x(Y) \subset Y'$.

We also have a diagram of $\Gal(K_v/F_u)$-groups:
\begin{equation}{\label{Mdiagram}}
\begin{tikzcd}
\mc{G}(\A_{K,S}) \arrow[r] \arrow[d] & G(\A_K) \arrow[d] \arrow[r, equal] & G(\A_K) \arrow[d] \\
\mc{G}(\mc{O}_v) \arrow[d, equal] \arrow[r] & G(K_v) \arrow[d, equal] \arrow[r, equal] & G(K_v) \arrow[d, equal] \\
 \mc{G}(\mc{O}_v) \arrow[r] & G(K_v) \arrow[r, equal] & G(K_v).
\end{tikzcd}    
\end{equation}
We can now define all the maps in Diagram \ref{bigdiagram} via Diagrams \ref{extensiondiagram}, \ref{ovYeqn}, \ref{Mdiagram}. In particular, all are examples of the $\Phi$ and $\Psi$ constructions. We then observe that all the squares commute by Lemma \ref{Philem}, Lemma \ref{Psilem}, and Lemma \ref{PhiPsilem}.

This finishes the proof and hence establishes the existence of the total localization map.
% \subsection{Total Localization Map: Proof of Bijection}

% We now prove that the total localization map $l^F$ as in Equation \eqref{totallocalization} gives a bijection.

% We first prove injectivity. Suppose $[\nu, x], [\nu', x'] \in H^1_{\alg}(\mc{E}_2(K/F), G(\A_K))$ such that $l^F([\nu, x]) = l^F([\nu', x'])$. Let $(\nu, x), (\nu, x')$ be cocycle representatives of these cohomology classes, let $\nu_{K_v}, \nu'_{K_v}$ be the base changes of $\nu$ and $\nu'$ respectively and let $(\nu^v, x^v), (\nu'^v, x'^v) \in Z^1_{\alg}(\mc{E}_{\iso}(K_v/F_u), G(K_v))$ be the cocycles given by applying $l^F_u$ to $(\nu, x)$ and $(\nu', x')$ respectively. Since $(\nu^v, x^v)$

% For each $u \in V_F$ we have $(\nu^v, x^v) \sim (\nu'^v, x'^v)$. Hence, there exists $g_v \in G(K_v)$ such that $g_v \cdot (\nu^v, x^v) = (\nu'^v, x'^v)$.

% sketcherino: first prove that some suitable twist of $x$ and $x'$ agree on a section of $\Gal(K/F)$ and similarly that twist of $\nu$ and $\nu'$ agree on the $v$s. First fact implies $\nu=\nu'$ which then implies $x=x'$. My concern is that you could have the following thing. $\nu$ and $\nu'$ are nontrivial at exactly $2$ places $v_1$ and $v_2$ and differ by a nontrivial weyl element at one place but not the other. Then I don't see how to show they're equivalent but potentially the localizations are. Note this can't happen on the torus case.

\subsection{Basic Subsets}
We now define, for $K/F$ a finite extension of number fields, a set $H^1_{\bas}(\mc{E}_2(K/F), G(\A_K))$. In particular, we define $H^1_{\bas}(\mc{E}_2(K/F), G(\A_K)) \subset H^1_{\alg}(\mc{E}_2(K/F), G(\A_K))$ to be the set of classes represented by algebraic cocycles $(\nu, x)$ such that $\nu: D_2 \to G$ factors through the center $Z_G$ of $G$. This set is given via the $H^1_Y(E,M)$ construction with $Y_{\bas}=\ov{Y_{\bas}} := \Hom_{K}(D_2,Z_G)$.

We now define a set $H^1_{\bas}(\mc{E}_1(K/F), G(\A_K)/Z_G(K) )$. We let $D_1=\Gm_F$ and define $Y_1$ to be the subset of $\Hom_K(D_1 , G)$ factoring through $Z_G$. We let $\ov{Y_1}=Y_1$ and define $E=\mc{E}_1(K/F)$ and $M=G(\A_K)/Z_G(K))$. We define $\xi_1: Y_1 \to \Hom_K(D_1(\A_K)/D_1(K) , G(\A_K)/Z_G(K))$ to be the natural map.

\begin{remark}
It would be tempting to try to define a set $H^1_{\alg}(\mc{E}_1(K/F), G(\A_K)/Z_G(K))$ where we do not require the elements of $Y_1$ to be central. However, the map $\xi_1$ does not make sense in this case.
\end{remark}
\subsection{Some Maps of Global Cohomology}{\label{bgmaps}}
We now claim there are canonical maps
\begin{equation}{\label{globalmaps}}
    H^1_{\bas}(\mc{E}_3(K/F), G(K)) \to H^1_{\bas}(\mc{E}_2(K/F), G(\A_K)) \to H^1_{\bas}(\mc{E}_1(K/F), G(\A_K)/Z_G(K) ).
\end{equation}
We will spend the rest of this subsection constructing these maps.

We begin by constructing the map $H^1_{\bas}(\mc{E}_2(K/F), G(\A_K)) \to H^1_{\bas}(\mc{E}_1(K/F), G(\A_K)/Z_G(K) )$. From \cite[\S 6.3]{Kot9}, we have an extension:
\begin{equation*}
    1 \to D_2(\A_K)/D_2(K) \to \mc{F} \to \Gal(K/F) \to 1,
\end{equation*}
and maps of extensions
\begin{equation*}
\begin{tikzcd}
1 \arrow[r] & D_2(\A_K) \arrow[d, "a", swap] \arrow[r] & \mc{E}_2(K/F) \arrow[r] \arrow[d, "\tilde{a}", swap] & \Gal(K/F) \arrow[r] \arrow[d, equal] & 1 \\
1 \arrow[r] & D_2(\A_K)/D_2(K) \arrow[r] & \mc{F} \arrow[r] & \Gal(K/F) \arrow[r] & 1 \\
1 \arrow[r] & D_1(\A_K)/D_1(K) \arrow[u, "b"] \arrow[r] & \mc{E}_1(K/F) \arrow[r] \arrow[u, "\tilde{b}"] & \Gal(K/F) \arrow[u, equal] \arrow[r] & 1
\end{tikzcd}
\end{equation*}
where $a$ is the natural projection and $b$ is induced by the map of characters $\Sigma: \Z[V_K] \to \Z$ where $\Sigma(v)=1$ for each $v \in V_K$. This extension satisfies $a(\alpha_1)=b(\alpha_2) \in H^2(\Gal(K/F), D_2(\A_K) / D_2(K))$.

\begin{lemma}
We have $H^1(\Gal(K/F), D_2(\A_K)/D_2(K))=0$ and hence $\tilde{b}$ and $\tilde{a}$ are unique up to conjugacy by $D_2(\A_K) / D_2(K)$.
\end{lemma}
\begin{proof}
By \cite[Lemma A.6]{Kot9}, we have 
\begin{equation*}
    H^1(\Gal(K/F) , D_2(\A_K) / D_2(K) ) = \prod\limits_{v \in G \setminus V_K} H^1(G_v, \A^{\times}_K / K^{\times}).
\end{equation*}
In particular, it suffices to show that for each $G' \subset\ \Gal(K/F)$, we have $H^1(G', \A^{\times}_K/ K^{\times})=0$. Since $G'=\Gal(K/E)$ where $E \subset K$ is the fixed field of $G'$, this is a standard fact from global class field theory. 
\end{proof}
We construct the map in Equation \ref{globalmaps} using the following series of morphisms:
\begin{equation*}
\begin{tikzcd}
    H^1_{\bas}(\mc{E}_2(K/F), G(\A_K) ) \arrow[d] \\  
    H^1_{Y_{\bas}}(\mc{E}_2(K/F), G(\A_K)/Z_G(K)) \\
    H^1_{Y_{\bas}}(\mc{F} , G(\A_K)/Z_G(K)) \arrow[u, "\tilde{a}^*"] \arrow[d, "\tilde{b}^*", swap]\\
    H^1_{\bas}(\mc{E}_1(K/F), G(\A_K) / Z_G(K))
\end{tikzcd}
\end{equation*}
The first map is induced via functoriality from $G(\A_K) \to G(\A_K)/Z_G(K)$ and the maps $\tilde{a}^*, \tilde{b}^*$ are defined via pullback. We note that $\tilde{a}^*$ and $\tilde{b}^*$ only depend on $a$ and $b$ respectively by the previous lemma.

To complete the construction, we need to show that the map $\tilde{a}^*: H^1_{Y_{\bas}}(\mc{F} , G(\A_K)/Z_G(K)) \to H^1_{Y_{\bas}}(\mc{E}_2(K/F), G(\A_K)/Z_G(K))$ is an isomorphism. The map $D_2(\A_K) \to D_2(\A_K)/ D_2(K)$ is a surjection and hence it follows that the map $\tilde{a}: \mc{E}_2(K/F) \to \mc{F}$ is also surjective. Hence, $\tilde{a}^*$ is injective. To prove surjectivity of $\tilde{a}^*$, we take a cocycle of $(\nu, x)$ of $\mc{E}_2(K/F)$ valued in $G(\A_K)/Z_G(K)$. Then for $d \in D_2(K)$, we have $x_d = \nu(d) \mod Z_G(K) = 0$. In particular, $x$ factors through $\mc{F}$.

We now construct the map $H^1_{\bas}(\mc{E}_3(K/F), G(K)) \to H^1_{\bas}(\mc{E}_2(K/F), G(\A_K))$. Again, from \cite[\S 6.3]{Kot9}, we have an extension:
\begin{equation*}
    1 \to D_3(\A_K) \to \mc{F'} \to \Gal(K/F) \to 1,
\end{equation*}
and maps of extensions
\begin{equation*}
\begin{tikzcd}
1 \arrow[r] & D_3(K) \arrow[d, "a'", swap] \arrow[r] & \mc{E}_3(K/F) \arrow[r] \arrow[d, "\tilde{a'}", swap] & \Gal(K/F) \arrow[r] \arrow[d, equal] & 1 \\
1 \arrow[r] & D_3(\A_K) \arrow[r] & \mc{F'} \arrow[r] & \Gal(K/F) \arrow[r] & 1 \\
1 \arrow[r] & D_2(\A_K) \arrow[u, "b'"] \arrow[r] & \mc{E}_2(K/F) \arrow[r] \arrow[u, "\tilde{b'}"] & \Gal(K/F) \arrow[u, equal] \arrow[r] & 1,
\end{tikzcd}
\end{equation*}
where $a'(\alpha_3)=b'(\alpha_2)$.
\begin{lemma}
We have $H^1(\Gal(K/F), D_3(\A_K))=0$ and hence $\tilde{b'}$ and $\tilde{a'}$ are unique up to conjugacy by $D_3(\A_K)$.
\end{lemma}
\begin{proof}
We have 
\begin{equation*}
    H^1(\Gal(K/F) , D_3(\A_K) ) = \bigoplus\limits_{u \in V_F} H^1(\Gal(K_v/F_u), D_3(K_v)).
\end{equation*}
The groups on the right vanish by \cite[Lemma 7.1.(1)]{Kot9}.
\end{proof}
Let $Y'_{\bas}$ be the set of homomorphisms $\Hom_K(D_3 , Z_G)$. Then we construct the map in Equation \ref{globalmaps} using the following series of morphisms:
\begin{equation*}
\begin{tikzcd}
    H^1_{\bas}(\mc{E}_3(K/F), G(K) ) \arrow[d] \\  
    H^1_{Y'_{\bas}}(\mc{E}_3(K/F), G(\A_K)) \\
    H^1_{Y'_{\bas}}(\mc{F'} , G(\A_K)) \arrow[u, "\tilde{a'}^*"] \arrow[d, "\tilde{b'}^*", swap]\\
    H^1_{\bas}(\mc{E}_2(K/F), G(\A_K))
\end{tikzcd}
\end{equation*}
The first map is induced via functoriality from $G(K) \to G(\A_K)$ and the maps $\tilde{a'}^*, \tilde{b'}^*$ are defined via pullback. We note that $\tilde{a'}^*$ and $\tilde{b'}^*$ only depend on $a'$ and $b'$ respectively by the previous lemma.

To complete the construction, we need to show that the map $\tilde{a'}^*: H^1_{Y'_{\bas}}(\mc{F'} , G(\A_K)) \to H^1_{Y'_{\bas}}(\mc{E}_3(K/F), G(\A_K))$ is an isomorphism. On the one hand the map is surjective since if we have a cocycle $(\nu, x)$ of $\mc{E}_3(K/F)$, we can push it forward to $\mc{F'}$ to get a cocycle $(\nu, x')$ that pulls back to $(\nu, x)$. On the other hand, if $(\nu_1, x_1)$ and $(\nu_2, x_2)$ are algebraic cocycles of $\mc{F'}$ that pull back to cocycles that are equivalent via $m \in G(\A_K)$, then this implies that $\nu_1 = \nu_2$ and that for all $w \in \mc{F'}$, $m^{-1} x_1(w) w(m)=x_2$ since this is true on $D_3(\A_K)$ and the image of $\mc{E}_3$ and these sets generate $\mc{F'}$.

We now show that the map
\begin{equation*}
    H^1_{\bas}(\mc{E}_3(K/F), G(K)) \to H^1_{\bas}(\mc{E}_2(K/F), G(\A_K))
\end{equation*}
commutes with localization. In \cite[\S 7]{Kot9}, Kottwitz defines  for a place $v$ of $K$, a localization map 
\begin{equation*}
  H^1_{\alg}(\mc{E}_3(K/F), G(K)) \to H^1_{\alg}(\mc{E}_{\iso}(K_v/F_u), G(K_v))
\end{equation*}
that is entirely analogous to the localization map for $H^1(\mc{E}_2(K/F), G(\A_K))$. Then we have
\begin{lemma}{\label{loccompatlem}}
The following diagram commutes
\begin{equation*}
\begin{tikzcd}
H^1_{\bas}(\mc{E}_3(K/F), G(K)) \arrow[r] \arrow[d] & H^1_{\bas}(\mc{E}_2(K/F), G(\A_K)) \arrow[d] \\
H^1_{\bas}(\mc{E}_{\iso}(K_v/F_u), G(K_v)) \arrow[r, equals] & H^1_{\bas}(\mc{E}_{\iso}(K_v/F_u), G(K_v)),
\end{tikzcd}  
\end{equation*}
where the vertical maps are the respective localization maps and the upper horizontal map is the one constructed in this section.
\end{lemma}
\begin{proof}
First it suffices to show the lemma in the case where the base field $F$ is the fixed field $E^v \subset K$ of the decomposition group of $K$ at $v$ so we assume this.

Then all the maps in the above diagram are compositions of $\Psi$ and $\Phi$ maps. Hence our strategy is to expand the above diagram to one of the form
\begin{equation*}
\adjustbox{scale=0.80, center}{
    \begin{tikzcd}
H^1_{\bas}(\mc{E}_3(K/E^v), G(K)) \arrow[r] \arrow[d] & H^1_{Y'_{\bas}}(\mc{E}_3(K/E^v), G(\A_K)) \arrow[d] & H^1_{Y'_{\bas}}(\mc{F}', G(\A_K))  \arrow[l, "\sim"] \arrow[r] \arrow[d] & H^1_{\bas}(\mc{E}_2(K/E^v), G(\A_K)) \arrow[d] \\
H^1_{\bas}(\mc{E}^v_3(K/E^v), G(K_v)) \arrow[d] \arrow[r, equals] & H^1_{\bas}(\mc{E}^v_3(K/E^v), G(K_v)) \arrow[d] & H^1_{Y'_{v, \bas}}(\mc{F}'_v, G(K_v)) \arrow[l, "\sim"] \arrow[d] \arrow[r] & H^1_{\bas}(\mc{E}^v_2(K/E^v), G(K_v)) \arrow[d] \\
H^1_{\bas}(\mc{E}_{\iso}(K_v/F_v), G(K_v))  \arrow[r, equal] & H^1_{\bas}(\mc{E}_{\iso}(K_v/F_v), G(K_v))  \arrow[r, equal] &  H^1_{\bas}(\mc{E}_{\iso}(K_v/F_v), G(K_v)) \arrow[r, equal] & H^1_{\bas}(\mc{E}_{\iso}(K_v/F_v), G(K_v)) 
\end{tikzcd} 
}
\end{equation*}
where each small square will consist of $\Psi$ and $\Phi$ maps and hence commute by Lemmas \ref{Philem}, \ref{Psilem}, \ref{PhiPsilem}.
The above diagram will be induced from a diagram of extensions:
\begin{equation*}
\adjustbox{scale=1, center}{
    \begin{tikzcd}
\mc{E}_3(K/E^v) \arrow[r, equal] \arrow[d] & \mc{E}_3(K/E^v) \arrow[d] \arrow[r] & \mc{F}'  \arrow[d] & \mc{E}_2(K/F) \arrow[d] \arrow[l] \\
\mc{E}^v_3(K/E^v)  \arrow[r, equal] & \mc{E}^v_3(K/E^v) \arrow[r] & \mc{F}'_v & \mc{E}^v_2(K/E^v) \arrow[l]  \\
\mc{E}_{\iso}(K_v/F_v) \arrow[r, equal] \arrow[u] & \mc{E}_{\iso}(K_v/F_v)  \arrow[r, equal] \arrow[u] &  \mc{E}_{\iso}(K_v/F_v) \arrow[r, equal] \arrow[u] & \mc{E}_{\iso}(K_v/F_v) \arrow[u]
\end{tikzcd} 
}
\end{equation*}
as well as a diagram of coefficients
\begin{equation*}
\adjustbox{scale=1, center}{
    \begin{tikzcd}
G(K) \arrow[r] \arrow[d] & G(\A_K) \arrow[d] \arrow[r, equal] & G(\A_K)  \arrow[d] & G(\A_K) \arrow[d] \arrow[l, equal] \\
G(K_v)  \arrow[r, equal] & G(K_v) \arrow[r, equal] & G(K_v) & G(K_v) \arrow[l, equal]  \\
G(K_v) \arrow[r, equal] \arrow[u, equal] & G(K_v)  \arrow[r, equal] \arrow[u, equal] &  G(K_v) \arrow[r, equal] \arrow[u, equal] & G(K_v) \arrow[u, equal]
\end{tikzcd} 
}
\end{equation*}
and a diagram of $\ov{Y}$ sets
\begin{equation*}
\adjustbox{scale=1, center}{
    \begin{tikzcd}
\Hom_K(D_3, Z_G) \arrow[r, equal] \arrow[d] & \Hom_K(D_3, Z_G) \arrow[d] \arrow[r, equal] & \Hom_K(D_3, Z_G)  \arrow[d] \arrow[r]& \Hom_K(D_2, Z_G) \arrow[d]  \\
\Hom_{K_v}(D_3, Z_G)  \arrow[r, equal] \arrow[d] & \Hom_{K_v}(D_3, Z_G) \arrow[r, equal] \arrow[d] & \Hom_{K_v}(D_3, Z_G) \arrow[d] \arrow[r] & \Hom_{K_v}(D_2, Z_G)  \arrow[d] \\
\Hom_{K_v}(\Gm, Z_G) \arrow[r, equal]  & \Hom_{K_v}(\Gm, Z_G)   \arrow[r, equal]  &  \Hom_{K_v}(\Gm, Z_G)  \arrow[r, equal]  & \Hom_{K_v}(\Gm, Z_G)  
\end{tikzcd} 
}
\end{equation*}

We need to define all the objects and morphisms in these diagrams. We recall that as in \cite[\S 7.3]{Kot9}, the gerb $\mc{E}^v_3(K/E^v)$ is defined to be $D_3(K_v)\rtimes_{D_3(K)} \mc{E}_3(K/E^v)/N$ where we have $D_3(K) \xrightarrow{\pi} D_3(K_v)$ and $D_3(K) \xrightarrow{\iota} \mc{E}_3(K/E^v)$ and then $N :=\{ (\pi(d)^{-1}, \iota(d)): d \in D_3(K)\}$. The gerb $\mc{E}^v_2(K/E^v)$ was defined already in Section \ref{BGloc}. The gerb $\mc{F}'_v$ is defined to be $D_3(K_v) \rtimes_{D_3(\A_K)} \mc{F}'  / N' $ where now we have $D_3(\A_K) \xrightarrow{\pi'} D_3(K_v)$ and $D_3(\A_K) \xrightarrow{\iota'} \mc{F}'$ and $N' := \{ (\pi'(d)^{-1}, \iota'(d)): d \in D_3(\A_K)\}$. 

In the diagram of gerbs, all the vertical maps except the ones in the third column are parts of the localization map and have already been defined. All the maps in the top row have been constructed already in this section and the maps of the bottom row are all the identity. This leaves the four maps involving $\mc{F}'_v$. The upper vertical one is the natural inclusion of $\mc{F}'$ into the semi-direct product. It is easy to check that the map
\begin{equation*}
    \mc{E}_3(K/E^v) \to \mc{F}'
\end{equation*}
induces a map
\begin{equation*}
    \mc{E}^v_3(K/E^v) = \mc{D}_3(K_v) \rtimes_{ D_3(\A_K)} \mc{E}_3(K/E^v) /N \to \mc{D}_3(K_v) \rtimes_{ D_3(\A_K)} \mc{F}' /N'.
\end{equation*}
and the map
\begin{equation*}
    \mc{E}_2(K/F) \to \mc{F}'
\end{equation*}
induces a map 
\begin{equation*}
    \mc{E}^v_2(K/E^v) = \mc{D}_2(K_v) \rtimes_{ D_2(\A_K)} \mc{E}_2(K/E^v) /N \to \mc{D}_3(K_v) \rtimes_{ D_3(\A_K)} \mc{F}' /N'.
\end{equation*}
In fact, the map $\mc{E}_3(K/E^v) \to \mc{F}'$ induces an isomorphism of extensions.

We can then define by composition
\begin{equation*}
    \mc{E}_{\iso}(K_v/F_v) \to \mc{F}'_v := \mc{E}_{\iso}(K_v/F_v) \to \mc{E}^v_3(K/E^v) \to \mc{F}'_v.
\end{equation*}
Now, in the gerb diagram, all  the squares except the bottom right one are known to commute by construction. We want to deduce that this square commutes up to conjugacy by proving that all the maps are canonical up to conjugacy. To prove the maps are canonical we need to show that $H^1(\Gal(K/E^v), D_3(K_v))=H^1(\Gal(K/E^v), D_2(K_v))=0$. This follows from Lemma \ref{vanishlem} and \cite[Lemma 14.4]{Kot9}. %By the construction of the localization map in \cite[\S 7]{Kot9}, the image of the $\alpha$ in $H^2(\Gal(K/E^v), D_3(K_v))$ corresponding to $\mc{E}^v_3(K/E^v)$ agrees with the image of the class $\alpha_3 \in H^2(\Gal(K/E^v), D_3(K))$ in this set. The same is true for the images of $\alpha$ and $\alpha_2$ in $H^2(\Gal(K/E^v), D_2(K_v))$. Then the desired result follows from the fact that these classes already agree in $H^2(\Gal(K/E^v), D_3(\A_K))$ corresponding to $\mc{F}'$.

The diagrams of $\ov{Y}$-sets and coefficients clearly commute, and this then implies the statement of the lemma.
\end{proof}
\begin{remark}
It is easy to check that the map 
\begin{equation*}
H^1_{\bas}(\mc{E}_3(K/F), G(K)) \to H^1_{\bas}(\mc{E}_2(K/F), G(\A_K))
\end{equation*}
can also be defined for $H^1_{\alg}$ and that the above lemma is also true. We will not need this fact in this paper.
\end{remark}
\subsection{Key global diagram: \texorpdfstring{$G_{\der}$}{Gder} simply connected case}
Let $G$ be a connected reductive group over a global field $F$ and let $K/F$ be a finite Galois extension such that $G$ is split over $K$. In this section and the following we construct the key global diagram for $G$. Namely, the commutative diagram
\begin{equation}{\label{keycommdiagram}}
\begin{tikzcd}
\bigoplus\limits_{u \in V_F} H^1_{\bas}(\mc{E}_{\iso}(K_v/F_u), G(K_v)) \arrow[d] & H^1_{\bas}(\mc{E}_2(K/F), G(\A_K))  \arrow[l, "l^F", swap] \arrow[r] \arrow[d] & H^1_{\bas}(\mc{E}_1(K/F), G(\A_K)/Z_G(K)) \arrow[d] \\
   \bigoplus\limits_{u \in V_F} X^*(Z(\widehat{G}))_{\Gal(K_v/F_u)} & \left[\bigoplus\limits_{v \in V_K} X^*(Z(\widehat{G}))\right]_{\Gal(K/F)} \arrow[r, "\Sigma"] \arrow[l, "\sim",swap]& X^*(Z(\widehat{G}))_{\Gal(K/F)} 
\end{tikzcd}    
\end{equation}
where the bottom left map is given as the composition
\begin{align*}
    \left[\bigoplus\limits_{v \in V_K} X^*(Z(\widehat{G}))\right]_{\Gal(K/F)} &= \left[X^*(Z(\widehat{G}))\otimes \Z[V_K]\right]_{\Gal(K/F)}\\
    &= \bigoplus\limits_{u \in V_F}  (X^*(Z(\widehat{G})) \otimes \Z[V_u])_{\Gal(K/F)}\\
    &\cong \bigoplus\limits_{u \in V_F} X^*(Z(\widehat{G}))_{\Gal(K_v/F_u)},
\end{align*}
where $V_u$ consists of the places of $K$ over $u$. The map $\Sigma$ is induced by the map $\bigoplus\limits_u X^*(Z(\widehat{G})) \to X^*(Z(\widehat{G}))$ summing all the coordinates together. Such a diagram is already known to exist in the case of tori (see \cite[pg 6]{Kot9}). In this section  we prove it for $G$ such that $G_{\der}$ is simply connected. In the next section we tackle the general case using $z$-extensions. In fact we construct a bit more than this diagram because we are also able to construct the middle vertical map for algebraic cocycles (not just basic). Namely, we get a map
\begin{equation*}
H^1_{\alg}(\mc{E}_2(K/F), G(\A_K)) \to \left[ \bigoplus\limits_{v \in V_K} X^*(Z(\widehat{G}))\right]_{\Gal(K/F)}.    
\end{equation*}

Suppose now that $G$ is a connected reductive group over $F$ and $G_{\der}$ is simply connected. Then we define $D$ to be the torus given by $G/ G_{\der}$. Note that by assumption, $\widehat{D}$ and $Z(\widehat{G})$ are canonically isomorphic. We now consider the following diagram:
\begin{equation}
\begin{tikzcd}
\bigoplus\limits_{u \in V_F} H^1_{\bas}(\mc{E}_{\iso}(K_v/F_u), G(K_v)) \arrow[d] & H^1_{\bas}(\mc{E}_2(K/F), G(\A_K))  \arrow[l, "l^F", swap] \arrow[r] \arrow[d] & H^1_{\bas}(\mc{E}_1(K/F), G(\A_K)/Z_G(K)) \arrow[d] \\
 \bigoplus\limits_{u \in V_F} H^1_{\bas}(\mc{E}_{\iso}(K_v/F_u), D(K_v)) \arrow[d, "\simeq"]  & H^1_{\alg}(\mc{E}_2(K/F), D(\A_K)) \arrow[d, "\simeq"] \arrow[r] \arrow[l, "\sim"]  & H^1_{\alg}(\mc{E}_1(K/F), D(\A_K)/D(K)) \arrow[d, "\simeq"] \\
   \bigoplus\limits_{u \in V_F} X^*(Z(\widehat{G}))_{\Gal(K_v/F_u)} & \left[\bigoplus\limits_{v \in V_K} X^*(Z(\widehat{G}))\right]_{\Gal(K/F)} \arrow[r, "\Sigma"] \arrow[l, "\sim"]& X^*(Z(\widehat{G}))_{\Gal(K/F)} 
\end{tikzcd}    
\end{equation}
The top and middle left arrows are the total localization maps and the top and middle right arrows are the ones we constructed in the previous subsection. The vertical arrows from the first to second row are all induced via functoriality from the map $G \to D$. The vertical bottom left arrow is the sum of local Kottwitz maps $\kappa_G$ and the remaining two vertical arrows are as in the diagram \cite[pg 6]{Kot9}. The commutativity of the lower right square is as in loc cit. and the commutativity of the bottom left square is \cite[Lemma 7.4]{Kot9}. All the maps comprising the upper and middle right arrows are given by compositions of the $\Psi$ and $\Phi$ maps while the upper and middle vertical arrows come from changing coefficients. Hence this square commutes by an easy application of Lemmas \ref{PhiPsilem} and \ref{Philem}. To show the upper left square commutes, we need only show it commutes for a fixed place $u$. Since the localization map is a composition of $\Psi$ and $\Phi$ and restriction maps, this again follows from Lemma \ref{PhiPsilem} and Lemma \ref{Philem}.

\subsection{Key global diagram: general case}
In this section we use the theory of $z$-extensions to construct the key global diagram for connected reductive $G$ over $F$ and split by $K$. Our argument will be analogous to similar arguments for $H^1_{\alg}(\mc{E}_2(K/F), G(K))$ appearing in \cite{Kot9}.

We need some preliminaries. Unfortunately, we need to prove each lemma for the gerb $\mc{E}_2(K/F)$ with algebraic cocycles valued in $G(\A_K)$ as well as the gerb $\mc{E}_1(K/F)$ with basic cocycles valued in $G(\A_K)/Z_G(K)$.

\begin{lemma}{\label{fiberact}}
Let
\begin{equation*}
    1 \to Z \xrightarrow{i} G' \xrightarrow{p} G \to 1
\end{equation*}
be a central extension of linear algebraic groups over $F$.
\begin{enumerate}
\item The group $H^1_{\alg}(\mc{E}_2(K/F), Z(\A_K))$ acts on the fibers of 
\begin{equation*}
    p: H^1_{\alg}(\mc{E}_2(K/F), G'(\A_K)) \to H^1_{\alg}(\mc{E}_2(K/F), G(\A_K)).
\end{equation*}
When $p: G'(\A_K) \to G(\A_K)$ is surjective, these actions are transitive.
\item The group $H^1_{\bas}(\mc{E}_1(K/F), Z(\A_K)/Z(K))$ acts on the fibers of 
\begin{equation*}
    p: H^1_{\bas}(\mc{E}_1(K/F), G'(\A_K)/Z_{G'}(K)) \to H^1_{\alg}(\mc{E}_2(K/F), G(\A_K)/Z_G(K)).
\end{equation*}
When $Z$ is the Weil restriction of a split $K$-torus, these actions are transitive.
\end{enumerate}
\end{lemma}
\begin{proof}
We first construct the action of $H^1_{\alg}(\mc{E}_2(K/F), Z(\A_K))$ on the fibers of $H^1_{\alg}(\mc{E}_2(K/F), G'(\A_K))$. Pick $b \in H^1_{\alg}(\mc{E}_2(K/F), G'(\A_K))$ and $c \in H^1_{\alg}(\mc{E}_2(K/F), Z(\A_K))$ and pick cocycle representatives $(\nu, x)$ of $b$ and $(\mu, y)$ of $c$. Then $(\nu\mu, xy)$ is another cocycle of $H^1_{\alg}(\mc{E}_2(K/F), G'(\A_K))$ and this cocycle agrees with $(\nu, x)$ when we project to $Z^1(\mc{E}_2(K/F), G(\A_K))$ via $p$. If we replace $(\mu, y)$ with $z \cdot(\mu, y)$ and $(\nu, x)$ with $m \cdot (\nu, x)$ for $z \in Z(\A_K)$ and $m \in G'(\A_K)$ then the corresponding product gives $mz \cdot (\nu\mu, xy)$ which is clearly in the same cohomology class as $(\nu\mu, xy)$. An analogous argument proves the first part of $(2)$.

We now show the second part of $(1)$: that when $p: G'(\A_K) \to G(\A_K)$ is surjective, the action is transitive. Pick $b, b' \in H^1_{\alg}(\mc{E}_2(K/F), G'(\A_K))$ in the same fiber under $p$ and fix cocycle representatives $(\nu, x)$ and $(\nu', x')$.  We can assume that the cocycles become equal in $Z^1_{\alg}(\mc{E}_2(K/F), G(\A_K))$. Indeed, since the cocycles are in the same cohomology class in $H^1_{\alg}(\mc{E}_2(\A_K),G(\A_K))$, we can pick some $m \in G(\A_K)$ such that $m \cdot (\nu, x) = (\nu', x')$. We then use surjectivity to pick $m' \in G'(\A_K)$ such that $p(m')=m$ and observe that $m' \cdot (\nu, x)$ and $(\nu', x')$ agree in $Z^1_{\alg}(\mc{E}_2(K/F), G(\A_K))$. It is then easy to check that  $(\nu, x)$ and $(\nu', x')$ differ by a unique cocycle $(\mu, y) \in Z^1_{\alg}(\mc{E}_2(K/F), Z(\A_K))$.

We now show the second part of $(2)$. We first remark that since $H^1(\Gal(\ov{K}/K), Z(\ov{\A_K}))$ vanishes by assumption, we have a surjection $G'(\A_K) \twoheadrightarrow G(\A_K)$ and hence a surjection $G'(\A_K)/Z_{G'}(K) \twoheadrightarrow G(\A_K)/Z_G(K)$. We claim that 
\begin{equation*}
Z(\A_K)/Z(K) \xrightarrow{i} G'(\A_K)/Z_{G'}(K) \xrightarrow{p} G(\A_K)/Z_G(K) \to 1
\end{equation*}
is exact. We have already shown that $p$ is surjective so it remains to show that $\im(i) = \ker(p)$. It is clear that $\im(i) \subset \ker(p)$. To show the other inclusion, pick $g' \in \ker(p)$. Lift $g'$ to an element $\tilde{g'} \in G'(\A_K)$. Then we must have that $p(\tilde{g'}) \in Z_G(K)$. Since $H^1(\Gal(\ov{K}/K), Z)=0$, we have that
\begin{equation*}
    1 \to Z(K) \to Z_{G'}(K) \to Z_G(K) \to 1
\end{equation*}
is exact. Hence we can pick $z \in Z_{G'}(K)$ such that $p(\tilde{g'})=p(z)$. Now, $p(z^{-1}\tilde{g'})=1$ and so there is some $z' \in Z(\A_K)$ so that $i(z')=z^{-1}\tilde{g'}$. It follows that the projection of $z'$ to $Z(\A_K)/Z(K)$ maps to the projection of $z^{-1}\tilde{g'}$ to $G'(\A_K)/Z_{G'}(K)$ which is precisely $g'$. Hence, $g' \in \im(i)$.

At this point, we can prove transitivity using the same argument as in $(1)$. Namely, we pick cocycles $(\nu, x), (\nu', x') \in Z^1_{\bas}(\mc{E}_1(K/F), G'(\A_K)/Z_{G'}(K))$ which we can assume map to the same cocycle in $Z^1_{\bas}(\mc{E}_1(K/F), G(\A_K)/Z_G(K))$ since we have a surjection $p: G'(\A_K)/Z_{G'}(K) \to G(\A_K)/Z_G(K)$. Then by middle-exactness, it follows that $(\nu, x)$ and $(\nu', x')$ differ by a cocycle of $Z^1_{\bas}(\mc{E}_1(K/F), Z(\A_K)/Z(K))$ as desired.
\end{proof}
\begin{remark}
We really do need the stronger assumption to prove part $(2)$. Indeed, consider
\begin{equation*}
    1 \to \mu_8 \to \Gm \xrightarrow{p} \Gm \to 1,
\end{equation*}
where $p$ is the $8$th power map and let $K=\Q(\sqrt{7})$. Then by the counter-example to the Grunwald-Wang theorem, there is an element $g \in \Gm(\A_K)$ such that $p(g)=16$ but no such element can be contained in $\Gm(K)$. Hence, 
\begin{equation*}
    \mu_8(\A_K)/\mu_8(K) \to \Gm(\A_K)/\Gm(K) \xrightarrow{p} \Gm(\A_k)/\Gm(K)
\end{equation*}
is not middle-exact.
\end{remark}
\begin{lemma}{\label{surjtori}}
Let
\begin{equation*}
    1 \to T_1 \to T_2 \to T_3 \to 1
\end{equation*}
be a short exact sequence of $F$-tori that are split by $K$. Then the natural maps
\begin{equation*}
    H^1_{\alg}(\mc{E}_2(K/F), T_2(\A_K)) \to H^1_{\alg}(\mc{E}_2(K/F), T_3(\A_K))
\end{equation*}
and
\begin{equation*}
    H^1_{\bas}(\mc{E}_1(K/F), T_2(\A_K)/T_2(K)) \to H^1_{\bas}(\mc{E}_1(K/F), T_3(\A_K)/T_3(K))
\end{equation*}
are surjective.
\end{lemma}
\begin{proof}
The surjectivity of the first map follows from the isomorphism for such tori given in \cite{Kot9}:
\begin{equation*}
    H^1_{\alg}(\mc{E}_2(K/F), T(\A_K)) \cong (X_*(T) \otimes \Z[V_K])_{\Gal(K/F)}
\end{equation*}
and the fact that the functor $T \mapsto (X_*(T) \otimes \Z[V_K])_{\Gal(K/F)}$ is right exact.

The surjectivity of the second map follows for the same reason from the isomorphism of Kottwitz:
\begin{equation*}
    H^1_{\bas}(\mc{E}_1(K/F), T(\A_K)/T(K)) \cong X_*(T)_{\Gal(K/F)}
\end{equation*}
\end{proof}
\begin{lemma}{\label{middleexact}}
Let 
\begin{equation*}
    1 \to N \xrightarrow{i} G' \xrightarrow{p} G \to 1 
\end{equation*}
be a short exact sequence of linear algebraic groups over $F$.
\begin{enumerate}
    \item If $p: G'(\A_K) \to G(\A_K)$ is surjective, then 
\begin{equation*}
    H^1_{\alg}(\mc{E}_2(K/F), N(\A_K)) \xrightarrow{i} H^1_{\alg}(\mc{E}_2(K/F), G'(\A_K)) \xrightarrow{p} H^1_{\alg}(\mc{E}_2(K/F), G(\A_K))
\end{equation*}
is an exact sequence of pointed sets.
\item Now suppose that $i(Z_N) \subset Z_{G'}$ and that $N(\A_K)/Z_N(K) \to G'(\A_K)/Z_{G'}(K) \to G(\A_K)/Z_G(K) \to 1$ is exact. Then
\begin{equation*}
    H^1_{\alg}(\mc{E}_2(K/F), N(\A_K)/Z_N(K)) \xrightarrow{i} H^1_{\alg}(\mc{E}_2(K/F), G'(\A_K)/Z_{G'}(K)) \xrightarrow{p} H^1_{\alg}(\mc{E}_2(K/F), G(\A_K)/Z_G(K))
\end{equation*}
is an exact sequence of pointed sets.
\end{enumerate}
\end{lemma}
\begin{proof}
We prove $(1)$ first. If $b \in H^1_{\alg}(\mc{E}_2(K_F), G'(\A_K))$ is in the image of $i$, then $p(b)$ is trivial since $N(\A_K) = \ker(G'(\A_K) \to G(\A_K))$. On the other hand, if $p(b)$ is trivial, then we can pick a cocycle representative $(\nu, x)$ and using the surjectivity of $p$, may assume that the image of $x$ lies in $\ker(p)=i(N(\A_K))$. Then $(\nu, x)$ gives an element of $Z^1_{\alg}(\mc{E}_2(K/F), N(\A_K))$ and hence $b$ lies in the image of $i$.

To prove $(2)$, use the same argument noting that we have a stronger assumption to preclude the possible failure of middle exactness.
\end{proof}
The following proposition is an analogue of \cite[Prop. 2.8]{Kot9}.
\begin{proposition}{\label{surjprop}}
Let
\begin{equation*}
    1 \to Z \xrightarrow{i} G' \xrightarrow{p} G \to 1
\end{equation*}
be a short exact sequence of linear algebraic $F$-groups such that $Z$ is a torus that splits over $K$ and is central in $G'$. Then the natural maps 
\begin{equation*}
    p: H^1_{\alg}(\mc{E}_2(K/F), G'(\A_K)) \to H^1_{\alg}(\mc{E}_2(K/F), G(\A_K))
\end{equation*}
and 
\begin{equation*}
    p: H^1_{\bas}(\mc{E}_1(K/F), G'(\A_K)/Z_{G'}(K)) \to H^1_{\bas}(\mc{E}_1(K/F), G(\A_K)/Z_G(K))
\end{equation*}
are surjective. Moreover, they induce bijections between $H^1_{\alg}(\mc{E}_2(K/F), G(\A_K))$ and the quotient of $H^1_{\alg}(\mc{E}_2(K/F), G'(\A_K))$ by the action of $H^1_{\alg}(\mc{E}_2(K/F), Z(\A_K))$ as well as between $H^1_{\bas}(\mc{E}_1(K/F), G(\A_K)/Z_G(K))$ and the quotient of $H^1_{\bas}(\mc{E}_1(K/F), G'(\A_K)/Z_{G'}(K))$ by $H^1_{\bas}(\mc{E}_1(K/F), Z(\A_K)/Z(K))$.
\end{proposition}
\begin{proof}
We first give the proof for the statements involving $\mc{E}_2(K/F)$. To begin, pick $b \in H^1_{\alg}(\mc{E}_2(K/F), G(\A_K))$ and let $(\nu, x)$ be a cocycle representative of $b$. Since $D_2$ is $K$-split, we have that the image of $\nu$ in $G_K$ is a split $K$-torus $T \subset G$. Then pulling back along $p$ we have a short exact sequence of $K$-tori
\begin{equation*}
    1 \to Z \to T' \to T \to 1.
\end{equation*}
Since $Z$ and $T$ are $K$-split, this implies that $T'$ is as well. Hence the short exact sequence splits and so we have an exact sequence
\begin{equation*}
    1 \to \Hom_K(D_2, Z) \to \Hom_K(D_2, T') \to \Hom_K(D_2, T) \to 1.
\end{equation*}
Let $\nu' \in \Hom_K(D_2, T')$ be a lift of $\nu$.

We now claim that $p: G'(\A_K) \to G(\A_K)$ is surjective. Indeed this would be implied by the vanishing of $H^1(\Gal(\ov{K}/K), Z(\A_K))$. This set is isomorphic to $\bigoplus\limits_{v \in V_K} H^1(\Gal(\ov{K_v}/K_v), Z(K_v))$ and thus vanishes by Hilbert's theorem $90$.

Now, for each $\sigma \in \Gal(K/F)$, choose a lift $\dot{\sigma} \in \mc{E}_2(K/F)$. For each $\dot{\sigma} \in \mc{E}_2(K/F)$, we choose an element $x'_{\dot{\sigma}} \in G'(\A_K)$ lifting $x_{\dot{\sigma}}$. We claim that $\Int(x'_{\dot{\sigma}}) \circ \sigma(\nu')$ is independent of our choice of $\dot{\sigma}$ and $x'_{\dot{\sigma}}$. Indeed, if we pick a different lift $\dot{\sigma}'$, then we have $\dot{\sigma'}=\dot{\sigma}d$ for some $d \in D_2(\A_K)$. Then a lift of $x_{\dot{\sigma}'}$ is of the form $zx'_{\dot{\sigma}}\sigma(\nu')(d)$ for $z \in Z(\A_K)$, which implies our claim.

Since by definition $\Int(x_w) \circ \sigma(\nu) = \nu$ for each $w \in \mc{E}_2(K/F)$ projecting to $\sigma$, it follows that $\nu'$ and $\Int(x'_{\dot{\sigma}}) \circ \sigma(\nu')$ are two lifts of $\nu$. A priori, we have $\Int(x'_{\dot{\sigma}}) \circ \sigma(\nu') \in \prod\limits_v \Hom_{K_v}(D_2, T')$. But we claim that in fact, this element lies in $\Hom_K(D_2,T')$. If not, there would exist places $v_1, v_2$ and for $i=1,2$, elements $x_{\dot{\sigma}, i} \in G(K_{v_i})$ with lifts $x'_{\dot{\sigma}, i} \in G'(K_{v_i})$ such that $\Int(x_{\dot{\sigma},i}) \circ \sigma(\nu)= \nu$ but $\Int(x'_{\dot{\sigma},1}) \circ \sigma(\nu')=\nu'_1$ and $\Int(x'_{\dot{\sigma},2}) \circ \sigma(\nu')=\nu'_2$ are not equal to the images in $\Hom_{K_{v_1}}(D_2, T')$ and $\Hom_{K_{v_2}}(D_2, T')$ of the same element of $\Hom_K(D_2, T')$. We claim that we can choose $x'_1, x'_2 \in G'(K)$ such that $\Int(x'_i)(\sigma(\nu'))=\nu'_i$. Indeed, $\sigma(\nu')$ and $\nu'_i$ lie in $\Hom_{K_{v_i}}(D_2, S)$ for some $K$-split maximal torus $S$ of $G'$ and then the proof of \cite[Lemma 1.1.3.(a)]{Kot3} implies that for each $i=1,2$, we have $\sigma(\nu')$ and $\nu'_i$ are conjugate by some Weyl group element, and then our claim follows from the fact that the Weyl group of $S$ is a constant group scheme over $K$. Now consider ${x'}^{-1}_1x'_2 \in G'(K)$ and its projection to $x^{-1}_1x_2 \in G(K)$. We have that $x^{-1}_1x_2$ centralizes $T$ by assumption, and therefore that ${x'}^{-1}_1x'_2$ normalizes but does not centralize $T'$. But now consider the action of ${x'}^{-1}_1x'_2$ on $T'$ by conjugation. It acts trivially $Z$ and the induced action on $T$ is also trivial. We have already observed that we have a splitting $T' \cong Z \oplus T$. This makes it clear that the action of ${x'}^{-1}_1x'_2$ on $T'$ is by a unipotent matrix. But $N_{G'}(T')/Z_{G'}(T')$ is finite so some power of ${x'}^{-1}_1x'_2$ acts trivially on $T'$. Since ${x'}^{-1}_1x'_2$ is unipotent, this implies that $\Int({x'}^{-1}_1x'_2) \in Z_{G'}(T')$, contrary to assumption. Hence we have proven that $\Int(x'_{\dot{\sigma}}) \circ \sigma(\nu') \in \Hom_K(D_2, T')$.

Hence, there exists a $\lambda_{\sigma} \in \Hom_K(D_2, Z)$ such that $\Int(x'_{\dot{\sigma}}) \circ \sigma(\nu') = \nu' + \lambda_{\sigma}$. We claim that $\sigma \mapsto \lambda_{\sigma} \in Z^1(\Gal(K/F), \Hom_K(D_2, Z))$. Indeed, since for $\sigma_1, \sigma_2 \in \Gal(K/F)$ we have that $x'_{\dot{\sigma_1}} \sigma_1(x'_{\dot{\sigma_2}})$ is a lift of $x_{\sigma_1\sigma_2}$, we have
\begin{align*}
\nu' + \lambda_{\sigma_1\sigma_2} &= \Int(x'_{\dot{\sigma_1}} \sigma_1(x'_{\dot{\sigma_2}})) \circ \sigma_1\sigma_2(\nu')\\
&= \Int(x'_{\dot{\sigma_1}}) \circ \sigma_1(\Int(x'_{\dot{\sigma_2}}) \circ \sigma_2(\nu')) \\
&= \Int(x'_{\dot{\sigma_1}}) \circ ( \sigma_1(\nu') + \sigma_1(\lambda_{\sigma_2}))\\
&= \nu' + \lambda_{\sigma_1} + \sigma_1(\lambda_{\sigma_2}).
\end{align*}

In order to trivialize this cohomology class, we need to enlarge $G'$. Define $Z'' := \Res_{K/F} Z_K$ and push out $G'$ along $Z \hookrightarrow Z''$ to get a diagram
\begin{equation*}
    \begin{tikzcd}
    1 \arrow[r] & Z \arrow[r, "i"] \arrow[d] & G' \arrow[r, "p"] \arrow[d] & G \arrow[r] \arrow[d, equal] & 1\\
    1 \arrow[r] & Z'' \arrow[r, "j"] & G'' \arrow[r, "q"] & G \arrow[r] & 1.
    \end{tikzcd}
\end{equation*}
Then the $\Gal(K/F)$-module $\Hom_k(D_2, Z'') \cong X^*(D_2) \otimes X_*(Z'')$ is coinduced from the $\Z$-module $X^*(D_2) \otimes X_*(Z)$ and hence $H^1(\Gal(K/F), \Hom_K(D_2, Z''))$ vanishes. It follows that there exists a $\mu \in \Hom_K(D_2, Z'')$ such that $\lambda_{\sigma} = \sigma(\mu) - \mu$. Then $\nu'' := \nu' -\mu$ is a lift of $\nu$ to $\nu''$ such that 
\begin{equation*}
\Int(x'_{\dot{\sigma}})\circ \sigma(\nu'')=\nu'',    
\end{equation*}
for all $\sigma \in \Gal(K/F)$.

Now define a $1$-cochain of $\mc{E}_2(K/F)$ valued in $G''(\A_K)$ by $x''_{d \dot{\sigma}} := \nu''(d)x'_{\dot{\sigma}}$. One can easily check that $(w_1, w_2) \mapsto z_{w_1, w_2}$ defined by $x''_{w_1w_2}= z_{w_1, w_2}x''_{w_1} w_1(x''_{w_2})$ is a $2$-cocycle of $\mc{E}_2(K/F)$ valued in $Z''(\A_K)$. We see that changing $w_1$ and $w_2$ by elements of $D_2(\A_K)$ does not change the value of $z_{w_1, w_2}$ and hence that $z_{w_1, w_2}$ is inflated from a $2$-cocycle of $\Gal(K/F)$. But
\begin{equation*}
    H^2(\Gal(K/F), Z''(\A_K))=\bigoplus\limits_{u \in V_F} H^2(\Gal(K_v/F_u), Z''(K_v))
\end{equation*}
which is trivial by Shapiro's lemma. Hence there is a function $y: \Gal(K/F) \to Z''(\A_K)$ such that $(\sigma_1, \sigma_2) \mapsto y^{-1}_{\sigma_1\sigma_2}\sigma_1(y_{\sigma_2})y_{\sigma_1}$ is a $2$-coboundary equal to $z$. We can then pullback $y$ to $\mc{E}_2(K/F)$ and define $x'''_w := y_w x''_w$. It is easy to check that $(\nu'', x''')$ is an algebraic $1$-cocycle of $\mc{E}_2(K/F)$ valued in $G''(\A_K)$ and a lift of $(\nu, x)$. We have now shown that $H^1_{\alg}(\mc{E}_2(K/F), G''(\A_K)) \to H^1_{\alg}(\mc{E}_2(K/F), G(\A_K))$ is surjective.

Our goal is to use this to show the surjectivity: $H^1_{\alg}(\mc{E}_2(K/F), G'(\A_K)) \to H^1_{\alg}(\mc{E}_2(K/F), G(\A_K))$. We first define $C := Z''/Z = G''/G$. The commutative diagram
\begin{equation*}
    \begin{tikzcd}
    1 \arrow[r] & Z \arrow[r, "i"] \arrow[d] & G' \arrow[r, "p"] \arrow[d] & G \arrow[r] \arrow[d, equal] & 1\\
    1 \arrow[r] & Z'' \arrow[r, "j"] \arrow[d] & G'' \arrow[r, "q"] \arrow[d] & G \arrow[r] & 1.\\
    & C \arrow[r, equal] & C & &
    \end{tikzcd}
\end{equation*}
induces a diagram
\begin{equation*}
    \begin{tikzcd}
    H^1_{\alg}(\mc{E}_2(K/F), Z(\A_K)) \arrow[r, "i"] \arrow[d] & H^1_{\alg}(\mc{E}_2(K/F), G'(\A_K)) \arrow[r, "p"] \arrow[d] & H^1_{\alg}(\mc{E}_2(K/F), G(\A_K)) \arrow[d, equal] \\
    H^1_{\alg}(\mc{E}_2(K/F), Z''(\A_K)) \arrow[r, "j"] \arrow[d] & H^1_{\alg}(\mc{E}_2(K/F), G''(\A_K)) \arrow[r, "q"] \arrow[d] & H^1_{\alg}(\mc{E}_2(K/F), G(\A_K)) \\
    H^1_{\alg}(\mc{E}_2(K/F), C(\A_K)) \arrow[r, equal] & H^1_{\alg}(\mc{E}_2(K/F), C(\A_K)). & 
    \end{tikzcd}
\end{equation*}
We claim the sequence
\begin{equation*}
    H^1_{\alg}(\mc{E}_2(K/F), G'(\A_K)) \to H^1_{\alg}(\mc{E}_2(K/F), G''(\A_K)) \to H^1_{\alg}(\mc{E}_2(K/F), C(\A_K))
\end{equation*}
is exact. Indeed, by Lemma \ref{middleexact}, it suffices to show that $G''(\A_K) \to C(\A_K)$ is surjective. This follows from the exactness of 
\begin{equation*}
    1 \to Z \to Z'' \to C  \to 1
\end{equation*}
and Hilbert's Theorem 90 applied to $Z$.

So far we have shown that starting with a $b \in H^1_{\alg}(\mc{E}_2(K/F), G(\A_K))$, we can find a $b'' \in H^1_{\alg}(\mc{E}_2(K/F), G''(\A_K))$ such that $q(b)=b''$. Now by Lemma \ref{surjtori}, the map $H^1_{\alg}(\mc{E}_2(K/F), Z''(\A_K)) \to H^1_{\alg}(\mc{E}_2(K/F), C(\A_K))$ is surjective and so we can find $b_2 \in H^1_{\alg}(\mc{E}_2(K/F), Z''(\A_K))$ such that the projections of $b_2$ and $b''$ to $H^1_{\alg}(\mc{E}_2(K/F), C(\A_K))$ are equal. Then $b^{-1}_2b'' \in H^1_{\alg}(\mc{E}_2(K/F), G''(\A_K))$ and projects to the trivial element of $H^1_{\alg}(\mc{E}_2(K/F), C(\A_K))$. Hence by exactness, there exists $b' \in H^1_{\alg}(\mc{E}_2(K/F), G'(\A_K))$ so that the image of $b'$ in $H^1_{\alg}(\mc{E}_2(K/F), G''(\A_K))$ equals $b^{-1}_2b''$. By the commutativity of the diagram, we have $p(b')=b$ since $q(b^{-1}_2b'')=b$ by Lemma \ref{fiberact}. We have now proven the desired surjectivity.

To prove the last statement, we apply Lemma \ref{fiberact} using the fact that $p: G'(\A_K) \to G(\A_K)$ is surjective by Hilbert's Theorem $90$ applied to $Z$.

We have now proven the statements for $\mc{E}_2(K/F)$. The argument for $\mc{E}_1(K/F)$ is highly analogous but we comment on the differences.
\begin{enumerate}
    \item In the second paragraph when we show $p: G'(\A_K) \to G(\A_K)$ is surjective, we instead need to show that \begin{equation*}
        Z(\A_K)/Z(K) \to G'(\A_K)/Z_{G'}(K) \to G(\A_K)/Z_G(K) \to 1
    \end{equation*}
    is exact. The surjectivity follows from the surjectivity of $G'(\A_K)  \to G(\A_K)$ and the middle-exactness follows from the vanishing of $H^1(\Gal(\ov{K}/K), Z(\ov{K}))$ as in Lemma \ref{fiberact}.
    \item In the third paragraph we use the exactness of the above sequence to prove that all lifts of $x_{\sigma'}$ are of the form $z\nu(d)x'_{\dot{\sigma}}$ for $z \in Z(\A_K)/Z(K)$ and $d \in D_1(\A_K)/D_1(K)$.
    \item Most of the fourth paragraph is unnecessary since we only work with basic cocycles.
    \item In the $7$th paragraph, we define a $1$-cochain of $\mc{E}_1(K/F)$ valued in $G''(\A_K)/Z_{G''}(K)$ by $x''_{d\dot{\sigma}} := \nu''(d)x'_{\dot{\sigma}}$ and then need to define a $2$-cocycle $(w_1, w_2) \mapsto z_{w_1, w_2} \in Z''(\A_K)/Z(K)$ by $x''_{w_1, w_2} = z_{w_1, w_2} x''_{w_1}w_1(x''_{w_2})$. For this to make sense we need the exactness of 
    \begin{equation*}
        Z''(\A_K)/Z''(K) \to G''(\A_K)/Z_{G''}(K) \to G(\A_K)/Z_G(K).
    \end{equation*}
    The follows as in Lemma \ref{fiberact} from the vanishing of $H^1(\Gal(\ov{K}/K), Z''(\ov{K}))$.
    \item In paragraph $7$ we also need the vanishing of $H^2(\Gal(K/F), Z''(\A_K)/Z''(K))$. This follows from the exact sequence
    \begin{equation*}
        1 \to Z''(K) \to Z''(\A_K) \to Z''(\A_K)/Z(K) \to 1
    \end{equation*}
    and the fact that the other groups in the sequence have vanishing cohomology.
    \item In the $8$th paragraph, to get the desired diagram on cohomology, we need to show we have a diagram
    \begin{equation*}
    \begin{tikzcd}
    Z(\A_K)/Z(K) \arrow[r, "i"] \arrow[d] & G'(\A_K)/Z_{G'}(K) \arrow[r, "p"] \arrow[d] & G(\A_K)/Z_G(K) \arrow[d, equal] \\
    Z''(\A_K)/Z''(K) \arrow[r, "j"] \arrow[d] & G''(\A_K)/Z_{G''}(K) \arrow[r, "q"] \arrow[d] & G(\A_K)/Z_G(K) \\
    C(\A_K)/C(K) \arrow[r, equal] & C(\A_K)/C(K). & 
    \end{tikzcd}
\end{equation*}
All the maps are already known to exist except $G'(\A_K)/Z_{G'}(K) \to G''(\A_K)/Z_{G''}(K)$. This one exists because by construction the map $G' \to G''$ induces a map $Z_{G'} \to Z_{G''}$.
\item In the $8$th paragraph, in order to apply Lemma \ref{middleexact} to prove that
\begin{equation*}
    H^1_{\bas}(\mc{E}_1(K/F), G'(\A_K)/Z_{G'}(K)) \to H^1_{\bas}(\mc{E}_1(K/F), G''(\A_K)/Z_{G''}(K)) \to H^1_{\bas}(\mc{E}_1(K/F), C(\A_K)/C(K)
\end{equation*}
is exact, we need to show that 
\begin{equation*}
    G'(\A_K)/Z_{G'}(K) \to G''(\A_K)/Z_{G''}(K) \to C(\A_K)/C(K) \to 1
\end{equation*}
is exact. Surjectivity follows from the surjectivity of $G''(\A_K) \to C(\A_K)$. As in Lemma \ref{fiberact}, we can show middle-exactness by proving that $Z_{G''}(K)$ surjects onto $C(K)$. This follows because $Z''(K) \subset Z_{G''}(K)$ and $Z''(K)$ surjects onto $C(K)$ because $H^1(\Gal(\ov{K}/K), Z(\ov{K})) =1$.
\item In the final paragraph, we apply Lemma \ref{fiberact} using that $H^1(\Gal(\ov{K}/K), Z(\ov{K}))$ vanishes.
\end{enumerate}
\end{proof}
We are now ready to construct the diagram in the previous section for connected reductive $G$ over $F$ and split by a finite Galois extension $K$.  Choose a $z$-extension $G'$ of $G$. In particular we have a short exact sequence of $F$-groups
\begin{equation*}
    1 \to Z \to G' \to G \to 1
\end{equation*}
such that $Z$ is central in $G'$, the torus $Z$ is obtained by Weil-restriction from a split $K$-torus, and $G'_{\der}$ is simply connected.

Now, by Proposition \ref{surjprop}, we have surjections
\begin{equation*}
    p: H^1_{\alg}(\mc{E}_2(K/F), G'(\A_K)) \to H^1_{\alg}(\mc{E}_2(K/F), G(\A_K))
\end{equation*}
and 
\begin{equation*}
    p: H^1_{\bas}(\mc{E}_1(K/F), G'(\A_K)/Z_{G'}(K)) \to H^1_{\bas}(\mc{E}_1(K/F), G(\A_K)/Z_G(K))
\end{equation*}
that induce bijections:
\begin{equation*}
    H^1_{\alg}(\mc{E}_2(K/F), G(\A_K)) \cong H^1_{\alg}(\mc{E}_2(K/F), G'(\A_K)) / H^1_{\alg}(\mc{E}_2(K/F), Z(\A_K)),
\end{equation*}
and
\begin{equation*}
    H^1_{\bas}(\mc{E}_1(K/F), G(\A_K)/Z_G(K)) \cong  H^1_{\bas}(\mc{E}_1(K/F), G'(\A_K)/Z_{G'}(K)) / H^1_{\bas}(\mc{E}_1(K/F), Z(\A_K)/Z(K)).
\end{equation*}
We claim that the natural maps
\begin{equation*}
    p: \left[ \bigoplus\limits_{v \in V_K} X^*(Z(\widehat{G'}))\right]_{\Gal(K/F)} \to \left[ \bigoplus\limits_{v \in V_K} X^*(Z(\widehat{G}))\right]_{\Gal(K/F)}
\end{equation*}
and
\begin{equation*}
    p: X^*(Z(\widehat{G'}))_{\Gal(K/F)} \to X^*(Z(\widehat{G}))_{\Gal(K/F)}
\end{equation*}
induce bijections 
\begin{equation*}
    \left[ \bigoplus\limits_{v \in V_K} X^*(Z(\widehat{G}))\right]_{\Gal(K/F)} \cong \left[ \bigoplus\limits_{v \in V_K} X^*(Z(\widehat{G'}))\right]_{\Gal(K/F)} / \left[ \bigoplus\limits_{v \in V_K} X^*(\widehat{Z})\right]_{\Gal(K/F)},
\end{equation*}
and 
\begin{equation*}
    X^*(Z(\widehat{G}))_{\Gal(K/F)} \cong X^*(Z(\widehat{G'}))_{\Gal(K/F)} / X^*(\widehat{Z})_{\Gal(K/F)}.
\end{equation*}
Indeed, this follows from the exact sequence
\begin{equation*}
    0 \to \Lambda_Z \to \Lambda_{G'} \to \Lambda_G \to 0
\end{equation*}
where $\Lambda_G$ is Borovoi's fundamental group ( recall $\Lambda_G \cong X^*(Z(\widehat{G}))$ \,) and the fact that tensor product and co-invariants functors are right-exact.

Finally, we remark that by construction, the maps
\begin{equation*}
    H^1_{\alg}(\mc{E}_2(K/F), G'(\A_K)) \to \left[ \bigoplus\limits_{v \in V_K} X^*(Z(\widehat{G'}))\right]_{\Gal(K/F)}
\end{equation*}
and
\begin{equation*}
    H^1_{\bas}(\mc{E}_1(K/F), G'(\A_K)/Z_{G'}(K)) \to X^*(Z(\widehat{G'}))
\end{equation*}
constructed in the previous section are equivariant with respect to the actions of 
\begin{equation*}
    H^1_{\alg}(\mc{E}_2(K/F), Z(\A_K))  \cong \left[ \bigoplus\limits_{v \in V_K} X^*(\widehat{Z})\right]_{\Gal(K/F)},
\end{equation*}
and
\begin{equation*}
      H^1_{\bas}(\mc{E}_1(K/F), Z(\A_K)/Z(K))  \cong X^*(\widehat{Z})_{\Gal(K/F)},
\end{equation*}
respectively.

Together, these facts give us unique maps
\begin{equation*}
    H^1_{\alg}(\mc{E}_2(K/F), G(\A_K)) \to \left[ \bigoplus\limits_{v \in V_K} X^*(Z(\widehat{G}))\right]_{\Gal(K/F)}
\end{equation*}
and
\begin{equation*}
    H^1_{\bas}(\mc{E}_1(K/F), G(\A_K)/Z_{G}(K)) \to X^*(Z(\widehat{G}))
\end{equation*}
making the following diagrams commute:
\begin{equation*}
 \begin{tikzcd}
 H^1_{\alg}(\mc{E}_2(K/F) , G'(\A_K)) \arrow[r, "p"] \arrow[d] & H^1_{\alg}(\mc{E}_2(K/F), G(\A_K)) \arrow[d]\\
 \left[ \bigoplus\limits_{v \in V_K} X^*(Z(\widehat{G'}))\right]_{\Gal(K/F)} \arrow[r, "p"] & \left[ \bigoplus\limits_{v \in V_K} X^*(Z(\widehat{G}))\right]_{\Gal(K/F)}
 \end{tikzcd}   
\end{equation*}
and 
\begin{equation*}
 \begin{tikzcd}
 H^1_{\bas}(\mc{E}_1(K/F) , G'(\A_K)/Z_{G'}(K)) \arrow[r, "p"] \arrow[d] & H^1_{\bas}(\mc{E}_1(K/F), G(\A_K)/Z_G(K)) \arrow[d]\\
 X^*(Z(\widehat{G'}))_{\Gal(K/F)} \arrow[r, "p"] & X^*(Z(\widehat{G}))_{\Gal(K/F)}.
 \end{tikzcd}   
\end{equation*}
The maps we have constructed do not depend on our choice of $z$-extension. This follows from \cite[Lemma 2.4.4]{Kot5} where Kottwitz shows that if we have a map of reductive groups $f: G_1 \to G_2$ and $z$-extensions $H_i$ of $G_i$ for $i=1,2$, then $H_3 := H_2 \times_{G_2} H_1$ is a $z$-extension of $G_1$ and we have a commutative diagram:
\begin{equation*}
    \begin{tikzcd}
    H_1 \arrow[d, "\pi_1"]   & H_3 \arrow[r, "\tilde{f}"] \arrow[l] \arrow[d, "\pi_3"] & H_2 \arrow[d, "\pi_2"] \\
    G_1 \arrow[r, equals] & G_1 \arrow[r, "f"] & G_2.
    \end{tikzcd}
\end{equation*}
In particular, to prove our maps do not depend on choice of $z$-extension, we let $G_1=G_2$ and let the map $f$ be the identity. Note in that in the $\mc{E}_1$ case, we also need that the maps $H_3 \to H_1$ and $H_3 \to H_2$ are both surjections and hence induce maps $Z_{H_3} \to Z_{H_1}$ and $Z_{H_3} \to Z_{H_2}$. 

We can also use this lemma to prove that the map we have constructed for $\mc{E}_2$ is functorial for connected reductive $G$ and the map for $\mc{E}_1$ is functorial for connected reductive $G$ and maps $G_1 \to G_2$ that induce a map $Z_{G_1} \to Z_{G_2}$. To prove this last functoriality, we need that if $f: G_1 \to G_2$ induces a map $Z_{G_1} \to Z_{G_2}$ then $\tilde{f}: H_3 \to H_2$ induces a map $Z_{H_3} \to Z_{H_2}$. To see this, pick $z \in Z_{H_3}$. Then $f(\pi_3(z)) \in Z_{G_2}$ since both $f$ and $\pi_3$ induce maps of centers. Then $\tilde{f}(z) \in \pi^{-1}_2(Z_{G_2})=Z_{H_2}$ since $H_2$ is a central extension of $G_2$.

We have now constructed all the maps in Diagram \ref{keycommdiagram} in the general case. It remains to show the diagram commutes. Let $1 \to Z \to G' \to G \to 1$ be a $z$-extension of $G$. Then we can form Diagram \ref{keycommdiagram} for $G'$ and for $Z$. We get a map from the diagram of $G'$ to that of $G$ by functoriality and all the maps between these diagrams are surjective since they are given as quotients by the analogous objects in the diagram for $Z$.  Since $G'_{\der}$ is simply connected, we have proven in the previous section that the diagram for $G'$ is commutative. Moreover, all the squares between the diagram for $G'$ and the diagram for $G$ commute by a combination of the functoriality we proved in the previous paragraph, the functoriality of the Kottwitz map, and the compatibility of localization with the $\Phi$ construction. It follows by a simple diagram chase that the diagram for $G$ must also be commutative.

\subsection{Inflation}
In this section, we recall the results of \cite[\S 8]{Kot9} which allow us to obtain inflation maps that are compatible with localization. In particular, we can define $H^1_{\alg}(\mc{E}_2, G(\ov{\A_F})) = \varinjlim H^1_{\alg}(\mc{E}_2(K/F), G(\A_K))$ and analogously for the other cohomology sets. We can then promote Diagram \ref{keycommdiagram} to a commutative diagram

\begin{equation}{\label{limkeycommdiagram}}
\begin{tikzcd}
\bigoplus\limits_{u \in V_F} H^1_{\bas}(\mc{E}_{\iso}, G(\ov{F_u})) \arrow[d] & H^1_{\bas}(\mc{E}_2, G(\ov{\A_F}))  \arrow[l, "l^F", swap] \arrow[r] \arrow[d] & H^1_{\bas}(\mc{E}_1, G(\ov{\A}_F)/Z_G(\ov{F})) \arrow[d] \\
   \bigoplus\limits_{u \in V_F} X^*(Z(\widehat{G}))_{\Gamma_{F_u}} & X_2(G) \arrow[r, "\Sigma"] \arrow[l, "\sim"]& X^*(Z(\widehat{G}))_{\Gamma_F} 
\end{tikzcd}    
\end{equation}
where $X_2(G)= \varinjlim\limits_{K} (X^*(Z(\widehat{G})) \otimes \Z[V_K])_{\Gal(K/F)}$.

We first recall the localization maps for the gerb $\mc{E}_{\iso}$. Suppose $K/F$ and $L/F$ are finite Galois extensions of local fields and that $K \subset L$. Then we have the following diagram of extensions
\begin{equation*}
\begin{tikzcd}
1 \arrow[r] & \Gm(K) \arrow[d, equal] \arrow[r] & \mc{E}_{\iso}(K/F) \arrow[r] & \Gal(K/F) \arrow[r] & 1\\
1 \arrow[r] & \Gm(K) \arrow[r] \arrow[d] &  \mc{E}'_{\iso}(K/F) \arrow[u] \arrow[r] \arrow[d] & \Gal(L/F) \arrow[r] \arrow[d, equal] \arrow[u, "\rho"]& 1 \\
1 \arrow[r] & \Gm(L) \arrow[r] & \mc{E}^{\inf}_{\iso}(K/F) \arrow[r]  &  \Gal(L/F) \arrow[r] \arrow[d, equal] & 1 \\
1 \arrow[r] & \Gm(L) \arrow[u, "p_{L/K}"] \arrow[r] & \mc{E}_{\iso}(L/F) \arrow[u, "\eta_{L/K}"]\arrow[r] & \Gal(L/F) \arrow[r] & 1
\end{tikzcd}
\end{equation*}
The gerb $\mc{E}'_{\iso}(K/F)$ is defined to be the fiber product $\mc{E}_{\iso}(K/F) \times_{\Gal(K/F)} \Gal(L/F)$ via the natural projection $\rho: \Gal(L/F) \to \Gal(K/F)$. We define $\mc{E}^{\inf}_{\iso}(K/F)$ as the pushout of $\mc{E}'_{\iso}(K/F)$ along the natural inclusion $\Gm(K) \hookrightarrow \Gm(L)$. Finally, the map $p_{L/K}$ is given by $x \mapsto x^{[L:K]}$ and Kottwitz show it induces a map $\eta_{L/K}$ as in  the above diagram. To each extension in the above diagram, we assign a set $M$ and a set $Y$ as in Definition \ref{cocycledef}. For the top extension, we let $\ov{Y_1}=Y_1=\Hom_K(\Gm, G)$ and $M_1=G(K)$. For the second extension we let $M_2=G(K)$ and $\ov{Y_2}=Y_2=\Hom_K(\Gm, G)$. We further set $M_3=M_4=G(L)$ and $\ov{Y_3}=Y_3=\ov{Y_4}=Y_4=\Hom_L(\Gm,G)$. We have a natural map $M_2 \to M_3$ given by inclusion and a map $\ov{Y_2} \to \ov{Y_3}$ given by base change. The map $Y_3 \to Y_4$ is pre-composition with $p_{L/K}$. Then we get the desired local inflation map
\begin{equation*}
    H^1_{\alg}(\mc{E}_{\iso}(K/F) , G(K)) \to H^1_{\alg}(\mc{E}_{\iso}(L/F), G(L))
\end{equation*}
as a composition of a change of $G$ map by $\rho$, a $\Phi$ map, and a $\Psi$ map.

We now consider global inflation. Suppose $K/F$ and $L/F$ are finite Galois extensions of global fields and that $K \subset L$. Recall the objects $A_i, X_i, D_i$ of Key Example \ref{keyexample}. Because we are changing the field extension, we use the notations $A_{i,K}, X_{i,K}, D_{i,K}$ instead. Then we have the following diagram of extensions
\begin{equation*}
\begin{tikzcd}
1 \arrow[r] & \Hom(X_{i,K}, A_{i,K}) \arrow[d, equal] \arrow[r] & \mc{E}_i(K/F) \arrow[r] & \Gal(K/F) \arrow[r] & 1\\
1 \arrow[r] & \Hom(X_{i,K}, A_{i,K}) \arrow[r] \arrow[d] &  \mc{E}'_i(K/F) \arrow[u] \arrow[r] \arrow[d] & \Gal(L/F) \arrow[r] \arrow[d, equal] \arrow[u, "\rho"]& 1 \\
1 \arrow[r] & \Hom(X_{i,K}, A_{i,L}) \arrow[r] & \mc{E}^{\inf}_i(K/F) \arrow[r]  &  \Gal(L/F) \arrow[r] \arrow[d, equal] & 1 \\
1 \arrow[r] & \Hom(X_{i,L}, A_{i,L})  \arrow[u, "p_i"] \arrow[r] & \mc{E}_i(L/F) \arrow[u, "\tilde{p_i}"]\arrow[r] & \Gal(L/F) \arrow[r] & 1
\end{tikzcd}
\end{equation*}
The gerb $\mc{E}'_i(K/F)$ is defined via pullback along $\rho$. The map $\Hom(X_{i,K}, A_{i,K}) \to \Hom(X_{i,K}, A_{i,L})$ is given by post-composition with the obvious isomorphism $A_{i,K} = A_{i,L}^{\Gal(L/K)}$. Then the gerb $\mc{E}^{\inf}_i(K/F)$ is defined via pushout along this map. The map $p_i: \Hom(X_{i,L}, A_{i,L}) \to \Hom(X_{i,K}, A_{i,L})$ is defined via pre-composition with maps $p_i: X_{i,K} \to X_{i,L}$. We have $X_{1,K}=\Z=X_{1,L}$ and we define $p_1$ to be multiplication by $[L:K]$. We recall that $X_{2,K}=\Z[V_K]$ and hence we define $p_2: X_{2,K} \to X_{2,L}$ by $p_2(v)= \sum\limits_{w \mid v} [L_w:K_v]w$. The map $p_3$ is defined via restriction of $p_2$. Then Kottwitz shows (\cite[Lemma 8.3]{Kot9}) that $\tilde{p_i}$ exists and is unique up to conjugation by $\Hom(X_{i,K}, A_{i,L})$.

We now define sets $M_{i,j}$ and $\ov{Y_{i,j}}$ and $Y_{i,j}$ where $i=1,2,3$ corresponds with the index $i$ in $\mc{E}_i(K/F)$ and $j=1,2,3,4$ indicates which extension in the above diagram we consider. For $M_{i,j}$ we have:
\begin{equation*}
    \begin{array}{|c|c|c|c|}
    \hline
       G(\A_K)/Z_G(K)  & G(\A_K)/Z_G(K) & G(\A_L)/Z_G(L) & G(\A_L)/Z_G(L) \\
       \hline
        G(\A_K) & G(\A_K) & G(\A_L) & G(\A_L)\\
        \hline
        G(K) & G(K) & G(L) & G(L)\\
        \hline
    \end{array}
\end{equation*}
For $Y_{i,j}$ we have:
\begin{equation*}
    \begin{array}{|c|c|c|c|}
    \hline
       \Hom_K(D_{1,K}, Z_G)  & \Hom_K(D_{1,K}, Z_G) & \Hom_K(D_{1,K}, Z_G) & \Hom_L(D_{1,L}, Z_G) \\
       \hline
        \Hom_K(D_{2,K},G) & \Hom_K(D_{2,K},G) & \Hom_K(D_{2,K},G) & \Hom_L(D_{2,L}, G)\\
        \hline
        \Hom_K(D_{3,K}, G) & \Hom_K(D_{3,K}, G) & \Hom_K(D_{3,K}, G) & \Hom_L(D_{3,L}, G)\\
        \hline
    \end{array}
\end{equation*}
Finally $\ov{Y_{i,j}}$ is:
\begin{equation*}
    \begin{array}{|c|c|c|c|}
    \hline
       \Hom_K(D_1, Z_G)  & \Hom_K(D_1, Z_G) & \Hom_K(D_1, Z_G) & \Hom_L(D_1, Z_G) \\
       \hline
        G(\A_K) \cdot \Hom_K(D_2,G) & G(\A_K) \cdot \Hom_K(D_2,G) & G(\A_K) \cdot \Hom_K(D_2,G) & G(\A_L) \cdot \Hom_L(D_2, G)\\
        \hline
        \Hom_K(D_3, G) & \Hom_K(D_3, G) & \Hom_K(D_3, G) & \Hom_L(D_3, G)\\
        \hline
    \end{array}
\end{equation*}
In all cases, the maps $\xi_{i,j}$ are obvious. Finally, the map $\ov{Y_{i,3}} \to \ov{Y_{i,4}}$ is given by pre-composition with $p_i: D_{i,K} \to D_{i,L}$.

Then the desired inflation maps
\begin{align*}
     H^1_{\bas}(\mc{E}_1(K/F) , G(\A_K)/Z_G(K)) & \to H^1_{\bas}(\mc{E}_1(L/F), G(\A_L)/Z_G(L))\\
     H^1_{\alg}(\mc{E}_2(K/F) , G(\A_K)) & \to H^1_{\alg}(\mc{E}_2(L/F), G(\A_L))\\
     H^1_{\alg}(\mc{E}_3(K/F) , G(K)) & \to H^1_{\alg}(\mc{E}_3(L/F), G(L))
\end{align*}
are given as a composition of a change of $G$ map, a $\Phi$ map and a $\Psi$ map. Note that one can also define inflation for basic sets when $i=2,3$ in analogy with the case when $i=1$.

One must also check that inflation is compatible with localization when $i=2$. Note that since the inflation map is a composition of change of $G$ and $\Phi$ and $\Psi$ maps, it is compatible with these forms of functoriality. The localization map is also a composition of such maps and so to prove functoriality it suffices to show that the local inflation map is the ``localization'' of the global one for $D_2$. This compatibility is stated for instance in \cite[\S 10.9]{Kot9}.

\section{Normalizing Transfer Factors}{\label{s: normalizingtransferfactors}}
In this section we define global transfer factors using the theory of global $\mb{B}(G)$ developed in \cite{Kot9} as well as our modest additions in the previous section. In this section we use the bold lettering to refer to algebraic groups defined over a number field.

Let $\mathbf{G}$ be a connected reductive group over a number field $F$. Fix a quasisplit inner form $\mathbf{G}^*$ of $\mathbf{G}$ and an inner twist $\Psi: \mathbf{G}^*_{\overline{F}} \to \mathbf{G}_{\overline{F}}$ (i.e. $\Psi$ is an isomorphism and $\Psi^{-1} \circ \sigma(\Psi)$ is inner for all $ \sigma \in \Gamma_F := \Gal(F^{\sep}/F)$). 

An extended pure inner twist $(\Psi_1, z^1)$ consists of an inner twist $\Psi_1: \mb{G}^*_{\ov{F}} \to \mb{G}_{\ov{F}}$ and a cocycle $z^1 \in Z^1_{\bas}(\mc{E}_3(K/F), \mb{G}^*(K))$ for some finite Galois extension $K/F$ such that the projection of $z^1$ to $Z^1(\Gamma_F, \mb{G}^*_{\ad}(\ov{F}))$ equals $\sigma \mapsto \Psi^{-1}_1 \circ \sigma(\Psi_1)$. An isomorphism of two extended pure inner twists $(\Psi_1, z^1)$ and $(\Psi_2, z^2)$ is a map $f: \mb{G_1} \to \mb{G_2}$ defined over $F$ and an element $g \in \mb{G}^*(F^{\sep})$ such that $\Psi^{-1}_2 \circ f \circ \Psi_1 = \Int(g)$ and $z^1_e=g^{-1} z^2_e \sigma_e(g)$ for all $e \in \mc{E}_3(K/F)$ projecting to $\sigma_e \in \Gal(K/F)$. One can easily check that if $(f, g)$ is an automorphism of the extended pure inner twist $(\Psi, z)$ then $f$ is given by $\Int(\Psi(g)): \mb{G} \to \mb{G}$ and $\Psi(g) \in \mb{G}(F)$ so that $f$ is given by conjugation by an element of $\mb{G}(F)$. The map $(\Psi_1, z^1) \mapsto z^1$ induces a bijection between isomorphism classes of extended pure inner twists and $H^1_{\bas}(\mc{E}_3, \mb{G}^*(\ov{F}))$. An analogous construction works locally with the gerb $\mc{E}_{\iso}$. 

We now impose the following assumptions on $\mathbf{G}$:

\begin{enumerate}
    \item $\Psi$ lifts to an extended pure inner twist (by \cite[Cor. 3.13.13]{KalTai} this is true if $\mb{G}$ has connected center).
    \item $\mathbf{G}$ satisfies the Hasse principle.
    \item $\mb{G}_{\der}$ is simply connected.
\end{enumerate}

\subsection{Defining Invariants}{\label{normtrans}}

To begin we need to describe the construction of some invariants. Fix a semisimple element $\gamma^* \in \mb{G}^*(F)$. Let $\gamma \in \mb{G}(\A_F)$ be stably conjugate to $\Psi(\gamma^*)$. We choose $g \in \mb{G}^*(\ov{\A_F})$ such that
\begin{equation}
    \Psi(g \gamma^* g^{-1})=\gamma.
\end{equation}

Denote the centralizer in $\mb{G}^*$ of $\gamma^*$ by $I^{\mb{G}^*}_{\gamma^*}$. Note that this is connected since we are assuming that $\mb{G}_{\der}$ is simply connected. In \cite[\S 6]{Kot6}, Kottwitz describes an obstruction, $\obs(\gamma) \in \mf{K}(I^{\mb{G}^*}_{\gamma^*})$, constructed by Langlands in \cite[CH VII]{Lan2} to the existence of an element of $\mb{G}(F)$ in the $\mb{G}(\A_F)$ conjugacy class of $\gamma$. The construction proceeds by first describing the obstruction in the case that $\mb{G}$ satisfies the Hasse principle and then reducing to that case by considering $\mb{G}_{\Sc}$. In our case we are actually assuming $\mb{G}$ satisfies the Hasse principle so this simplifies the discussion. We now describe the construction in the present case where $\mb{G}$ satisfies the Hasse principle. See also \cite[\S 4.1]{KalTai} whose exposition we follow somewhat closely, generalizing to the non-regular case.

Let $\sigma \mapsto u_{\sigma}$ be a set-theoretic lift of the cocycle $\sigma \mapsto \Psi^{-1} \circ \sigma(\Psi) \in Z^1(\Gamma_F, \mb{G}^*_{\ad}(\ov{F}))$ to $\mb{G}^*(\ov{F})$. Then we can easily check
\begin{equation}
    g^{-1}u_{\sigma}\sigma(g) \in I^{\mb{G}^*}_{\gamma^*}(\ov{\A}_F).
\end{equation}
Moreover, it is easy to see that the projection
\begin{equation}
    \sigma \mapsto g^{-1}u_{\sigma}\sigma(g) \in I^{\mb{G}^*}_{\gamma^*}(\ov{\A}_F)/ Z_{I^{\mb{G}^*}_{\gamma^*}}(\ov{F})
\end{equation}
is independent of the choice of our lift $u$ and gives a $1$-cocycle and hence a cohomology class in $H^1( \Gamma_F, I^{\mb{G}^*}_{\gamma^*}(\ov{\A}_F)/Z(I^{\mb{G}^*}_{\gamma^*})(\ov{F}))$.

Now, by \cite[Theorem 2.2]{Kot6} we have a map
\begin{equation}
 H^1( \Gamma_F, I^{\mb{G}^*}_{\gamma_0}(\ov{\A}_F)/Z(I^{\mb{G}^*}_{\gamma^*})(\ov{F})) \to \pi_0(Z(I^{\mb{G}^*}_{\gamma_0})^{\Gamma_F})^D.   
\end{equation}
We denote the image of the above cohomology class in $\pi_0(Z(I^{\mb{G}^*}_{\gamma^*})^{\Gamma_F})^D$ by $\obs (\gamma)$. It is easy to see that this class is independent of our choice of $g$. 

Given two semi-simple elements $\gamma, \gamma' \in \mb{G}(\A_F)$ that are conjugate in $\mb{G}(\ov{\A}_F)$, we find a $g \in \mb{G}(\ov{\A}_F)$ such that $g \gamma g^{-1}=\gamma'$. Then we define the invariant $\inv(\gamma', \gamma') \in H^1(\Gamma_F, I^{\mb{G}}_{\gamma}(\ov{\A}_F))$ given as the cohomology class corresponding to the cocycle $\sigma \mapsto g^{-1} \sigma(g)$. Here we are thinking of $I^{\mb{G}}_{\gamma}$ as a group defined over $\A_F$ but not necessarily $F$. It is easy to check that $\inv(\gamma, \gamma')$ is independent of $g$ and is trivial precisely when $\gamma$ and $\gamma'$ are conjugate in $\mb{G}(\A_F)$. Finally, observe that at each place $v$ of $F$ we can define a local invariant $\inv_v$ in the analogous way and that we have a natural isomorphism
\begin{equation}
    H^1(\Gamma_F, I^{\mb{G}}_{\gamma}(\ov{\A}_F)) = \bigoplus_v  H^1(\Gamma_F, I^{\mb{G}}_{\gamma}(\ov{F}_v)),
\end{equation}
(where the right-hand side is a direct sum of pointed sets). Indeed, for each cohomology class on the left there is a finite Galois extension $K/F$ so that the cocycle comes from some  $H^1(\Gal(K/F), I^{\mb{G}}_{\gamma}(\A_K))$. Each cocycle in $Z^1(\Gal(K/F), I^{\mb{G}}_{\gamma}(\A_K))$ has finite image hence factors through $\mc{I}^{\mb{G}}_{\gamma}(\mc{O}_{K_u})$ for all but finitely many places $u$ of $K$ where $\mc{I}^{\mb{G}}_{\gamma}$ is a suitable integral model of $I^{\mb{G}}_{\gamma}$ away from finitely many places. The set $H^1(\Gal(K_u/F_v), \mc{I}^{\mb{G}}_{\gamma}(\mc{O}_{K_u}))$ is then trivial at those $v$ such that $K_u/F_v$ is unramified by a standard application of Hensel's lemma and Lang's theorem.

We now prove
\begin{proposition}
The invariant $\obs(\gamma)$ satisfies the following properties
\begin{enumerate}
    \item The invariant $\obs(\gamma)$ depends only on the $\mb{G}(\A_F)$-conjugacy class of $\gamma$.\\
    \item If $\gamma, \gamma' \in \mb{G}(\A_F)$ are stably conjugate, then
    \begin{equation}
        \obs(\gamma')= \inv(\gamma, \gamma') \cdot \obs(\gamma),
    \end{equation}
    where the product on the right is given as in \cite[Lemma 2.4]{Kot6} \\
    \item The invariant $\obs(\gamma)$ is trivial if and only if the $\mb{G}(\A_F)$-conjugacy class of $\gamma$ contains an $F$-point. 
\end{enumerate}
\end{proposition}
\begin{proof}
For the most part, the proofs in \cite[4.1.1, 4.1.2]{KalTai} go through unchanged.

For the first part, suppose that $\gamma' \in \mb{G}(\A_F)$ is conjugate to $\gamma$. Then pick some $x \in \mb{G}(\A_F)$ so that $x\gamma x^{-1}=\gamma'$. Then $\obs(\gamma')$ is given by the class of
\begin{equation*}
    \sigma \mapsto g^{-1} \psi^{-1}(x^{-1}) u_{\sigma}\sigma(\psi^{-1}(x))\sigma(g).
\end{equation*}
Then by definition of $u(\sigma)$,
\begin{equation*}
    \Int(u_{\sigma})(\sigma(\psi^{-1}(x)))=\psi^{-1}(\sigma(x)),
\end{equation*}
so that the above becomes
\begin{equation*}
    \sigma \mapsto g^{-1}\psi^{-1}(x^{-1}\sigma(x)) u_{\sigma}\sigma(g).
\end{equation*}
Finally, since $x \in \mb{G}(\A_F)$, we have $\sigma(x)=x$ so that the above equals $\obs(\gamma)$ as desired.

For the second part, we choose $x \in \mb{G}(\ov{\A}_F)$ such that $x\gamma x^{-1}= \gamma'$. Then  we get by a similar computation that $\obs(\gamma')$ is given by the class of
\begin{equation*}
    g^{-1} \psi^{-1}(x^{-1}\sigma(x))u_{\sigma}\sigma(g)=(\Int(g^{-1}) \circ \psi^{-1})(x^{-1}\sigma(x)) \cdot g^{-1}u_{\sigma}\sigma(g).
\end{equation*}
We then note that we are precisely in the situation of \cite[Lemma 2.4]{Kot6}.

Lastly, we show the third and most important part of the proposition. If the $\mb{G}(\A_F)$-conjugacy class of $\gamma$ has an $F$-point, then by the first part of the proposition, we can assume that $\gamma$ is that $F$-point. Then we can pick $g \in \mb{G}^*(\ov{F})$ which implies that $g^{-1}u_{\sigma} \sigma(g) \in I^{\mb{G}^*}_{\gamma^*}(\ov{F})$. Now we cite \cite[Theorem 2.2]{Kot6} where it is proven that the kernel of the map
\begin{equation*}
    H^1( \Gamma_F, I^{\mb{G}^*}_{\gamma^*}(\ov{\A}_F)/Z(I^{\mb{G}^*}_{\gamma^*})(\ov{F})) \mapsto \pi_0(Z(I^{\mb{G}^*}_{\gamma^*})^{\Gamma_F})^D,
\end{equation*}
is given by the image of the map
\begin{equation*}
    H^1( \Gamma_F, I^{\mb{G}^*}_{\gamma^*}(\ov{F})/Z(I^{\mb{G}^*}_{\gamma^*})(\ov{F})) \to H^1( \Gamma_F, I^{\mb{G}^*}_{\gamma^*}(\ov{\A}_F)/Z(I^{\mb{G}^*}_{\gamma^*})(\ov{F})).
\end{equation*}
This implies that $\obs(\gamma)$ is trivial.

Conversely, suppose that $\obs(\gamma)$ is trivial. Then by \cite[Theorem 2.2]{Kot6} we have that the class of $g^{-1}u_{\sigma}\sigma(g)$  in $ H^1( \Gamma_F, I^{\mb{G}^*}_{\gamma^*}(\ov{\A}_F)/Z(I^{\mb{G}^*}_{\gamma^*})(\ov{F}))$ lies in the image of the map
\begin{equation*}
    H^1( \Gamma_F, I^{\mb{G}^*}_{\gamma^*}(\ov{F})/Z(I^{\mb{G}^*}_{\gamma^*})(\ov{F})) \to H^1( \Gamma_F, I^{\mb{G}^*}_{\gamma^*}(\ov{\A}_F)/Z(I^{\mb{G}^*}_{\gamma^*})(\ov{F})).
\end{equation*}
in particular, this means we can pick an $x \in I^{\mb{G}^*}_{\gamma^*}(\ov{\A}_F)$ such that the image of $x^{-1}g^{-1}u_{\sigma}\sigma(g)\sigma(x) \in Z^1(\Gamma_F, I^{\mb{G}^*}_{\gamma^*}(\ov{\A}_F)/Z(I^{\mb{G}^*}_{\gamma^*})(\ov{F}))$ lies in $I^{\mb{G}^*}_{\gamma^*}(\ov{F})/Z(I^{\mb{G}^*}_{\gamma^*})(\ov{F})$. Hence for each $\sigma \in \Gamma_F$ we can find an element $c \in Z(I^{\mb{G}^*}_{\gamma^*})(\ov{F})$ such that 
\begin{equation*}
    cx^{-1}g^{-1}u_{\sigma}\sigma(g)\sigma(x)\in I^{\mb{G}^*}_{\gamma^*}(\ov{F}),
\end{equation*}
which implies that $x^{-1}g^{-1}u_{\sigma}\sigma(g)\sigma(x) \in I^{\mb{G}^*}_{\gamma^*}(\ov{F})$.

Therefore, we can assume without loss of generality that $g^{-1} u_{\sigma} \sigma(g) \in I^{\mb{G}^*}_{\gamma^*}(\ov{F})$ for all $\sigma$. Now observe that
\begin{equation*}
    \psi(g^{-1}u_{\sigma}\sigma(g)u(\sigma)^{-1})=\psi(g)^{-1} \sigma(\psi(g))
\end{equation*}
and hence
\begin{equation*}
   z:= \sigma \mapsto \psi(g^{-1}u_{\sigma}\sigma(g)u_{\sigma}^{-1}),
\end{equation*}
gives an element of $Z^1(\Gamma_F, \mb{G}(\ov{F}))$ whose image is cohomologically trivial in $Z^1(\Gamma_F, \mb{G}(\ov{\A}_F))$. Hence by the Hasse principle for $\mb{G}$, we have $z \in Z^1(\Gamma_F, \mb{G}(\ov{F}))$ is cohomologically trivial and so there equals $h \in \mb{G}^*(\ov{F})$ such that 
\begin{equation*}
    1=\psi(gh)^{-1} \sigma(\psi(gh))=\psi(g^{-1}u_{\sigma}\sigma(g)u_{\sigma}^{-1}).
\end{equation*}
Then $\psi(gh) \in \mb{G}(\A_F)$ so that
\begin{equation*}
    \gamma' := \psi(gh)^{-1}\gamma \psi(gh)
\end{equation*}
is in the $\mb{G}(\A_F)$-conjugacy class of $\gamma$. On the other hand, by definition
\begin{equation*}
    \psi(\gamma^*)=\psi(g)^{-1} \gamma \psi(g),
\end{equation*}
so that
\begin{equation*}
    \psi(h^{-1} \gamma^* h)=\gamma'
\end{equation*}
which implies $\gamma' \in \mb{G}(\ov{F})$ hence $\mb{G}(F)$ as desired.
\end{proof}

\subsection{Refined Invariants}{\label{sec:refinedinv}}

Using Kottwitz's theory of $\mb{B}(G)$ for local and global fields as well as our work in \S\ref{globalBGappendix}, we construct, in the case that $\mb{G}$ satisfies the assumptions of the previous subsection, a refinement of the invariant $\obs(\gamma)$. We loosely follow \cite[\S 4.1]{KalTai} and freely use the notation from Section \ref{globalBGappendix}.

Since we assume that $\Psi$ lifts to an extended pure inner twist, there exists a finite Galois extension $K$ of $F$ and a cocycle $(\nu, z^{\iso}) \in Z^1_{\bas}( \mc{E}_3(K/F), \mb{G}^*(K))$ that lifts the element $z \in Z^1(\Gamma_F, \mb{G}^*_{\ad}(\ov{F}))$ corresponding to $\Psi$. In particular, we have that
\begin{equation}
\Psi^{-1} \circ \sigma_e(\Psi)=\Int (\ov{z^{\iso}_{e}})
\end{equation}
where $\ov{g}$ denotes the projection $\mb{G}^*(\ov{F}) \to \mb{G}^*_{\ad}(\ov{F})$, for any $e \in \mc{E}_3$ that projects to $\sigma_e \in \Gamma_F$.

We now pick semisimple $\gamma^* \in \mb{G}^*(F)$ and $\gamma \in \mb{G}(\A_F)$ conjugate to $\gamma^*$ in $\mb{G}(\ov{\A_F})$ and choose a $g \in \mb{G}^*(\ov{\A_F})$ such that
\begin{equation*}
    \Psi(g\gamma^*g^{-1})=\gamma.
\end{equation*}
We now use the maps constructed in Section \ref{bgmaps} to produce a cocycle  $(\nu_2, z^{\iso, 2}) \in  Z^1_{\bas}(\mc{E}_2(K/F), \mb{G}^*(\A_K))$ such that for each $\sigma \in \Gal(K/F)$, if $e_2, e_3$ are lifts of $\sigma$ to $\mc{E}_2(K/F)$ and $\mc{E}_3(K/F)$ respectively, then $\ov{z^{\iso}_{e_3}} = \ov{z^{\iso, 2}_{e_2}} \in  \mb{G}^*_{\ad}(\ov{F})$. In particular, starting with $(\nu, z^{\iso})$, we post-compose $z^{\iso}$ with the map $G(K) \to G(\A_K)$, then pushforward to $\mc{F}'$ and pullback to $\mc{E}_2(K/F)$ to get $z^{\iso,2}$. The desired property is then clear from the definitions of the pullback and pushforward maps and the fact that the image of $\nu$ is central in $\mb{G}^*$.

We define an abstract $\mc{E}_2(K/F)$-cocycle valued in $I^{\mb{G}^*}_{\gamma^*}(\A_K)$ by
\begin{equation}
    e \mapsto g^{-1} z^{\iso, 2}_e \sigma_e(g).
\end{equation}
It is easy to see that we indeed have $g^{-1} z^{\iso, 2}_e \sigma_e(g) \in I^{\mb{G}^*}_{\gamma^*}(\A_K)$ from the fact that $\ov{z^{\iso}_{e_3}} = \ov{z^{\iso, 2}_{e_2}} \in  \mb{G}^*_{\ad}(\ov{F})$.

If we restrict this cocycle to $D_2(\A_K)$ it equals $\nu_2$ and hence induces an element $\inv[z^{\iso}](\gamma^*, \gamma) \in H^1_{\bas}(\mc{E}_2, I^{\mb{G}^*}_{\gamma^*}(\ov{\A_F}))$. Via the localization maps, we get an element $\inv[z^{\iso}](\gamma^*, \gamma)(u) \in  H^1_{\bas}(\mc{E}_{\iso}, I^{\mb{G}^*}_{\gamma^*}(\ov{F_u}))$. Crucially, this class is trivial for all but finitely many places $u$ as shown in \S \ref{totalloc}. Let $z^{\iso}(u)$ be the localization of the cocycle $z^{\iso}$ at the place $u$. We record the following lemma.
\begin{lemma}{\label{bgloclem}}
The class of $e \mapsto g^{-1}_vz^{\iso}(u)_e \sigma_e(g_v)$ in $H^1_{\bas}(\mc{E}_{\iso}, I^{\mb{G}^*}_{\gamma^*}(\ov{F_u}))$ equals that of $\inv[z^{\iso}](\gamma^*, \gamma)(u)$.
\end{lemma}
\begin{proof}
We work at some fixed extension $K$ such that $z^{\iso}$ factors through $H^1_{\bas}(\mc{E}_2(K/F), \mb{G}^*(\A_K))$ and $g \in \mb{G}^*(\A_K)$.

The localization map defined in Section \ref{BGloc} at $u$ is a composition of maps
\begin{align*}
& H^1_{\alg}(\mc{E}_2(K/F), \mb{G}^*(\A_K)) \longrightarrow  H^1_{\alg}(\mc{E}_2(K/E^v), \mb{G}^*(\A_K))\\ 
& \longrightarrow H^1_{\alg}(\mc{E}^v_2(K/E^v), \mb{G}^*(K_v))
\longrightarrow H^1_{\alg}(\mc{E}_{\iso}(K_v/F_u), \mb{G}(K_v)).
\end{align*}
These maps are all defined on the level of cocycles and are induced from maps of extensions.

In particular, we have a natural map $\tilde{\mu_v}: \mc{E}^{\iso}(K_v/F_u) \to \mc{E}^v_2(K/E^v)$ as in Section \ref{BGloc}. Pick some $e \in \mc{E}^{\iso}(K_v/F_u)$.Then $\tilde{\mu_v}(e) \in \mc{E}^v_2(K/E^v)$ and we have $\tilde{\mu_v}(e)=d l(e')$ where $e' \in \mc{E}_2(K/E^v)$ and $d \in D_2(K_v)$ and where $l: \mc{E}_2(K/E^v) \to \mc{E}^v_2(K/E^v)$. 

Then by definition of the localization map, we have that the cocycle $z$ giving $\inv[z^{\iso}](\gamma^*, \gamma)(u)$ that is induced by the above maps from $g^{-1}z^{\iso, 2}\sigma(g)$ satisfies the following equality:
\begin{equation*}
    z(e)=[g^{-1} z^{\iso, 2}_{e'}\sigma_e(g)]_v\nu_v(d) =g^{-1}_v [z^{\iso, 2}_{e'}\nu_v(d)]\sigma_e(g_v)=g^{-1}_v z^{\iso ,2}(v)_e \sigma_e(g_v).
\end{equation*}
This is the desired equality.
\end{proof}

\begin{proposition}{\label{obseqz}}
The image of $\inv[z^{\iso}](\gamma^*, \gamma)$ in $H^1_{\bas}(\mc{E}_1, I^{\mb{G}^*}_{\gamma^*}(\ov{\A_F})/Z_{I^{\mb{G}^*}_{\gamma^*}}(\ov{F}))$ lies in $H^1(\Gamma_F,  I^{\mb{G}^*}_{\gamma^*}(\ov{\A_F}) / Z_{I^{\mb{G}^*}_{\gamma^*}}(\ov{F}))$ and agrees with $\obs(\gamma) \in X^*(Z(\widehat{I^{\mb{G}^*}_{\gamma^*}}))_{\Gal(\ov{F}/F)}$.
\end{proposition}
\begin{proof}
Suppose $(\nu_1, z^{\iso, 1})$ is a cocycle representative of the image of $\inv[z^{\iso}](\gamma^*, \gamma)$ in $H^1_{\bas}(\mc{E}_1,  I^{\mb{G}^*}_{\gamma^*}(\ov{\A_F})/Z_{I^{\mb{G}^*}_{\gamma^*}}(\ov{F}))$. Then $\nu_1$ is trivial since it is equal to the pre-composition of $\nu$ by $D_1 \to D_2 \to D_3$ which is trivial. This implies the first claim. 

For the second claim, we note that by definition, the projection of $z^{\iso}$ to $\mb{G}^*_{\ad}$ is constant on $D_3(K)$ and descends to give the cocycle $z \in Z^1(\Gamma_F, \mb{G}^*_{\ad}(\ov{F}))$. By construction the same will also be true $z^{\iso, 2}$ if we quotient out by the image of $\nu_2$ and $z^{\iso}_1$ if we quotient by the image of $\nu_1$. But $\nu_1$ is in fact trivial so this gives our desired result.
\end{proof}

We have now constructed $\inv[z^{\iso}](\gamma^*, \gamma)$ which is a refinement of $\obs(\gamma)$.

\subsection{Overview of Transfer Factors}{\label{transfactdefnsect}}

In this subsection, we briefly review the theory of transfer factors.  To that end, we consider a connected reductive group $G$ defined over a local field $F$ of characteristic $0$ and let $\mc{E}^r(G)$ denote the set of isomorphism classes of refined endoscopic data as in \cite[\S2.3]{BM2}. A refined endoscopic datum is a tuple $(H, s, \eta)$ where $H$ is a quasisplit connected reductive group over $F$, where $s \in Z(\widehat{H})^{\Gamma_F}$, and $\eta: \widehat{H} \to \widehat{G}$ is such that $I^{\widehat{G}}_{\eta(s)} = \eta(\widehat{H})$ and the $\widehat{G}$-conjugacy class of $\eta$ is $\Gamma_F$-stable. We say that $(H, s, \eta)$ and $(H', s', \eta')$ are isomorphic if there exists an $F$-isomorphism $\alpha: H' \to H$ and a choice of $\widehat{\alpha}: \widehat{H} \to \widehat{H'}$ ($\widehat{\alpha}$ is determined by $\alpha$ up to $\widehat{H}$-conjugacy) and $g \in \widehat{G}$ such that the diagram 
\begin{equation*}
\begin{tikzcd}
 \widehat{H} \arrow[d, " \widehat{\alpha}", swap] \arrow[r, "\eta"] & \widehat{G} \arrow[d, "\Int(g)"]\\
 \widehat{H'} \arrow[r, "\eta'", swap] & \widehat{G}
\end{tikzcd}    
\end{equation*}
commutes and $\widehat{\alpha}(s) = s'$. 

Transfer factors arise in the comparison of orbital integrals over $G$ and endoscopic groups $H$. To describe them, we fix a refined endoscopic datum $(H, s, \eta)$ and fix a lift of $\eta$ to a map ${}^L \eta : {}^LH \to {}^LG$. Such a lift will always exist if $G_{\der}$ is simply connected but may not exist in general. Langlands and Shelstad (\cite{LS2}) construct a local transfer factor which is a function $\Delta: H(F)_{G-\sr} \times G(F)_{\sr} \to \C$, where $G(F)_{\sr}$ denotes the subset of strongly regular semisimple elements of $G(F)$ and $H(F)_{G-\sr}$ denote the subset of semisimple elements of $H(F)$ that transfer to strongly regular elements of $G(F)$. The local transfer factor is canonical up to multiplication by a scalar in $\C^{\times}$. 

Now suppose that $F$ is a global field and fix $\mb{G}$ a reductive group over $F$ and an inner twist $\psi: \mb{G}^* \to \mb{G}$ between $\mb{G}$ and its quasi-split inner form $\mb{G}^*$. Then after fixing global analogues of the data $(\mb{H}, s, {}^L\eta)$, Langlands and Shelstad construct a global transfer factor, which is a function $\Delta: \mb{H}(\A_F)_{\mb{G}-\sr} \times \mb{G}(\A_F)_{\sr} \to \C$. Following the construction in \cite[\S7.3]{KS}, this global factor is defined as a product over each place $v$ of $F$ of local transfer factors of $(\mb{H}_v, \mb{G}_v)$. We require that these local transfer factors are compatibly chosen in the sense that they are constructed from fixed global Whittaker data, $a$-data and $\chi$-data by taking localizations at each place.  Unlike the local factors, this global factor is made completely canonical by observing that when $\gamma^{\mb{H}} \in \mb{H}(F)_{\mb{G}-\sr}$ and transfers to $\gamma^* \in \mb{G}^*(F)$ and $\gamma \in \mb{G}(\A_F)$, one has
\begin{equation}{\label{transfactcanoneqn}}
    \Delta(\gamma_{\mb{H}}, \gamma) = \langle \obs(\gamma), \widehat{\varphi}^{-1}_{\gamma^*, \gamma^{\mb{H}}}(s) \rangle,
\end{equation}
where $\varphi_{\gamma^*, \gamma^{\mb{H}}}: S_{\mb{H}} \to S$ is an admissible isomorphism of the maximal tori $S_{\mb{H}} \subset \mb{H}$ and $S \subset \mb{G}^*$ equal to the centralizers of $\gamma^{\mb{H}}$ and $\gamma^*$ in their respective groups.

In practice, one often uses transfer factors to relate the Langlands correspondences of $\mb{G}(F_v)$ and $\mb{H}(F_v)$. To do so, one needs a canonical normalization of local transfer factors that is compatible with analogous constructions in representation theory. When $\mb{G}$ is quasi-split, a canonical normalization is given by fixing a Whittaker datum as described in \cite[\S5.3]{KS}. When $\mb{G}$ is not quasi-split, the problem of finding a canonical normalization of transfer factors at each place $v$ of $F$ compatible with the Langlands correspondence is quite subtle and described extensively in \cite{Kal1}. There are two approaches to solving this problem. The first uses $\mb{B}(F,\mb{G})$ and is described in \cite{KalTai} and the second uses the rigid gerb of Kaletha (\cite{KalethaGlobalRigid}). 

However, in the context of the trace formula for Shimura varieties one needs to extend transfer factors to the domain $\mb{H}(\A_F)_{(\mb{G}, \mb{H})-\reg} \times \mb{G}(\A_F)_{\semis}$, where $\mb{H}(\A_F)_{(\mb{G}, \mb{H})-\reg}$ denotes the locus of $(\mb{G}, \mb{H})$-regular semisimple elements of $\mb{H}(\A_F)$. These are the semisimple $\gamma^{\mb{H}} \in \mb{H}(\A_F)$ such that if $\gamma^* \in \mb{G}^*(F)$ is a transfer of $\gamma^{\mb{H}}$, then $I^{\mb{G}^*}_{\gamma^*}$ is an inner form of $I^{\mb{H}}_{\gamma^{\mb{H}}}$. In \cite{LS1}, Langlands and Shelstad show that one can extend local and global transfer factors by continuity to $H(F)_{(G,H)-\reg} \times G(F)_{\semis}$ and $\mb{H}(\A_F)_{(\mb{G}, \mb{H})-\reg} \times \mb{G}(\A_F)_{\semis}$ respectively. It will be useful to record explicit formulas for these transfer factors on the $H(F)_{(G,H)-\reg} \times G(F)_{\semis}$ locus. This will be done in \S \ref{transfactsect}.

It is asserted in \cite[\S2.4]{LS1} (immediately before their \S2.5) that the global transfer satisfies \eqref{transfactcanoneqn} for $\gamma^{\mb{H}} \in \mb{H}(F)_{(\mb{G}, \mb{H})-\reg}$ that transfers to $\gamma^* \in \mb{G}^*(F)$ and $\gamma \in \mb{G}(\A_F)$ but no proof of this assertion is given. In \cite[Lemma 4.1.(i)]{ArthurGlobalDescent}, Arthur proves a formula that implies this fact in the case of elliptic endoscopy and elliptic semisimple elements. In Corollary \ref{cor: globaltranscor} we record a proof of this fact for all $\mb{G}$ arising from $\mb{B}(F, \mb{G}^*)$.

% To use the work of Arthur, we need to first reduce to the case where $\gamma^{\mb{H}}$ and $\gamma^*$ are elliptic and $(\mb{H}, s, {}^L\eta)$ is an elliptic endoscopic datum. Consider the Levi subgroup $\mb{M}^*$ of $\mb{G}^*$ given as the centralizer of the maximal $F$-split torus in $Z(Z_{\mb{G}^*}(\gamma^*))$. By construction, $\gamma^*$ is elliptic in $\mb{M}^*$ and transfers to $\gamma \in (???) $ Arthur constructs a map $\mc{X}(G) \to \mc{Y}(G)$ taking pairs $()$

For applications (cf. \cite[\S4]{BM2}), it is useful to describe the transfer factors on the $(\mb{G}, \mb{H})$-regular locus explicitly without using limits. In \S\ref{transfactsect}, we use the theory of $\mb{B}_2(F, \mb{G})$ we have developed to construct local and global transfer factors on the $(\mb{G}, \mb{H})$-regular locus that agree with Kaletha's construction of $\mb{B}(F,\mb{G})$-normalized transfer factors on the strongly regular locus.

\subsection{Construction of factors}{\label{transfactsect}}
We now construct our transfer factors. We return to the notation of \S\ref{sec:refinedinv}. Fix an $F$-splitting $(B, T, \{X_{\alpha}\})$ of $\mb{G}^*$ and a nontrivial character $\chi: \A_F/F \to \C^{\times}$. Following \cite[\S5.3]{KS} (see also \cite{Kal1}), the $F$-splitting and $\chi$ induce a pair $(B, \lambda_{\chi})$ where $B$ has unipotent radical $N$ and $\lambda_{\chi}: N(\A_F)/N(F) \to \C^{\times}$ is a \emph{generic} character (i.e. $\lambda_{\chi}$ is non-trivial when restricted to each simple relative root subgroup). The $\mb{G}^*(F)$-conjugacy class of $(B, \lambda_{\chi})$ gives a global Whittaker datum $\mathfrak{w}$ of $\mathbf{G}^*$. For each place $v$ of $F$, our splitting plus the restriction, $\chi_v$, of $\chi$ to $F_v$ induce a local Whittaker datum $\mf{w}_v$. This is the same as the local Whittaker datum induced by the inclusion $N(F_v) \to N(\A_F)$. In this section we use the notation $\mb{B}(F,G)$ for $H^1_{\alg}(\mc{E}_{\iso}, G(\ov{F}))$, where $F$ is a local field.

We now fix an isomorphism class in $\mc{E}^r(\mb{G})$ and a representative $(\mb{H}, s, \eta)$ of this class. Since $\mb{G}_{\der}$ is simply connected, we can lift $\eta$ to a map ${}^L\eta: {}^L\mb{H} \to {}^L\mb{G}$. We fix $\Gamma_F$ equivariant splittings of $\widehat{\mb{G}}$ and $\widehat{\mb{G^*}}$. Then the map $\Psi: \mb{G^*} \to \mb{G}$ induces an isomorphism ${}^L\mb{G} \to {}^L\mb{G}^*$ preserving the splittings and hence we can consider $(\mb{H}, s, \eta)$ to be a refined endoscopic datum of $\mb{G}^*$.

At each place $v$ of $F$, we get a refined endoscopic datum $(\mb{H}_{F_v} , s, \eta)$ of $\mb{G}^*_{F_v}$. This combined with our choice of ${}^L \eta$ gives the ``Whittaker normalized'' transfer factor between $\mb{H}_v$ and $\mb{G}^*_v$ which we denote $\Delta[\mf{w}_v](\gamma^{\mb{H}}, \gamma^*)$ for $\gamma^{\mb{H}} \in \mb{H}(F_v)$ a $(\mb{G},\mb{H})$-regular semisimple element and $\gamma^*$ a semisimple element of $\mb{G}^*(F_v)$. As explained in  \cite[\S5.5]{KS2}, there are two normalizations of the factor which are compatible with twisted endoscopy and these are denoted by $\Delta'_{\lambda}$ and $\Delta^{\lambda}_D$. The first transfer factor is compatible with the arithmetic normalization of the local Langlands correspondence while the second is compatible with the geometric normalization. In this paper, we will use the second of these normalizations which notably differs from the choice made in \cite{KalTai}. Our construction of the various $\obs$ and $\inv$ invariants is analogous to \cite{KalTai} and hence differs from that of \cite{KS} (cf \cite[Rem. 4.2.2]{KalTai}).

We now record explicit formulas for $\Delta[\mf{w}_v](\gamma^{\mb{H}}, \gamma^*)$. Recall that these transfer factors are defined in \cite[\S2.4]{LS1} by taking a limit of $\Delta[\mf{w}_v](\gamma^{\mb{H}}_n, \gamma^*_n)$, such that $(\gamma^{\mb{H}}_n, \gamma^*_n) \in  \mb{H}(F_v)_{\mb{G}^*-\sr} \times \mb{G}^*(F_v)_{\sr}$ and each $\gamma^{\mb{H}}_n$ transfers to $\gamma^*_n$. We assume further that all $\gamma^{\mb{H}}_n$ (resp. $\gamma^*_n$) lie in a fixed maximal torus $S_{\mb{H}} \subset \mb{H}$ (resp. $S \subset \mb{G}^*$) containing $\gamma^{\mb{H}}$ (resp. $\gamma^*$). We let $\varphi_{\gamma^*, \gamma_{\mb{H}}}: S_{\mb{H}} \to S$ be an admissible isomorphism of these tori satisfying $\varphi_{\gamma^*, \gamma_{\mb{H}}}(\gamma^{\mb{H}}) = \gamma^*$, and we assume $(\gamma^{\mb{H}}_n, \gamma^*_n)$ have been chosen such that $\varphi_{\gamma^*, \gamma^{\mb{H}}}(\gamma^{\mb{H}}_n) = \gamma^*_n$. We also make use of fixed $\Gal(\ov{F}/F)$-stable Borel pairs $(B,T), (\wh{B}, \wh{T}), (B_{\mb{H}}, T_{\mb{H}}), (\wh{B_{\mb{H}}}, \wh{T_{\mb{H}}})$ and assume that $\eta(\wh{T_{\mb{H}}}) = \wh{T}$ and $\eta(\wh{B_{\mb{H}}}) = \wh{B}$. 

The Whittaker-normalized transfer factor $\Delta[\mf{w}_v]$ is a product of a number of terms:
\begin{equation*}
    \Delta[\mf{w}] = \epsilon_L(V, \chi_v) \Delta_I\Delta_{II}\Delta_{III_{2,D}}\Delta_{IV}, 
\end{equation*}
which we examine in turn. See also \cite[\S1.3]{Kal1} for a summary of these terms (though beware that this source uses the $\Delta'_{\lambda}$ normalization). 

The local $\epsilon$-factor  $\epsilon_L(V, \chi_v)$ for the virtual representation $V= X^*(T)_{\C} - X^*(T_{\mb{H}})_{\C}$ does not depend on $(\gamma^{\mb{H}}_n, \gamma^*_n)$ and hence is the same as in the strongly regular case. Similarly, $\Delta_I$ depends only on $S, T$ and $s$ and so is constant over our limit.

The term $\Delta_{II}$ requires picking $\chi$ and $a$-data for the roots $R(S,\mb{G}^*)$. Then it is given as a product
\begin{equation}{\label{eqn: delta2}}
    \Delta_{II} = \prod\limits_{\alpha} \chi_{\alpha}\left(\frac{\alpha(\gamma^*_n) - 1}{a_{\alpha}}\right),
\end{equation}
where the product is over sets of representatives of $\Gal(\ov{F}_v/F_v)$-orbits in $ R(S,\mb{G}^*) \setminus \varphi_{\gamma^*, \gamma^{\mb{H}}} R(S_{\mb{H}}, \mb{H})$. Note that this limit exists because by the $(\mb{G}^*, \mb{H})$-regular assumption, $\gamma^*$ does not vanish on the roots $\alpha$ that appear formula \eqref{eqn: delta2}. Moreover, the limit is clearly given by replacing $\gamma^*_n$ with $\gamma^*$ in formula \eqref{eqn: delta2}.

The term $\Delta_{III_{2,D}}$ equals $a(\gamma^*_n)$, where $a$ is a certain character of $S(F)$ coming from Langlands duality for tori. The character does not depend on $(\gamma^{\mb{H}}_n, \gamma^*_n)$ and hence the limit is simply $a(\gamma^*)$.

Finally, the term $\Delta_{IV}$ is given by
\begin{equation}{\label{eqn: delta4}}
    \Delta_{IV} = \frac{|\det(\Ad(\gamma^*_n) -1 \mid \Lie(\mb{G}^*) \setminus \Lie(S) )|^{\frac{1}{2}}_v}{|\det(\Ad(\gamma^{\mb{H}}_n) -1 \mid \Lie(\mb{H}) \setminus \Lie(S_{\mb{H}}) )|^{\frac{1}{2}}_v}.
\end{equation}
As $\gamma^*_n \to \gamma^*$ (resp. $\gamma^{\mb{H}}_n \to \gamma^{\mb{H}}$), the map $\Ad(\gamma^*_n) -1: \Lie(I^{\mb{G}^*}_{\gamma^*}) \to \Lie(\mb{G}^*)$ (resp. $\Ad(\gamma^{\mb{H}}_n) -1: \Lie(I^{\mb{H}}_{\gamma^{\mb{H}}}) \to \Lie(\mb{H})$) becomes trivial. Because of $(\mb{G}^*, \mb{H})$-regularity, the Lie algebras $\Lie(I^{\mb{H}}_{\gamma^{\mb{H}}})$ and $\Lie(I^{\mb{G}^*}_{\gamma^*})$ are of the same dimension so formula \eqref{eqn: delta4} becomes in the limit:
\begin{equation}{\label{eqn: limitdelta4}}
       \Delta_{IV} = \frac{|\det(\Ad(\gamma^*) -1 \mid \Lie(\mb{G}^*) \setminus \Lie(I^{\mb{G}^*}_{\gamma^*}) )|^{\frac{1}{2}}_v}{|\det(\Ad(\gamma^{\mb{H}}) -1 \mid \Lie(\mb{H}) \setminus \Lie(I^{\mb{H}}_{\gamma^{\mb{H}}}) )|^{\frac{1}{2}}_v}.
\end{equation}
We now prove the following Lemma using the above computations.
\begin{lemma}{\label{lem: transfactvanish}}
    For the number field $F$ and a pair $(\gamma^{\mb{H}}, \gamma^*) \in \mb{H}(F)_{(\mb{G}^*, \mb{H})-\reg} \times \mb{G}^*(F)_{\semis}$ such that $\gamma^{\mb{H}}$ transfers to $\gamma^*$, we have 
    \begin{equation*}
        \Delta[\mf{w}](\gamma^{\mb{H}}, \gamma^*) = 1.
    \end{equation*}
\end{lemma}
\begin{proof}
    We prove that for each term in the transfer factor, the product over all places equals $1$.

    The product of the local $\epsilon_L(V, \chi_v)$ terms equals a global root number of a virtual $\Gal(\ov{F}/F)$ representation coming from base change from $\Z$ (and hence also $\R$) and therefore equals $1$.

    The global $\Delta_I$ term vanishes as in the proof of \cite[Theorem 6.4A]{LS2} since this term doesn't depend on $(\gamma^{\mb{H}}, \gamma^*)$ beyond the fixed isomorphism $\varphi_{\gamma^*, \gamma^{\mb{H}}}$ which is also used for each pair $(\gamma^{\mb{H}}_n, \gamma^*_n)$.

    The global $\Delta_{II}$ also vanishes by the same argument as \cite[Theorem 6.4A]{LS2} once we fix a global $\chi$-datum and use the local $\chi$-data coming via pullback from the global one. In particular, the point is that the global $\chi$-datum gives functions from $\chi_{\alpha}: \A^{\times}_{F_{\alpha}} / F^{\times}_{\alpha} \to \C^{\times}$, and these will vanish on $\frac{\alpha(\gamma^*) - 1}{a_{\alpha}} \in F^{\times}_{\alpha}$.

    A similar principle proves the vanishing of the global $\Delta_{III_{2,D}}$ term. Namely, by fixing global $\chi$-data, we get that our local characters $a_v: F_v \to \C^{\times}$ come from a global character $a: S(\A_F)/S(F) \to \C^{\times}$ which therefore vanishes on $\gamma^*$.

    Finally, the vanishing of $\Delta_{IV}$ follows from the fact that the adelic absolute value on $\A^{\times}_F$ vanishes on $F^{\times}$.
\end{proof}

Now, in analogy with the definition given in equation \cite[ (4.2)]{KalTai}, we define a candidate transfer factor as a function
\begin{equation*}
    \Delta[\mf{w}_v, z^{\iso}(v)] :\mb{H}_v(F_v)_{(\mb{G}, \mb{H})-\reg} \times \mb{G}_v(F_v)_{\semis} \to \C
\end{equation*}
by
\begin{equation}
   \Delta[\mf{w}_v, z^{\iso}(v)](\gamma^{\mb{H}}, \gamma) := \Delta[\mf{w}_v](\gamma^{\mb{H}}, \gamma^*) \cdot \langle \inv[z^{\iso}](\gamma^*, \gamma)(v), \widehat{\varphi}^{-1}_{\gamma^*, \gamma^{\mb{H}}}(s) \rangle^{-1}, 
\end{equation}
where $\gamma^{\mb{H}} \in \mb{H}(F_v)$ is $(\mb{G},\mb{H})$-regular semisimple, $\gamma \in \mb{G}(F_v)$ and $\gamma^* \in \mb{G}^*(F_v)$ are semisimple, and $\Psi_v(\gamma^*)$ and $\gamma$ are stably conjugate.

Note that we have a natural pairing 
\begin{equation}
    \mb{B}(F_v, I^{\mb{G}^*}_{\gamma^*})
\times Z(\widehat{I^{\mb{G}^*}_{\gamma^*}})^{\Gal(\ov{F}_v/F_v)} \to \C^{\times},
\end{equation}
induced by the Kottwitz map
\begin{equation}
\kappa_{I^{\mb{G}^*}_{\gamma^*}} : \mb{B}(F_v, I^{\mb{G}^*}_{\gamma^*}) \to X^*(Z(\widehat{I^{\mb{G}^*}_{\gamma^*}})^{\Gal(\ov{F}_v/F_v)}).
\end{equation}
Then the term
\begin{equation}
    \langle \inv[z^{\iso}](\gamma^*, \gamma)(v), \widehat{\varphi}^{-1}_{\gamma^*, \gamma^H}(s) \rangle^{-1},
\end{equation}
is the above pairing with $\varphi_{\gamma^*, \gamma^{\mb{H}}}$ some admissible homomorphism transferring $\gamma^{\mb{H}}$ to $\gamma^*$.

We now fix place $v$ of $F$ and define $H := \mb{H}_{F_v}$ and similarly for $G$ and $G^*$. Our main result is as follows.
\begin{theorem}{\label{localtrans}}
The term $\Delta[\mf{w}_v, z^{\iso}(v)]$ is a local transfer factor as in \cite[\S2.4]{LS1} and is the continuous extension of the $\mb{B}(F_v, G)$-normalized local transfer factors from $H(F_v)_{G-\sr} \times G(F_v)_{\sr}$ to $H(F_v)_{(G, H)-\reg} \times G(F_v)_{\semis}$.
\end{theorem}
Before we give the proof of Theorem \ref{localtrans}, we need to establish some lemmas.
\begin{lemma}{\label{extendedbij}}
Let $\pi: \mb{B}(F_v, I^{G^*}_{\gamma^*}) \to \mb{B}(F_v, G^*)$ be the projection.  There is a natural bijection 
\begin{equation}{\label{conjugacybg}}
\begin{tikzcd}
(\Psi, z, [\gamma])/ \sim \arrow[r, leftrightarrow]& \{b \in \mb{B}_{\bas}(F_v, I^{G^*}_{\gamma^*}) : \pi(b) \in \mb{B}_{\bas}(F_v, G^*)\} 
\end{tikzcd}
\end{equation}
where $(\Psi, z)$ is an extended pure inner twist of $G^*$ and $[\gamma] \subset G(F_v)$ is a conjugacy class such that for one (and hence any) $\gamma \in [\gamma]$, there exists $g \in G^*(\ov{F_v})$ such that $\Psi(g\gamma^*g^{-1})=\gamma$. We say that $(\Psi_1, z_1, [\gamma]_1) \sim (\Psi_2, z_2, [\gamma]_2)$ if there exists an isomorphism of extended pure inner twists $(f, \delta)$ such that $f([\gamma_1])=[\gamma_2]$. 
\end{lemma}
\begin{proof}
The map from left to right in Equation \eqref{conjugacybg} is defined as follows. We pick a $\gamma \in [\gamma]$ and $g \in G^*(\ov{F_v})$ as above and define an element of $\mb{B}(F_v, I^{G^*}_{\gamma^*})$ via the cocycle $e \mapsto g^{-1} z_e \sigma_e(g)$. It is clear that the projection to $\mb{B}(F_v, G^*)$ recovers the cohomology class of $z$ which is basic. We show the map is well-defined. Indeed if we choose a different $\gamma_1 \in [\gamma], g_1 \in G^*(\ov{F_v})$ such that $\Psi(g_1\gamma^*g^{-1}_1)=\gamma_1$, then we can pick $h \in G(F_v)$ such that $\Int(h)(\gamma)=\gamma_1$. Then we get $g^{-1}_1 \Psi^{-1}(h)g \in I^{G^*}_{\gamma^*}(\ov{F_v})$ and hence $e \mapsto g^{-1}z_e\sigma_e(g)$ is in the same cohomology class as
\begin{align*}
    e &\mapsto [g^{-1}_1\Psi^{-1}(h)g]g^{-1}z_e\sigma_e(g)[\sigma_e(g^{-1}\Psi^{-1}(h^{-1})g_1)]\\
    &=g^{-1}_1\Psi^{-1}(h)z_e\sigma_e(\Psi^{-1}(h^{-1}))\sigma_e(g_1)\\
    &=g^{-1}_1 \Psi^{-1}(h)z_e[z_e^{-1}\Psi^{-1}(h^{-1})z_e]\sigma_e(g_1)\\
    &= g^{-1}_1z_e\sigma_e(g_1).
\end{align*}
If $(\Psi_1, z_1, [\gamma]) \sim (\Psi_2, z_2, [\gamma])$ and $(f, \delta)$ is an isomorphism then we can pick $\gamma_i \in G_i(F_v)$ so that $f(\gamma_1)=\gamma_2$ and $g_i \in G^*_i(\ov{F_v})$ so that $g_2=\delta g_1$. Then we have
\begin{equation*}
    g_2^{-1} z_{2,e}\sigma_{e}(g_2) = g^{-1}_1 z_{1,e} \sigma_e(g_1)
\end{equation*}
as desired. Hence, we indeed get a well defined map 
\begin{equation*}
    (\Psi, z, [\gamma])/ \sim \to \mb{B}_{\bas}(F_v, I^{G^*}_{\gamma^*}).
\end{equation*}

Conversely, let $b \in \mb{B}_{\bas}(F_v, I^{G^*}_{\gamma^*})$ such that $\pi(b) \in \mb{B}(F_v, G^*)$ is basic. Then let $z^b:\mc{E}^{\iso} \to I^{G^*}_{\gamma^*}$ be a cocycle representing $b$ and note that the  composition with the inclusion $i: I^{G^*}_{\gamma^*} \hookrightarrow G^*$ gives a cocycle $z^{\iso}$ representing $\pi(b)$. Let $\Psi: G^* \to G$ be an extend pure inner twist associated to $z^{\iso}$. Now we consider $\Psi(\gamma^*)$. We will be done if we can show that this is an element of $G(F_v)$. Pick $\sigma \in \Gamma_{F_v}$ and $e \in \mc{E}^{\iso}$ so that $\sigma_e=\sigma$. Then since we have $\Psi^{-1}  \circ \sigma(\Psi) = \Int(z^{\iso}_e)$, we get
\begin{align*}
    \sigma(\Psi(\gamma^*)) & =\sigma(\Psi)(\sigma(\gamma^*))\\
    &=\Psi(z^{\iso}_e \sigma(\gamma^*) {z^{\iso}_e}^{-1})\\
    &=\Psi(\sigma(\gamma^*))\\ 
    &=\Psi(\gamma^*),
\end{align*}
as desired. We leave it to the reader to check that the two maps we have constructed are inverses of each other.
\end{proof}

The next lemma is related to the analogous fact for Galois cohomology (\cite[Lemma 10.2]{Kot6}). We recall that for $G$ a connected reductive group over $F_v$ (where $v$ is possibly infinite), a \emph{fundamental torus} is defined to be a maximal $F_v$-torus of minimal split rank. In the case where $v$ is $p$-adic, such a $T$ will be elliptic.

\begin{lemma}{\label{bgsurj}}
Suppose that $G$ is a connected reductive group over $F_v$. Let $T \subset G$ be a fundamental torus. Then the image of $\mb{B}(F_v, T)$ in  $\mb{B}(F_v, G)$ contains $\mb{B}_{\bas}(F_v, G)$.
\end{lemma}
\begin{proof}
When $v$ is infinite, this follows from \cite[Lemma 13.2]{Kot9}.

When $v$ is finite, $T$ is elliptic maximal and we can essentially use \cite[Remark 13.3]{Kot9}. In particular, for each $b \in \mb{B}(F_v, T)$, the Newton point of $b$ must factor through $Z(G)$ since $T$ is elliptic. Hence $b$ has basic image in $\mb{B}(F_v, G)$. Then the desired result follows from the fact that in the $p$-adic case, the Kottwitz map is an isomorphism on basic elements and the map $X^*(\widehat{T})_{\Gamma_{F_v}} \to X^*(Z(\widehat{G}))_{\Gamma_{F_v}}$ is surjective (since coinvariants are right exact).
\end{proof}
\begin{lemma}
Suppose $v$ is a place of $F$, that $G^*$ is a quasisplit group over $F_v$, and that $\Psi: G^* \to G$ is an extended pure inner twist represented by a cocycle $z^{\iso}$. Suppose that $\gamma \in G(F_v)$ and $\gamma^* \in G^*(F_v)$ are semisimple and there exists $g \in G^*(\ov{F_v})$ such that $\Psi(g\gamma^*g^{-1})=\gamma$. Then there exist maximal $F_v$ tori $T^* \ni \gamma^*$ and $T \ni \gamma$ and an element $g_1 \in G^*(\ov{F_v})$ such that $\Psi(g_1\gamma^* g^{-1}_1)=\gamma$ and $\Psi(g_1T^*g^{-1}_1)=T$ and $\Psi \circ \Int(g)$ gives an isomorphism of $F_v$-tori.
\end{lemma}
\begin{proof}
By Lemma \ref{extendedbij} we can associate to the triple $(\Psi, z^{\iso}, [\gamma])$, an element $b \in \mb{B}_{\bas}(F_v, I^{G^*}_{\gamma^*})$. We now pick a fundamental torus $T^*$ of $I^{G^*}_{\gamma^*}$. Then by Lemma \ref{bgsurj} we can pick a lift $b' \in \mb{B}(F_v, T^*)$ that maps to $b$ under the canonical map $\mb{B}(F_v, T^*) \to \mb{B}(F_v, I^{G*}_{\gamma^*})$. We now fix a strongly regular $\gamma^*_{T^*} \in T^*(F)$ and by applying the construction in Lemma \ref{extendedbij} to $b'$ and the projection $\pi: T^* \to G^*$, we get a triple $(\Psi, z^{\iso}, [\gamma_T])$ corresponding to $b'$. Note, that by construction, we can indeed pick $(\Psi, z^{\iso})$ in this triple to be the $(\Psi, z^{\iso})$ we started with. We let $T$ be the centralizer of $\gamma_T$ in $G$. By the proof of loc. cit. we get a cocycle $z^{b'}$ representing $b'$ and a $g_1 \in G^*(\ov{F_v})$ so that $z^{b'}$ is given by $e \mapsto g^{-1}_1 z^{\iso}_e \sigma_e(g_1)$  and $\Psi(g_1 \gamma^*_{T^*} g^{-1}_1)=\gamma_{T}$. In particular, $\Psi(g_1 T^*g^{-1}_1)=T$ and induces an isomorphism of $F_v$ tori. Let $\gamma'=\Psi(g_1 \gamma^*g^{-1}_1) \in T$. Then $(\Psi, z^{\iso}, [\gamma'])$ is another representative of the equivalence class of triples corresponding to $b$ under Lemma \ref{extendedbij}. Since any automorphism of $(\Psi, z^{\iso})$ will be given by conjugation by an element of $G(F_v)$, it follows that $\gamma'$ and $\gamma$ are conjugate in $G(F_v)$. Thus by taking a $G(F_v)$-conjugate of $T$ and similarly modifying $g_1$ we can arrange that $\Psi(g_1\gamma^*g^{-1}_1)=\gamma$.
\end{proof}
\begin{lemma}{\label{sequencelem}}
For a fixed choice of $\gamma^* \in G^*(F_v)$ and $\gamma \in G(F_v)$ such that there exists $g \in G^*(\ov{F_v})$ with  $\Psi(g\gamma^*g^{-1})=\gamma$, we can find sequences of elements $(\gamma^*_i)$ and $(\gamma_i)$ such that each $\gamma_i \in G(F)_{\sr}$ and $\gamma^*_i \in G^*_{\sr}$ satisfying:
\begin{enumerate}
   \item $\lim\limits_{i \to \infty} \gamma_i = \gamma$,
   \item $\lim\limits_{i \to \infty} \gamma^*_i = \gamma^*$,
    \item For each $i$, we have that $\gamma^*_i$ is a $\Psi$-transfer of $\gamma_i$,
    \item There are fixed maximal tori $T^* \subset G^*$ and $T \subset G$ such that the centralizer of each $\gamma^*_i$ is $T^*$ and the centralizer of each $\gamma_i$ is $T$,
  \item For each $i$, we have that the cocycles defining $\inv[z^{\iso}](\gamma^*_i, \gamma_i)(v)$ and $\inv[ z^{\iso}](\gamma^*, \gamma)(v)$ are equal in $Z^1_{\bas}(\mc{E}_{\iso}, I^{G^*}_{\gamma_*}(\ov{F_v}))$.
\end{enumerate}
\end{lemma}
\begin{proof}
By the previous lemma, we find $g \in G^*(\ov{F_v})$ and maximal tori $T^* \subset G$ and $T \subset G$ such that $\Psi(gT^*g^{-1})=T$ and $\Psi(g \gamma^* g^{-1})=\gamma$. Pick a sequence of strongly regular elements $\gamma^*_i$ in $T^*$ and converging to $\gamma^*$. Then define $\gamma_i = \Psi(g \gamma^*_ig^{-1})$. The last point then follows from the proof of Lemma \ref{bgloclem} and the fact that we have $\Psi(g\gamma^*_ig^{-1})=\gamma_i$ and $\Psi(g \gamma^*g^{-1})=\gamma$ for the same $g$.
\end{proof}
We now complete the proof of Theorem \ref{localtrans}.

\begin{proof}
To prove Theorem \ref{localtrans}, it suffices to show that $\Delta[\frak{w}_v, z^{\iso}(v)]$ equals the unique continuous extension of the $\mb{B}(F_v,G)$-normalized local transfer factor (as in \cite[Prop. 4.3.1]{KalTai}) from the strongly regular to $(G,H)$-regular locus. Our transfer factor is identical to that of Kaletha and Taibi on the strongly regular locus, so it suffices to show that we can find sequences $(\gamma^H_i), (\gamma_i)$ of strongly regular elements converging to $\gamma^H, \gamma$ such that our factor satisfies $\lim\limits_{i \to \infty} \Delta[\mf{w}_v, z^{\iso}(v)](\gamma^H_i, \gamma_i) = \Delta[\mf{w}_v, z^{\iso}(v)](\gamma^H, \gamma)$.  The Whittaker-normalized transfer factor, $\Delta[\mf{w}_v]$, for $\mb{G}^*$ is already known to have this property. Hence if we have found sequences as above, we need only choose a sequence $(\gamma^*_i)$ for $\gamma^*_i \in G^*(F_v)$ and $\Psi_v(\gamma^*_i)$ stably conjugate to $\gamma_i$ such that
\begin{equation*}
    \lim\limits_{i \to \infty} \langle \inv[z^{\iso}](\gamma^*_i, \gamma_i)(v), \widehat{\varphi}^{-1}_{\gamma^*_i, \gamma^H_i}(s) \rangle = \langle \inv[z^{\iso}](\gamma^*, \gamma)(v), \widehat{\varphi}^{-1}_{\gamma^*, \gamma^H}(s) \rangle.
\end{equation*}
This is the content of Lemma \ref{sequencelem}.
\end{proof}

We can now define the global transfer factor $\Delta[\mathfrak{w}, z^{\iso}]$ to be the product of the local factors. Note that in order for such a product to make sense, we must have that $\Delta[\mf{w}_v, z^{\iso}(v)](\gamma^{\mb{H}}, \gamma)$ is trivial at all but finitely many places $v$. This follows from the analogous statement for $\Delta[\mf{w}_v]$ as well as the fact (derived in \S\ref{sec:refinedinv} from our work in \S\ref{totalloc}) that $\inv[z^{\iso}](\gamma^*, \gamma)(v)$ is trivial for all but finitely many $v$. Because of our normalization choices, we now expect $\Delta[\mf{w}, z^{\iso}]$ to satisfy
\begin{equation}{\label{transfactcanoneqnnorm}}
    \Delta[\mf{w}, z^{\iso}](\gamma_{\mb{H}}, \gamma) = \langle \obs(\gamma)^{-1}, \widehat{\varphi}^{-1}_{\gamma^*, \gamma^{\mb{H}}}(s) \rangle.
\end{equation}
This turns out to be the case.
\begin{corollary}{\label{cor: globaltranscor}}
The term $\Delta[\mathfrak{w}, z^{\iso}]$ is equal to the unique continuous extension to $\mb{H}(\A_F)_{(\mb{G},\mb{H})-\reg} \times \mb{G}(\A_F)_{\semis}$ of the canonical global absolute transfer factor of \cite[\S7.3]{KS} and satisfies Equation \eqref{transfactcanoneqnnorm} when restricted to $\mb{H}(F)_{(\mb{G},\mb{H})-\reg} \times \mb{G}(\A_F)_{\semis}$.
\end{corollary}
\begin{proof}
To prove the first claim, we note that $\Delta[\mf{w}, z^{\iso}]$ is defined to be a product of local factors that are continuous extensions of the $\mb{B}(F_v, G)$-normalized transfer factors of \cite[Prop. 4.3.1]{KalTai}. The product of these local factors equals the canonical global factor on the $\mb{H}(\A_F)_{\mb{G}-\reg} \times \mb{G}(\A_F)_{\sr}$ locus since it satisfies Equation \eqref{transfactcanoneqnnorm} (\cite[Prop. 4.3.2]{KalTai}). This implies the first claim.

To prove the second claim, we recall that by Lemma \ref{lem: transfactvanish}, we have
\begin{equation*}
    \prod\limits_v \Delta[\mf{w}_v](\gamma^{\mb{H}} , \gamma^*) =1,
\end{equation*}
for $(\gamma^{\mb{H}} , \gamma^*) \in \mb{H}(F)_{(\mb{G}^*,\mb{H})-\reg} \times \mb{G}^*(F)_{\semis}$  Hence, it suffices to show that 
\begin{equation*}
    \langle \inv[z^{\iso}](\gamma^*, \gamma) , \hat{\varphi}^{-1}_{\gamma^*, \gamma^{\mb{H}}}(s)\rangle = \langle \obs(\gamma), \hat{\varphi}^{-1}_{\gamma^*, \gamma^{\mb{H}}}(s)\rangle,
\end{equation*}
for $(\gamma^{\mb{H}} ,\gamma) \in \mb{H}(F)_{(\mb{G},\mb{H})-\reg} \times \mb{G}(\A_F)_{\semis}$ and $\gamma^* \in \mb{G}^*(F)$. But this follows from Proposition \ref{obseqz} and Diagram \eqref{keycommdiagram}.
\end{proof}

\printbibliography
\end{document}